%% file: main.tex
\documentclass[10pt, a4paper]{article}
\usepackage[margin=25mm]{geometry}
\pdfoutput=1

\input{Setup/preamble}
\input{Setup/commands}

\providecommand{\keywords}[1]
{
  \small	
  \textbf{Keywords:} #1
}

\title{\red{Robust nonconforming virtual element methods for general fourth order problems with varying coefficients}}
\author{Andreas Dedner\thanks{Corresponding author. Department of Mathematics, University of Warwick, Coventry, CV4 7AL, UK. Email: a.s.dedner@warwick.ac.uk} 
\ and 
Alice Hodson  
}

\date{}

\begin{document}
\maketitle

\setboolean{thesis}{false}
\setboolean{oldResultsDisplay}{false}

\input{Setup/abstract}
\input{Sections/intro}

\input{Sections/ctsproblem}
\input{Sections/virtualElementDisc}
\input{Sections/virtualElementSpaces}
\input{Sections/errorAnalysis}
\input{Sections/perturbationModifiedProblem}
\input{Sections/results}
\input{Sections/conclusion}
\input{Sections/acknowledgements}

\bibliographystyle{acm}
\bibliography{main}

\end{document}

%% file: Setup/preamble.tex
\usepackage{hyperref}
\usepackage{amsfonts}
\usepackage{float}
\usepackage{placeins}
\usepackage{booktabs}
\usepackage[font=small]{caption}
\usepackage{amsthm}
\usepackage{amsmath}
\usepackage{amssymb}
\usepackage{mathtools}
\usepackage{mathrsfs}
\usepackage{graphicx}
\usepackage{enumitem}
\usepackage{tikz}
\graphicspath{ {images/} }
\usepackage[disable]{todonotes} 
\usepackage{comment}
\usepackage{ifthen}
\usepackage{epstopdf}
\usepackage{soul}
\usepackage{multirow, subcaption}

\allowdisplaybreaks

\newboolean{thesis} 
\newboolean{oldResultsDisplay}

\numberwithin{equation}{section} 

\newtheorem{theorem}{Theorem}[section]
\newtheorem{lemma}[theorem]{Lemma}
\newtheorem{corollary}[theorem]{Corollary}

\theoremstyle{definition}
\newtheorem{remark}[theorem]{Remark}
\newtheorem{definition}[theorem]{Definition}
\newtheorem{assumption}[theorem]{Assumption}
\newtheorem{example}[theorem]{Example}

%% file: Setup/commands.tex
\newcommand{\ph}{\varphi}

\newcommand{\R}{\mathbb{R}}

\newcommand{\N}{\mathbb{N}}
\newcommand{\Z}{\mathbb{Z}}

\newcommand{\prob}{\mathbb{P}}

\newcommand{\cP}{\mathcal{P}}
\newcommand{\cT}{\mathcal{T}}

\newcommand{\polOrder}{l}
\newcommand{\localVEMSpace}{V_{h,\polOrder}^K}
\newcommand{\vemSpace}{V_{h,\polOrder}}
\newcommand{\enlargedVemSpace}{\widetilde{V}_{h,\polOrder}}
\newcommand{\HTwoNCSpace}{H^{2,nc}_{\polOrder}(\cT_h)}

\newcommand{\valueProj}{\Pi^K_0}
\newcommand{\gradProj}{\Pi^K_1}
\newcommand{\hessProj}{\Pi^K_2}
\newcommand{\edgeProj}{\Pi^e_0}
\newcommand{\edgeNormalProj}{\Pi^{e}_1}
\newcommand{\vertexValueDofs}{d_0^v}
\newcommand{\vertexDerivDofs}{d_1^v}
\newcommand{\edgeValueDofs}{d_0^e}
\newcommand{\edgeNormalDofs}{d_1^e}
\newcommand{\innerDofs}{d_0^i}

\newcommand{\Confinterpolation}{\Pi^{\polOrder-1}}
\newcommand{\GeneralInterpolation}{\widetilde{\Pi}}
\newcommand{\Conenonconf}{C^1 \text{-nc}}
\newcommand{\Conemod}{C^1 \text{-mod}}
\newcommand{\CzeroconfSpace}{C^1 \text{-} C^0}
\newcommand{\originalDofs}{\Lambda^K_M}
\newcommand{\vertiii}[1]{{\left\vert\kern-0.25ex\left\vert\kern-0.25ex\left\vert #1 
    \right\vert\kern-0.25ex\right\vert\kern-0.25ex\right\vert}}

\newcommand{\red}[1]{\textcolor{black}{#1}}
\newcommand{\strike}[1]{\ignorespaces}

%% file: Setup/abstract.tex
\begin{abstract}
    We present a class of nonconforming virtual element methods for general fourth order partial differential equations in two dimensions. We develop a generic approach for constructing the necessary projection operators and virtual element spaces. Optimal error estimates in the energy norm are provided for general linear fourth order problems with varying coefficients. We also discuss fourth order perturbation problems and present a novel nonconforming scheme which is uniformly convergent with respect to the perturbation parameter without requiring an enlargement of the space. Numerical tests are carried out to verify the theoretical results. We conclude with a brief discussion on how our approach can easily be applied to nonlinear fourth order problems.
\end{abstract} 
\hspace{10pt}

\keywords{virtual element method; fourth order problems; nonconforming; perturbation problem; DUNE.}

%% file: Sections/intro.tex
\section{Introduction}
In recent years the discretization of partial differential equations via the virtual element method (VEM) has seen a rapid increase.
Introduced in \cite{beirao_da_veiga_basic_2013} VEM began as an extension and generalization of both finite element and mimetic finite difference methods as discussed in \cite{da2014mimetic}.
In \cite{beirao_da_veiga_basic_2013} the appropriate local and global VEM spaces are constructed and the approximation properties analysed for the Laplace equation.
Another discretization of the Laplace problem was suggested in \cite{ahmad_equivalent_2013} while a nonconforming approach can be found in \cite{de_dios_nonconforming_2014}. An extension to general, nonlinear second order elliptic PDEs for both conforming and nonconforming spaces is discussed in \cite{cangiani_conforming_2015}.  
Similarly, another approach is taken in \cite{beirao2016virtual} for diffusion-convection-reaction problems. 

The versatility of VEM has been showcased through the wide variety of problems it has been applied to over recent years. This has led to the construction of $H$(div) and $H$(curl)-conforming virtual element spaces in \cite{hdiv}, conforming virtual elements for polyharmonic problems in \cite{antonietti_conforming_2018}, and the construction of methods for Stokes flow in \cite{da_veiga_divergence_2015, cangiani_non-conforming_2016, dassi_bricks_2018}, to name but a few.
Especially the ease \strike{in} \red{with} which VEM spaces can be constructed to enforce desirable properties of the discrete functions even on general polygonal meshes makes the approach very interesting for a wide range of problems. 
An example of this is the construction of divergence free vector spaces in \cite{da_veiga_divergence_2015}. A further example is the construction of discrete spaces with higher order continuity conditions. The construction of even a lowest order $C^1$ conforming space is not straightforward within the standard finite element setting and higher order nonconforming spaces suitable for fourth order problems are also not readily available. Consequently, many software packages provide a large number of spaces for second order problems but often only provide the lowest order Morley element \cite{morley1967triangular} for discretizing fourth order problems without requiring the use of splitting methods.

To construct conforming elements for fourth order problems, $C^1$ continuity is required which makes the methods highly complex. 
It is known that using traditional finite element methods, polynomials of at least degree five are needed to construct $C^1$ approximations which are piecewise polynomials. 
In contrast, it is shown in \cite{brezzi_virtual_2013} that the virtual element construction of $C^1$ approximations to fourth order plate bending problems is much simpler and arguably more elegant. 
Additionally, the conforming virtual element method for polyharmonic problems, $\Delta^p u = f $ for $p \geq 1$,  has been addressed in \cite{antonietti_conforming_2018} where the global VEM space consists of $C^{p-1}$ functions.
As well as this, the study of linear elliptic fourth order problems in three dimensions is considered in the conforming case in \cite{da_veiga_c1_2019}. Another example of $C^1$ conforming elements can be seen in the application of the lowest order VEM space to the Cahn-Hilliard equation, investigated in \cite{antonietti_c1_2016}. Further studies of the application of virtual elements to the Cahn-Hilliard equation can be found in \cite{liu2019virtual,liu2020fully}. 

In this work we focus on studying nonconforming virtual element methods but for a wide range of problems. \strike{including nonlinear models.} 
Although we focus on nonconforming VEM for fourth order problems, we highlight that due to the general framework we present, only minor modifications are needed to also include the study of $C^1$ conforming elements for these problems.
Existing works which study nonconforming fourth order problems include the nonconforming approximation of the biharmonic plate bending problem, which is considered in \cite{antonietti_fully_2018, zhao_nonconforming_2016, zhao_morley-type_2018}. 
A mixture of spaces have been suggested, some fully nonconforming \cite{antonietti_fully_2018,zhao_morley-type_2018} and others which include some level of continuity \cite{zhao_nonconforming_2016} though not the full $C^1$ continuity you would see in a fully conforming space. 
More recently, we see a $C^0$ conforming approach to fourth order perturbation problems being considered in \cite{zhang_nonconforming_2020}.
\strike{To our knowledge, the application of higher order VEM also to more general nonconstant coefficients and nonlinear fourth order problems is not available at the time of writing.}
\red{There are other general approaches that can be employed such as hybrid high-order (HHO) methods \cite{di2015hybrid} which, like VEM, can easily handle polyhedral meshes. The connection between VEM and HHO has been analysed and discussed for the Poisson problem in \cite{lemaire2019bridging}. Hybrid high-order methods have also been extended to fourth order problems in \cite{bonaldi2017hybrid} where a novel HHO method for the Kirchhoff-Love plate bending problem is presented. To our knowledge, the application of both higher order VEM and HHO methods to more general nonconstant coefficients and nonlinear fourth order problems is not available at the time of writing.  }

Arguably the most important ingredient of VEM is the construction of projection operators. In the available literature on fourth order problems, projection operators are constructed based on the underlying variational problem.
The main idea of this approach is to construct only one projection which depends on the local contribution to the bilinear form. 
In \cite{cangiani_conforming_2015,ahmad_equivalent_2013} a different approach was taken for discretizing second order problems, which makes it straightforward to apply the method to nonlinear models. In this paper we generalize this approach and demonstrate how it can be applied to a wide range of fourth order problems. A major advantage of this approach is that it can be included more easily into existing software frameworks. A central building block for implementing Galerkin type schemes is the evaluation of nodal basis functions and their derivatives at given quadrature points. To extend this to our VEM setting, these methods have to be replaced with the evaluation of projection operators defined on each element. We implemented this approach within the DUNE \cite{dunegridpaperII,dedner2010generic} software framework, requiring little change to the existing code base. From the user perspective switching between a finite element to a virtual element discretization is seamless
especially within the available Python frontend \cite{dedner_dune_2018,dednerPythonBindings2020}.

    \red{ We refer throughout this paper to two well known VEM spaces, the nonconforming space discussed in both \cite{antonietti_fully_2018,zhao_morley-type_2018} and the $C^0$ conforming space discussed in \cite{zhao_nonconforming_2016,zhang_nonconforming_2020} by demonstrating how they fit into our generalized framework.}
    \red{The main contributions of this paper are detailed as follows.}

    \red{ (i) We present a general approach for constructing nonconforming VEMs for any order of accuracy, for solving a general fourth order PDE problem with nonconstant coefficients. We extend the work in \cite{cangiani_conforming_2015} on second order elliptic problems and using the same techniques, we prove optimal error estimates in the energy norm (Section \ref{section: error analysis}). }
    
    \red{ (ii) In Section~\ref{section: VEM spaces}, we introduce a general approach for constructing VEM projection operators based on constraint linear least square formulations that ensures they are fully computable from the available degrees of freedom. }
    
    \red{ (iii) The fourth order perturbation problem is studied in Section~\ref{section: perturbation problem}. 
    We discuss how two standard spaces mentioned previously fit into our framework although note that our approach for constructing the projection operators differs from the approach taken in \cite{antonietti_fully_2018,zhao_morley-type_2018,zhang_nonconforming_2020}.
    We furthermore present a novel fully nonconforming scheme which unlike modifications seen in \cite{zhang_nonconforming_2020}, does not require an enlargement of the standard nonconforming VEM space.}

    \red{ (iv) Lastly, we carry out numerical experiments in Section~\ref{section: numerical testing} to confirm the a priori error analysis. We conclude by showing how our construction of projection operators allows us to straightforwardly solve complex nonlinear fourth order problems. }

%% file: Sections/ctsproblem.tex
\section{The continuous problem}\label{section: cts problem}
Throughout this paper, we adopt the standard notation for Sobolev spaces $H^s(\mathcal{D})$ for nonnegative integers $s$, and for a domain $\mathcal{D}$. We denote the norm and seminorm by $\| \cdot \|_{s,\mathcal{D}}$ and $|\cdot |_{s,\mathcal{D}}$ respectively. If $\mathcal{D} = \Omega$ then the subscript shall be omitted. The notation $(\cdot,\cdot)_{\mathcal{D}}$ will be used to denote the $L^2 (\mathcal{D})$ inner product. 
For a nonnegative integer $\polOrder$, let $\prob_{\polOrder}(\mathcal{D})$ denote the set of all polynomials up to degree $\polOrder$ over $\mathcal{D}$. We use the convention that $\prob_{-1}(\mathcal{D})=\{ 0\}.$
We denote the standard $L^2 (\mathcal{D})$ orthogonal projection onto the polynomial space $\prob_{\polOrder}(\mathcal{D})$ by $\cP^{\polOrder}_{\mathcal{D}}$.  
The tensor of all derivatives of a given order $|\mu|$ is denoted with $D^{|\mu|}\ph$. 
Let $\partial_n \ph = \nabla \ph \cdot n$ denote the normal derivative of a function $\ph$ over $\partial \mathcal{D}$ and let $\partial_s \ph = \nabla \ph \cdot \tau $ denote its tangential derivative where we use $\tau$ to denote a tangential vector.

Consider a general linear fourth order problem defined on a polygonal domain $\Omega \subset \R^2$
described by a bilinear form
\begin{align}
    a(u,v) := \int_{\Omega} \kappa(x) D^2 u : D^2 v \ \mathrm{d} x + \int_{\Omega} \beta(x) Du \cdot Dv \, \mathrm{d}x + \int_{\Omega} \gamma(x) u v \, \mathrm{d}x 
    \label{eqn: bilinear form cts}
\end{align}
for $u,v \in H^2_0(\Omega)$ with
    $
    H^2_0(\Omega) = \{ v \in H^2(\Omega) : v = \partial_n v = 0  \text{ on } \partial \Omega \}.
    $

We make the minimal assumptions that $\kappa, \beta, \gamma \in L^{\infty}(\Omega)$. In later sections we impose further conditions on the coefficients. For now we assume
that the coefficients satisfy $\kappa \geq \kappa_0 > 0$, for a constant $\kappa_0$ and $\beta, \gamma \geq 0$.
\red{We define ${\beta_0 = \min{\beta(x)}}$, and $\gamma_0 = \min{\gamma(x)}$.}
Note that we could also consider an even more general setting, e.g., take $\beta \in L^{\infty}(\Omega)^{2 \times2}$ as in \cite{cangiani_conforming_2015}. The results in this paper can be easily extended to cover this case but to keep the presentation simple we only consider scalar coefficients.

The variational problem for a given $f \in L^{2}(\Omega)$
reads as follows: find $u \in H^2_0(\Omega)$ such that
\begin{align}
    a(u,v) = (f,v) \quad \forall \, v \in H^2_0(\Omega).\label{eqn: cts problem}
\end{align}

Since the bilinear form is symmetric, we can define an energy norm $\vertiii{\cdot} $ by 
$\vertiii{v}^2 = a(v,v).$ 
It follows easily that the bilinear form is coercive and continuous with respect to the energy norm,
\begin{align*}
    a(u,v) &\leq \vertiii{u} \vertiii{v}, \quad \text{ for all } u,v \in H^2_0(\Omega), \\
    a(v,v) &\geq \vertiii{v}^2, \quad \text{ for all } v \in H^2_0(\Omega). 
\end{align*} 
Hence it follows from the Lax-Milgram Lemma that \eqref{eqn: cts problem} has a unique solution. 

\begin{remark}
  Assuming that the solution $u$ to \eqref{eqn: cts problem} is smooth enough, we can derive the corresponding strong form of the PDE
  \begin{equation}\label{eqn: strong form}
      \begin{split}
          \sum_{i,j=1}^2 \partial_{ij} (\kappa \partial_{ij} u) - \sum_{i=1}^2 \partial_i (\beta \partial_i u) + \gamma u &= f \quad \text{ in } \Omega, \\
          u = \partial_n u &= 0 \quad \text{ on } \partial \Omega.  
      \end{split}
  \end{equation}

Note that if we were considering constant coefficients, taking $\kappa,\beta,\gamma \in \R$, as in \cite{da_veiga_c1_2019}, then the strong form would reduce to the PDE studied there
    \begin{align*}
        \kappa \Delta^2 u - \beta \Delta u + \gamma u &= f \quad \text{ in } \Omega, \\
        u = \partial_n u &= 0 \quad \text{ on } \partial \Omega.
    \end{align*}
\end{remark}

%% file: Sections/virtualElementDisc.tex
\section{The discrete problem and an abstract convergence result}\label{sec: discrete problem}

In this section we provide some general ingredients needed for the discretization of our problem and present a Strang-type abstract error estimate. Let $\cT_h$ denote a tessellation of the computational domain $\Omega$ and denote the set of all edges in $\cT_h$ by $\mathcal{E}_h$. We split this set into boundary edges, $\mathcal{E}_h^{\text{bdry}} := \{ e \in \mathcal{E}_h : e \subset \partial \Omega \}$ and internal edges $\mathcal{E}_h^{\text{int}} := \mathcal{E}_h \backslash \mathcal{E}_h^{bdry}$. 
Similarly, denote the set of vertices in $\cT_h$ by $\mathcal{V}_h = \mathcal{V}_h^{\text{int}} \cup \mathcal{V}_h^{\text{bdy}}$, which again is made up of interior and boundary vertices.
\newcounter{assumptions1}

For an integer $s>0$, define the \emph{broken Sobolev space} $H^s(\mathcal{T}_h)$ by 
\begin{align*}
    H^s(\mathcal{T}_h) := \{  v \in L^2(\Omega) : v|_K \in H^s(K), \ \forall \, K \in \mathcal{T}_h  \},
\end{align*} 
and on this space define the \emph{broken $H^s$ seminorm}
\begin{align*}
 | v_h |^2_{s,h} = \sum_{K \in \cT_h} |v_h|^2_{s,K}.
\end{align*}
For a function $v \in H^2(\cT_h)$ we define the jump operator $[ \cdot ]$ across an edge $e \in \mathcal{E}_h$ as follows. For an internal edge, $e \in \mathcal{E}_h^{\text{int}}$, define $[ v ] := v^+ - v^- $ where $v^{\pm}$ denotes the trace of $v|_{K^{\pm}}$ where $e \subset \partial K^{+} \cap \partial K^{-}$. For boundary edges, $e \in \mathcal{E}_h^{\text{bdry}}$, let $[v] := v|_e$.
We denote with $\prob_{k}(K)$ the space of polynomials over a grid element $K$ and define the piecewise polynomial space $\prob_{k} (\cT_{h})$ for any $k \in \N$ with
    \begin{align*}
        \prob_{k} (\cT_{h}) := \{ p \in L^2(\Omega) :  p|_K \in \prob_{k}(K),  \ \forall K \in \cT_{h} \}.
    \end{align*}
 
We now make the following basic assumptions. In particular, we stress that throughout the paper the polynomial order $\polOrder$ is fixed. 

\begin{assumption}\label{assumptions: discrete assumptions} 
    Assume the following holds for any fixed $h > 0$ and for a fixed $\polOrder \geq 2$.
    \begin{enumerate}[label=(\textit{A}\arabic*)]
        \item The mesh $\cT_h$ consists only of \textit{simple polygons}. \red{A simple polygon refers to the criteria that the boundary of each element must not intersect itself and is made up of a finite number of straight line segments.  } \label{mesh}
        \item The finite dimensional function space $\vemSpace$ satisfies $\prob_{\polOrder} (\cT_{h})\subset\vemSpace$ \strike{for some fixed $\polOrder \geq 2$} and $\vemSpace \subset \HTwoNCSpace$. We define the nonconforming space $\HTwoNCSpace \subset H^{2}(\cT_h)$ as 
        \begin{align*}
            \HTwoNCSpace  := \Big\{ v \in & \ H^2(\cT_h) : v \text{ continuous at internal vertices, } v(v^i) = 0  \quad \forall v^i \in \mathcal{V}_h^{\text{bdry}}, 
                \\
                &\int_e [ \, \partial_n v \, ]p \, \mathrm{d} s = 0\ \forall \, p \in \prob_{\polOrder-2}(e),\ 
                \int_e [\, v \, ]p \, \mathrm{d}s = 0 \ \forall \, p \in \prob_{\polOrder-3}(e),\ \forall \, e \in \mathcal{E}_h  \Big\}.
        \end{align*} 
        \label{assumption: A2}
        \item There exists $f_h \in \vemSpace^{'}$, which approximates the right hand side of our variational problem \eqref{eqn: cts problem}.  
        \item There exists a discrete bilinear form $a_h : \vemSpace \times \vemSpace \rightarrow \R$, such that for any $u_h, v_h \in \vemSpace$,  
        \begin{align*}
            a_h(u_h,v_h) = \sum_{K \in \cT_h} a_h^K(u_h,v_h).
        \end{align*}
        The bilinear form $a_h^K : \vemSpace |_K \times \vemSpace |_K \rightarrow \R$ is the restriction of $a_h$ to an element $K$. We denote the restriction of the VEM space $\vemSpace$ to an element $K$ by 
        $\localVEMSpace := \vemSpace|_K$.
        \item \emph{Stability property:} assume that there exists two constants $\alpha_*, \alpha^*$ such that 
        \begin{align*}
            \alpha_* a^K(v_h,v_h) \leq a_h^K(v_h,v_h) \leq \alpha^* a^K(v_h,v_h)
        \end{align*} 
        for all $v_h \in \vemSpace^K$.
        \label{assumption: stability property}
        \setcounter{assumptions1}{\value{enumi}}
    \end{enumerate}
\end{assumption} 
The criteria in the stability property \ref{assumption: stability property} is required to show that the discrete bilinear form is coercive and continuous. 

\begin{lemma}\label{lemma: discrete norms are norms}
    The broken Sobolev norm $|\cdot |_{2,h}$ is a norm on the spaces $H^2_0(\Omega)$ and $\HTwoNCSpace$. 
    Define the element wise discrete energy norm $\vertiii{w_h}^2_h = \sum_{K \in \cT_h} \vertiii{w_h}_K^2 $ 
    for functions $w_h \in \HTwoNCSpace$,
    where the element wise contributions are given by
    \begin{align*}
        \vertiii{w_h}^2_K = (\kappa D^2 w_h,D^2 w_h)_K + (\beta D w_h, D w_h)_K + (\gamma w_h,w_h)_K.
    \end{align*}
    Then, we have that $\vertiii{\cdot}_h$ is a norm on $\HTwoNCSpace$. Therefore, under Assumption \ref{assumption: A2}, it follows that both $|\cdot |_{2,h}$  and $\vertiii{\cdot}_h$ are a norm on $\vemSpace$.
\end{lemma}
\begin{proof}
    From \cite{antonietti_fully_2018,zhao_nonconforming_2016,brenner2004poincare} it follows that $|\cdot |_{2,h} $ is a norm on both $H^2_0(\Omega)$ and $\HTwoNCSpace$. Consequently, $\vertiii{\cdot}_h$ is a norm on $ \HTwoNCSpace$ under the given conditions on the coefficients $\kappa,\beta,\gamma$ stated in Section \ref{section: cts problem}.
\end{proof}

The following is now a direct consequence of the stability assumption \ref{assumption: stability property}:
\begin{theorem}[Existence and uniqueness of solutions to the discrete problem]
    Under Assumption \ref{assumptions: discrete assumptions} the discrete problem: find $u_h \in \vemSpace$ such that 
    \begin{align}\label{eqn: discrete problem}
        a_h(u_h,v_h) = \langle f_h,v_h \rangle \quad \forall \, v_h \in \vemSpace
    \end{align}
    admits a unique solution. 
\end{theorem}
We now have the following Strang-type error bound, the proof of which is standard and identical to the method from e.g. \cite{cangiani_conforming_2015}. 
\begin{theorem}[A priori error bound]\label{thm: a-priori error bound energy norm}
    Under Assumption \ref{assumptions: discrete assumptions} it holds that
    \begin{equation}
        \begin{split}
            \alpha_* \vertiii{u-u_h}_h \leq& \inf_{v_h \in \vemSpace} \alpha^* \vertiii{u-v_h}_h + \sup_{\substack{w_h \in \vemSpace \\ w_h \neq 0}} \frac{| \langle f_h,w_h \rangle - (f,w_h)|}{ \vertiii{w_h}_h} 
            + \sup_{\substack{w_h \in \vemSpace \\ w_h \neq 0}} \frac{| \mathcal{N}(u,w_h)|} {\vertiii{w_h}_h} \\
            &+ \inf_{p \in \prob_{\polOrder}(\cT_h)} \Big( (\alpha^* +1) \vertiii{u-p}_h + \sum_{K \in \cT_h} \sup_{\substack{w_h \in \vemSpace^K \\ w_h \neq 0}} \frac{|a^K(p,w_h)-a^K_h(p,w_h)|}{ \vertiii{w_h}_K } \Big),
        \end{split}
        \label{eqn: a-priori bound}
    \end{equation}
    where $\alpha_*$ and $\alpha^*$ are from the stability property \ref{assumption: stability property}.
    The nonconformity error is given by
    \begin{align}\label{eqn: nonconformity error}
        \mathcal{N}(u,w_h) = a(u,w_h)  - (f,w_h).
    \end{align}
\end{theorem}

We finish this section by collecting the remaining technicalities needed for the rest of the paper. In particular, we make the following regularity conditions on the mesh $\cT_h$ which are standard in the virtual element framework, see e.g., \cite{beirao_da_veiga_basic_2013}.

\begin{assumption}[Mesh assumptions]\label{assumption: mesh regularity}
    Assume there exists some $\rho >0$ such that the following hold.
    \begin{enumerate}[label=(\emph{A}\arabic*)]
        \setcounter{enumi}{\value{assumptions1}}
            \item For every element $K \in \cT_h$ and every edge $e \subset \partial K$, $h_e \geq \rho h_K$ where $h_e=|e|$ and $h_K$ is the diameter of $K$.
            \label{assump: mesh 1}
            \item Assume that each element is star shaped with respect to a ball of radius $\rho h_K$.
            \label{assump: star shaped wrt a ball}
        \setcounter{assumptions1}{\value{enumi}}
     \end{enumerate}
\end{assumption} 

Finally, we recall some standard results for the $L^2$ projection operator.
\begin{definition}
    For any $K \in \cT_h$ define the $L^2(K)$ \emph{orthogonal projection} onto the polynomial space $\prob_{\polOrder}(K)$, that is
    $\mathcal{P}^{\polOrder}_K : L^2(K) \rightarrow \prob_{\polOrder}(K)$ by, 
\begin{align*}
    ( \mathcal{P}^{\polOrder}_K v, p)_K = (v,p)_K \quad \text{ for all } p \in \prob_{\polOrder}(K),
\end{align*}
and for any edge $e \subset \partial K$ define the $L^2(e)$ orthogonal projection onto $\prob_{\polOrder}(e)$, $\mathcal{P}^{\polOrder}_e : L^2(e) \rightarrow \prob_{\polOrder}(e)$ by,
\begin{align*}
    ( \mathcal{P}^{\polOrder}_e v, p)_e = (v,p)_e \quad \text{ for all } p \in \prob_{\polOrder}(e). 
\end{align*} 
\end{definition}
A proof of the following error estimates can be obtained using for example the theory in either \cite{brenner_mathematical_2008,clarlet1987finite}.
\begin{theorem}\label{thm: interpolation estimates}
    Under Assumption \ref{assumption: mesh regularity}, for $\polOrder\geq 0$ and for any $w \in H^m (K)$ with ${1 \leq m \leq \polOrder +1}$, it follows that
\begin{align*}
    | w - \cP^{\polOrder}_K w |_{s,K} \lesssim h_K^{m-s} |w|_{m,K}
\end{align*}
for $s = 0,1,2$.
Further, for any edge shared by $K^{+}$,$K^{-} \in \cT_h$ 
and for any $w \in H^m(K^{+} \cup K^{-})$, with 
$1 \leq m \leq \polOrder+1$, it follows that
\begin{align*}
    |w-\cP^{\polOrder}_e w|_{s,e} 
    \lesssim 
    h_e^{m-s-\frac{1}{2}} | w |_{m,K^{+}\cup K^{-}}
\end{align*}
for $s=0,1,2$.
\end{theorem}

%% file: Sections/virtualElementSpaces.tex
\section{The virtual element spaces}\label{section: VEM spaces}
We dedicate \red{this} \strike{the next} section to the virtual element discretization. We specify the chosen degrees of freedom by a \emph{dof tuple} which allows us to easily encode a number of different local VEM spaces. A major part of the VEM method is the construction of projection operators. We detail a new construction of projection operators suitable for general VEM discretization of a wide range of fourth order problems with nonconstant coefficients.

Throughout this section we provide examples of nonconforming VEM spaces for fourth order problems ($C^1$ nonconforming spaces). In particular, we use as examples the original nonconforming space detailed in \cite{antonietti_fully_2018,zhao_morley-type_2018} with $ \vemSpace \not\subset H^1_0(\Omega)$. As well as this space, we look at the $C^0$ conforming space detailed in \cite{zhao_nonconforming_2016,zhang_nonconforming_2020}, such that $ \vemSpace \subset H^1_0(\Omega)$. 
We conclude this section with defining the global VEM spaces and the bilinear forms. 

\subsection{Degrees of freedom tuple} 
We begin this section by introducing the concept of a degrees of freedom (dof) tuple, used to generically describe the degrees of freedom relating to a VEM space on each element of the grid. So let $K\in\cT_h$ be fixed in the following.

\begin{definition}
    For the $C^1$ virtual element spaces, let the \emph{degrees of freedom tuple}, $M \in \Z^{5}$, be defined as  
    \begin{align}\label{eqn: dof tuple}
        M = \big( \vertexValueDofs, \vertexDerivDofs, \edgeValueDofs,\edgeNormalDofs ,\innerDofs \big)~.
    \end{align}
    The entries correspond to the number of moments used in the definition of our degrees of freedom, with $d^v_j$, for $j=0,1$, encoding the information for the vertex moments, $d^e_j$ for $j=0,1$, for the edge moments, and $\innerDofs$ for the inner moments. The subscript $j=0$ corresponds to moments of the function values and $j=1$ to derivative moments on the vertices and edges.
\end{definition}
From the dof tuple, we are able to define the corresponding degrees of freedom.
\begin{definition}\label{defn: local dofs general}
    For a function $v_h \in H^2(K)$, the \emph{local degrees of freedom} corresponding to the degrees of freedom tuple in \eqref{eqn: dof tuple} are given by the following.
    \begin{enumerate}[label=(\emph{D}\arabic*)]
        \item The values $h_v^{j} D^{j} v_h$ for each vertex $v$ of $K$ for $j=0,1$. 
        Here $h_v$ is some local length scale associated to the vertex $v$, e.g., an average of the diameters of all surrounding elements. \label{dof 1}
        \item The moments of $\partial_n^j v_h$ up to order $d^e_j$ on each $e \subset \partial K$, for $j =0,1,$
        \begin{align*}
            |e|^{-1+j} \int_e \partial_n^{j} v_h p \, \mathrm{d} s \quad \forall p \in \prob_{d^e_j}(e).
        \end{align*} 
        \label{dof 2}
        \item The moments of $v_h$ up to order $\innerDofs$ inside $K$,
        \begin{align*}
            \frac{1}{|K|} \int_K v_h p \, \mathrm{d}x \quad \forall p \in \prob_{\innerDofs}(K).
        \end{align*} \label{dof 3}
    \end{enumerate}
\end{definition}

\begin{remark}
    We use the convention that $D^0 v_h = v_h$ and $\partial_n^0 v_h = v_h$. If any of the entries in the dof tuple $M$ are zero, this implies that we take constant moments, if an entry is less than or equal to $-1$, then this corresponds to no moments. So for example $d^v_0=-1$ implies that no vertex values are used. The case $\vertexValueDofs=-1$ is for example relevant for the $C^0$ nonconforming space presented in \cite{cangiani_conforming_2015} but will not be considered in the discussion here.

    Note that for the $C^1$ nonconforming spaces, we always have $\vertexValueDofs = 0$ and $\vertexDerivDofs = -1$. We only begin to prescribe the vertex derivative values when we consider $C^1$ conforming spaces or even higher order conforming spaces see for example \cite{antonietti_conforming_2018, da_veiga_c1_2019}.
\end{remark}

We now give some examples of degrees of freedom tuples relating to some common VEM spaces to illustrate the idea. 

\begin{example}\label{example: dof tuples}
    For $\polOrder \geq 2$, consider the original $C^1$ nonconforming space introduced in \cite{antonietti_fully_2018}. The dof tuple describing this space is
    $$ M^{C^1}_{nc} = (0,-1,\polOrder-3,\polOrder-2,\polOrder-4). $$
    This dof tuple also corresponds to the space considered in \cite{zhao_morley-type_2018}. A visualization of these dofs on triangles can be seen in Figure \ref{figure: C^1 non-conforming dofs for l=2,3,4,5 on triangles}.

    For $\polOrder \geq 2$, consider the $C^1$ space introduced in \cite{zhao_nonconforming_2016, zhang_nonconforming_2020} which is $C^0$ conforming. This space is described by the dof tuple
    $$ M^{C^1}_{C^0 conf} = (0,-1,\polOrder-2,\polOrder-2,\polOrder-4). $$
    The dofs for this space are shown in Figure \ref{figure: C1C0Zhang dofs on triangles}.

    For $\polOrder \geq 3$, consider the $C^1$ conforming space given in \cite{antonietti_conforming_2018}. The dof tuple describing the local dofs for this space is given by
    $$ M^{C^1}_{conf} = (0,0,\polOrder-4,\polOrder-3,\polOrder-4). $$
\end{example}
\input{Sections/tikzDofs}
\begin{remark}\label{rmk: serendipity dofs}
    As mentioned already, the concept of defining a dof tuple extends to describing the dofs relating to $C^0$ VEM spaces. Consider the simplest of the $C^0$ conforming serendipity spaces discussed in \cite{da_veiga_serendipity_2015}. The dofs for this space can be described by the dof tuple
    $$ M^{C^0}_{sc} = (0,-1,\polOrder-2,-1,\polOrder-3). $$
\end{remark}

\subsection{Local spaces and projection operators}
We now focus on the crucial aspect of defining the VEM spaces and projection operators. Following the approach in \cite{ahmad_equivalent_2013,cangiani_conforming_2015} we introduce an enlarged local virtual element space to ensure that our projection operators satisfy certain $L^2$ projection properties which are stated at the end of this section.

\begin{definition}\label{defn: enlarged vem space and extended dof tuple}
    The \emph{enlarged virtual element space} $\enlargedVemSpace^K$ on $K\in\cT_h$ is defined as follows. 
    \begin{align}\label{eqn: enlarged space}
        \enlargedVemSpace^K := \{ v_h \in H^2(K) : \Delta^2 v_h \in \prob_{\polOrder}(K), \, v_h|_e \in \prob_{\polOrder}(e), \, \Delta v_h |_e \in \prob_{\polOrder-2}(e), \, \forall \, e \subset \partial K \} 
    \end{align}
    with dimension $ \dim{\enlargedVemSpace^K} := \widetilde{N}_V^K = n_e^K(2\polOrder -1) + \frac{1}{2}(\polOrder+1)(\polOrder+2)$, where $n_e^K$ denotes the number of edges in the polygon $K$. 

    The enlarged space $\enlargedVemSpace^K$ is characterized by the following \emph{extended degrees of freedom tuple}, denoted by $\widetilde{M}^K$ \red{$\in \Z^5$}, where 
    $$ \widetilde{M}^K = (0,-1,\polOrder-2,\polOrder-2,\polOrder). $$  
\end{definition}
Note that the number of extended dofs is equal to $\widetilde{N}_{dofs} = n_e^K(2\polOrder-1)+\frac{1}{2}(\polOrder+1)(\polOrder+2)$ so that we have $\widetilde{N}_V^K = \widetilde{N}_{dofs}$. We denote with $\widetilde{\Lambda}^K$ the set of extended degrees of freedom described by $\widetilde{M}^K$ as given by Definition \ref{defn: local dofs general}. We show next that this set of degrees of freedom is unisolvent in $\enlargedVemSpace^K$.

\begin{lemma}\label{lemma: extended degrees of freedom unisolvent}
    The set of extended degrees of freedom $\widetilde{\Lambda}^K$ is unisolvent over the space $\enlargedVemSpace^K$.
\end{lemma}

\ifthenelse{\boolean{thesis}}
{
    \begin{proof}
        We show that if all the degrees of freedom vanish for $v_h \in \enlargedVemSpace^K$ then $v_h \equiv 0$. Using Green's formula \cite{clarlet1987finite}, for a function $v_h \in \enlargedVemSpace^K$ it follows that 
        \begin{align*}
            |v_h|_{2,K}^2 = \int_K D^2 v_h : D^2 v_h &= \int_K \Delta v_h \Delta v_h+ \int_{K} 2 \partial_{12} v_h \partial_{12} v_h - \partial_{11} v_h \partial_{22} v_h - \partial_{22} v_h \partial_{11} v_h
            \\
            &= 
            \int_K \Delta^2 v_h v_h - \int_{\partial K} v_h \partial_n (\Delta v_h) + \int_{\partial K} \partial_n v_h (\Delta v_h - \partial_{ss} v_h )
            + \int_{ \partial K } \partial_{ns} v_h \partial_s v_h 
            \intertext{using intergration by parts on the last term we see that}
            \int_{\partial K} \partial_{ns} v_h \partial_s v_h &= \sum_{e \subset \partial K} \int_{ e } \partial_{ns} v_h \partial_s v_h  = - \sum_{e \subset \partial K}  \int_e \partial_{sns} (v_h) v_h + [v_h\partial_{ns}v_h](e^{+}) + [v_h\partial_{ns}v_h](e^{-}) 
        \end{align*}
        where $e^{\pm}$ denote the vertices of an edge $e$.
        Now assuming that all the degrees of freedom $\widetilde{\Lambda}^K$ vanish, implies that the vertex values vanish - see \ref{dof 1}. Therefore,
        \begin{align*}
            |v_h|_{2,K}^2 &= \int_K \Delta^2 v_h v_h - \int_{\partial K} v_h \big( \partial_n (\Delta v_h + \partial_{ss} v_h ) \big) + \int_{\partial K} \partial_n v_h \big( \Delta v_h - \partial_{ss} v_h \big). 
        \end{align*}
        Using integration by parts on the second term, it holds that
        \begin{align*}
            \int_{e} v_h \big( \partial_n (\Delta v_h + \partial_{ss} v_h ) \big) =& - \int_e \partial_n v_h (\Delta v_h + \partial_{ss} v_h ) \\
            &+ [(\Delta v_h + \partial_{ss} v_h ) v_h ](e^{+}) + [(\Delta v_h + \partial_{ss} v_h ) v_h ](e^{-}) 
        \end{align*}
        where the last two terms are zero as well. Finally, it follows that  
        \begin{align*}
            |v_h|_{2,K}^2 &= \int_K \Delta^2 v_h v_h + \int_{\partial K} 2\partial_n v_h \Delta v_h.
        \end{align*}
        Since $v_h \in \enlargedVemSpace^K$, it follows from \eqref{eqn: enlarged space} that $\Delta^2 v_h \in \prob_{\polOrder}(K)$ and $\Delta v_h \in \prob_{\polOrder-2}(e)$. We therefore see that $|v_h|_{2,K}^2 = 0$. As in \cite{antonietti_fully_2018}, this implies that $v_h = 0$ due to the boundary conditions.
    \end{proof}
}
{
    \begin{proof}
        We show that if all the degrees of freedom vanish for $v_h \in \enlargedVemSpace^K$ then $v_h \equiv 0$. Using Green's formula \cite{clarlet1987finite}, for a function $v_h \in \enlargedVemSpace^K$ it follows that 
        \begin{align*}
            |v_h|_{2,K}^2 = \int_K D^2 v_h : D^2 v_h &= \int_K \Delta v_h \Delta v_h+ \int_{K} 2 \partial_{12} v_h \partial_{12} v_h - \partial_{11} v_h \partial_{22} v_h - \partial_{22} v_h \partial_{11} v_h
            \\
            &= 
            \int_K \Delta^2 v_h v_h - \int_{\partial K} v_h \partial_n (\Delta v_h) + \int_{\partial K} \partial_n v_h (\Delta v_h - \partial_{ss} v_h )
            + \int_{ \partial K } \partial_{ns} v_h \partial_s v_h. 
        \end{align*}
        \red{By using integration by parts on the last term and setting all degrees of freedom $\widetilde{\Lambda}^K$ to zero, we see that
        \begin{align*}
            |v_h|_{2,K}^2 &= \int_K \Delta^2 v_h v_h - \int_{\partial K} v_h \big( \partial_n (\Delta v_h + \partial_{ss} v_h ) \big) + \int_{\partial K} \partial_n v_h \big( \Delta v_h - \partial_{ss} v_h \big). 
        \end{align*}
        Applying integration by parts on the second term we arrive at the following
        \begin{align*}
            |v_h|_{2,K}^2 &= \int_K \Delta^2 v_h v_h + \int_{\partial K} 2\partial_n v_h \Delta v_h.
        \end{align*} }
        Since $v_h \in \enlargedVemSpace^K$, it follows from \eqref{eqn: enlarged space} that $\Delta^2 v_h \in \prob_{\polOrder}(K)$ and $\Delta v_h \in \prob_{\polOrder-2}(e)$. We therefore see that $|v_h|_{2,K}^2 = 0$. As in \cite{antonietti_fully_2018}, this implies that $v_h = 0$ due to the boundary conditions.
    \end{proof}
}

\begin{remark}
In this paper we focus on the construction for $C^1$ nonconforming spaces. However, note that the following discussion on defining projection operators also covers some conforming spaces suggested in the literature for solving second order problems, e.g., the spaces from \cite{da_veiga_serendipity_2015}.
As well as this, as discussed above a conforming $C^1$ space can also be described in this framework (using the dof tuple $(0,0,\polOrder-4,\polOrder-3,\polOrder-4)$). The nonconforming $C^0$ space from \cite{cangiani_conforming_2015} also fits the framework using the dof tuple $(-1,-1,\polOrder-1,-1,\polOrder-2)$. Note that both of these spaces require a different enlarged space. Since we are not going to analyse these two spaces, we do not discuss these choices further but we would like to note that our definition of projection operators cover these cases as well with only minimum modifications.
\end{remark}
\red{We now turn our attention to the local VEM space $\localVEMSpace$, which we define as a subspace of the enlarged virtual element space $\enlargedVemSpace^K$.} 
In order to define \red{$\localVEMSpace$,} \strike{the local VEM space,} we first construct the following projections: an interior value projection, $\valueProj : \enlargedVemSpace^K \rightarrow \prob_{\polOrder}(K)$, an edge value projection $\edgeProj : \enlargedVemSpace^K \rightarrow \prob_{\polOrder}(e)$, and an edge normal projection $\edgeNormalProj : \enlargedVemSpace^K \rightarrow \prob_{\polOrder-1}(e)$.
These projections have to be computable from the degrees of freedom $\originalDofs$ of a given $v_h\in \enlargedVemSpace^K$.
Using $\valueProj,\edgeProj,\edgeNormalProj$ we can then define the VEM space $\localVEMSpace$, the gradient projection, ${\gradProj : \enlargedVemSpace^K \rightarrow (\prob_{\polOrder-1}(K))^2}$, and finally the hessian projection ${\hessProj : \enlargedVemSpace^K \rightarrow (\prob_{\polOrder-2}(K))^{2 \times 2}}$, which satisfy certain $L^2$ projection properties discussed in the following. These projections are then used to define the discrete bilinear form.

\red{ To begin with, we assume that we have a degrees of freedom tuple $M^K=(0,-1,\edgeValueDofs,\edgeNormalDofs,\innerDofs)$ such that $d_j^e \leq \polOrder-2$, for $j=0,1$, and $\innerDofs \leq \polOrder$. Fixing the dof tuple $M^K$ gives us a set of degrees of freedom, $\originalDofs$, such that $\originalDofs \subset \widetilde{\Lambda}^K$. }

\begin{definition}\label{defn: dof compatible}
    We say that a value projection $\valueProj$, an edge projection $\edgeProj$, and an edge normal projection $\edgeNormalProj$ are \emph{dof compatible} if for $v_h\in\enlargedVemSpace^K$ they are a linear combination of the \strike{original} dofs $\originalDofs(v_h)$, and satisfy the following additional properties.
    \begin{itemize}
        \item The value projection $\valueProj v_h \in \prob_{\polOrder}(K)$ satisfies 
        \begin{align}\label{eqn: value proj}
            \int_K \valueProj v_h p = \int_K v_h p \quad \forall \, p \in \prob_{\innerDofs}(K),
        \end{align}
        and $\valueProj q = q$ for all $q \in \prob_{\polOrder}(K).$
        \item For each edge $e\subset\partial K$, the edge projection $\edgeProj v_h \in \prob_{\polOrder}(e)$ satisfies
        \begin{align}\label{eqn: edg proj}
            \int_e \edgeProj v_h p = \int_e v_h p  \quad \forall \, p \in \prob_{\edgeValueDofs}(e), \qquad 
            \edgeProj v_h (e^{\pm}) = v_h (e^{\pm}),
        \end{align}
        and $\edgeProj q = q|_e$ for all $q \in \prob_{\polOrder}(K)$.
        \item For each $e\subset\partial K$, the edge normal projection $\edgeNormalProj v_h \in \prob_{\polOrder -1}(e)$ satisfies
        \begin{align}\label{eqn: edge normal proj}
            \int_e \edgeNormalProj v_h p &= \int_e \partial_n v_h p \quad \forall \, p \in \prob_{\edgeNormalDofs}(e),
        \end{align}
        and $\edgeNormalProj q = \partial_n q |_e$ for all $q \in \prob_{\polOrder}(K)$. 
    \end{itemize}    
\end{definition}
Note that there are multiple choices for defining the value, edge, and edge normal projections such that they are dof compatible. We provide an example for defining these projections based on constraint least squares problems after defining the gradient and hessian projections.

\begin{definition}
    The \emph{gradient projection} $\gradProj : \enlargedVemSpace^K \rightarrow (\prob_{\polOrder-1}(K))^2$ is now taken to be 
    \begin{align}\label{eqn: grad proj}
        \int_K \gradProj v_h p = - \int_K \valueProj v_h \nabla p + \sum_{e \subset \partial K } \int_e \edgeProj v_h p n, \quad \forall \, p \in \prob_{\polOrder-1 }(K)^2
    \end{align}
    and the \emph{hessian projection} $\hessProj : \enlargedVemSpace^K \rightarrow (\prob_{\polOrder-2}(K))^{2 \times 2}$ to be
    \begin{align}
        \int_K \hessProj v_h p = - \int_K \gradProj v_h \otimes \nabla p + \sum_{e \subset \partial K } \int_{e} \big(  \edgeNormalProj v_h n \otimes n p + \partial_s ( \edgeProj v_h ) \tau \otimes n p \big),
        \quad \forall \, p \in (\prob_{\polOrder-2}(K))^{2\times 2}.
    \end{align}
    Here $n,\tau$ denote the unit normal and tangent vectors of $e$, respectively. 
\end{definition}

One possible dof compatible choice for the value and two edge projections is shown in the following example.

\begin{example}\label{ex: choice of value, edge, edge normal projections}
We consider projection operators obtained from constraint least squares problems.
Consider the dof tuple $(0,-1,\edgeValueDofs,\edgeNormalDofs,\innerDofs)$.
    \begin{itemize}
        \item We define the value projection $\valueProj v_h \in \prob_{\polOrder}(K)$ as the solution to the problem
        \begin{equation*}
            \begin{split}
                \emph{Minimize: } \quad &\sum_{i=1}^{N_{dof}} (dof_i(\valueProj v_h) - dof_i(v_h))^2, \\
                \emph{ subject to: } \quad &\int_K \valueProj v_h p = \int_K v_h p \quad \forall \, p \in \prob_{\innerDofs}(K).
            \end{split}
        \end{equation*}
        \red{ 
            Since $\valueProj$ is defined by a linear least squares problem with equality constraints, where the right hand side is given by dofs, it follows that $\valueProj$ is a linear combination of the dofs.} 
            From this definition it is clear that \eqref{eqn: value proj} holds. 
        \item 
        \strike{It is clear that} If we choose the edge projection to be the unique solution in $\prob_{\polOrder}(e)$ of
        \begin{equation*}
            \begin{split}
                \int_e \edgeProj v_h p &= \int_e v_h p \quad \forall \, p \in \prob_{\edgeValueDofs}(e), 
                \qquad
                \edgeProj v_h (e^{\pm}) = v_h (e^{\pm}),\\
                \int_e \edgeProj v_h p  &= \int_e \valueProj v_h p \quad \forall \, p \in \prob_{\polOrder-2}(e) \backslash \prob_{\edgeValueDofs}(e)
            \end{split}      
        \end{equation*}
        then $\edgeProj v_h \in \prob_{\polOrder}(e)$ and \eqref{eqn: edg proj} is satisfied. \red{To see this, we note that $\valueProj$ is linear and the other equations only involve linear systems of equations.}
        \item  Finally, if we define the normal edge projection $\edgeNormalProj v_h \in \prob_{l-1}(e)$ to be the unique solution of
        \begin{align*}
            \int_e \edgeNormalProj v_h p = \int_e \partial_n v_h p \quad \forall \, p \in \prob_{\edgeNormalDofs}(e),
            \quad \text{ and } \quad  
            \int_{e} \edgeNormalProj v_h p = \int_e \partial_n (\valueProj v_h )p \quad \forall \, p \in \prob_{\polOrder-1}(e) \backslash \prob_{\edgeNormalDofs}(e)
        \end{align*}
        then we satisfy \eqref{eqn: edge normal proj}.
        Note that we could replace the final constraint with
        \begin{align*}
            \int_{e} \edgeNormalProj v_h p &= \int_e \gradProj v_h \cdot n p \quad \forall \, p \in \prob_{\polOrder-1}(e) \backslash \prob_{\edgeNormalDofs}(e)
        \end{align*}
        since $\gradProj v_h$ does not depend on $\edgeNormalProj v_h$. This is what we use in our implementation.
    \end{itemize}    
    Note that these definitions also cover the case of the $C^0$ nonconforming space with $\edgeNormalDofs=-1$. The gradient projection in
    this case is identical to the one given in \cite{cangiani_conforming_2015} but we
    get a projection for the hessian as well.
\end{example}

We can finally use given \emph{dof compatible} projections $\valueProj,\edgeProj,$ and $\edgeNormalProj$ to define the local virtual element space on $K$.
\begin{definition}
    The \emph{local virtual element space} $\localVEMSpace$ is given as the following subset of the enlarged space.
    \begin{equation}\label{eqn: restricted space}
        \begin{split}
            \localVEMSpace := \big\{ v_h \in \enlargedVemSpace^K : \ (v_h - \valueProj v_h, p)_K &= 0 \quad \forall \, p \in \prob_{\polOrder}(K) \backslash \prob_{\innerDofs}(K),\\
            (\partial_n^j v_h - \Pi^e_j v_h , p)_e &= 0 \quad \forall \, p \in \prob_{\polOrder-2}(e) \backslash \prob_{d_j^e}(e) \text{ for } j=0,1 \big\}.
        \end{split}
    \end{equation}
\end{definition}
We now show that the subset of dofs $\originalDofs$ are unisolvent for our local VEM space $\localVEMSpace$.
\begin{lemma}
    The \strike{original} set of degrees of freedom $\originalDofs$ is unisolvent for $\localVEMSpace$.
\end{lemma}
\begin{proof}
    Let $v_h \in \localVEMSpace$, and set the \strike{original} dofs $\originalDofs$ to zero. Then, noting that both the value and edge projection are computed using the dofs, it follows that $\valueProj v_h = 0, \edgeProj v_h = 0$, and $\edgeNormalProj v_h = 0$. From \eqref{eqn: restricted space}, we see that 
    \begin{align*}
        (v_h  - \valueProj v_h,p)_K &= (v_h,p)_K = 0 \quad \forall \, p \in \prob_{\polOrder}(K) \backslash \prob_{\innerDofs}(K),
    \intertext{and for $j=0,1,$}
        (\partial_n^j v_h - \Pi^e_j v_h , p)_e &= (\partial_n^j v_h,p)_e  = 0 \quad \forall \, p \in \prob_{\polOrder-2}(e) \backslash \prob_{d_j^e}(e), 
    \end{align*}
    i.e. the inner moments of order $\polOrder, \dots, \polOrder - ( \innerDofs +1)$ and the edge moments of order $ \polOrder -2, \dots, \polOrder - (\edgeValueDofs +1)$ of $v_h$ are also zero. Similarly, the higher order edge normal moments are also zero. So viewing $v_h$ as a function in the enlarged space, $v_h \in \localVEMSpace \subset \enlargedVemSpace^K$ it follows that $v_h \equiv 0 $ as a function in $\enlargedVemSpace^K$ with the extended degrees of freedom set to zero, using Lemma \ref{lemma: extended degrees of freedom unisolvent}. Therefore the \strike{original} dofs are unisolvent for $\localVEMSpace$.
\end{proof}

This leads us to a final crucial result of this subsection. We see in the next Lemma that by construction, our value, gradient, and hessian projections possess important $L^2$ projection properties. In particular, due to construction of the VEM space, the value projection is identical to the $L^2$ projection into the space of polynomials of degree $l$ for all VEM functions.

\begin{lemma}\label{lemma: l2 projections assumptions satisfied}
    If the value, edge, and edge normal projections are \emph{dof compatible}, 
    then for any $v_h \in \localVEMSpace$ it follows that, for $s=0,1,2$,
    \begin{align*}
        \Pi^K_s v_h = \cP^{\polOrder-s}_K (D^s v_h).
    \end{align*}
    \strike{Due to the dof compatibility, and the fact that $\prob_{\polOrder}(K)\subset \enlargedVemSpace^K$, it holds that $\prob_{\polOrder}(K) \subset \localVEMSpace$.}

    \strike{This implies that $\Pi^K_s p =  D^s p$ for $p \in \prob_{\polOrder}(K)$ for each $s=0,1,2$.}
\end{lemma}

\begin{proof}
\strike{Using the dof compatibility for the value projection $\valueProj$,} \red{From Definition \ref{defn: dof compatible} it is clear that}
\begin{align*}
    \int_K \valueProj v_h p &= \int_K v_h p \quad \forall p \in \prob_{\innerDofs}(K) 
\intertext{while the definition of the space $\localVEMSpace$ \eqref{eqn: restricted space} leads to }
    \int_K  \valueProj v_h p &= \int_K v_h p \quad \forall p \in \prob_{\polOrder}(K) \backslash \prob_{\innerDofs}(K).
\end{align*}
Combining the two we obtain the stated $L^2$ projection property of $\valueProj$. Similar arguments are used for the other projections: due to the polynomial exactness of $\edgeProj$ and the definition of the extended space $\enlargedVemSpace$ we have $\edgeProj v_h = v_h|_e$  and therefore for a polynomial $p \in \prob_{\polOrder-1}(K)$,
\begin{align*}
    \int_K \gradProj v_h p = - \int_K \valueProj v_h \nabla p + \sum_{e \subset \partial K} \int_e  \edgeProj v_h n p
   = - \int_K v_h \nabla p + \sum_{e \subset \partial K} \int_e  v_h n p
\end{align*}
where we used the $L^2$ projection property of $\valueProj$. Applying integration by parts to the RHS completes the proof. Finally we have for the hessian projection using the results already proven and the dof compatibility of the edge normal projection, $\edgeNormalProj$,
\begin{align*}
    \int_K \hessProj v_h p &= - \int_K \gradProj v_h \nabla p + \sum_{e \subset \partial K } \big( \int_{e} \edgeNormalProj v_h n \otimes n p + \partial_s ( \edgeProj v_h ) \tau \otimes n p \big) \\
    &= - \int_K \nabla v_h \nabla p + \sum_{e \subset \partial K} \big( \int_e \partial_n v_h n \otimes n p + \partial_s v_h \tau \otimes n p \big) 
 \end{align*}
so that the result follows using integration by parts.
\end{proof}

\red{Due to the dof compatibility, and the fact that $\prob_{\polOrder}(K)\subset \enlargedVemSpace^K$, it holds that $\prob_{\polOrder}(K) \subset \localVEMSpace$. This implies that $\Pi^K_s p =  D^s p$ for $p \in \prob_{\polOrder}(K)$ for each $s=0,1,2$.}


\subsection{Global spaces and the discrete bilinear form}
We conclude with the definition of the global virtual element space and the discrete bilinear form. \strike{, ensuring that it satisfies the abstract assumptions given in}
We keep the presentation brief, since it follows the general construction of VEM found in the literature.

We define the \emph{global space} $\vemSpace$ by
\begin{align}\label{eqn: global VEM space}
    \vemSpace := \{ v_h \in \red{H^2(\cT_h)} : v_h|_K \in \localVEMSpace \quad \forall \, K \in \cT_h \}. 
\end{align}
\strike{Recall that $\HTwoNCSpace$ is the $H^2$ nonconforming space detailed in}

\begin{definition}\label{defn: H^{2,nc} conforming dof tuple}
    \red{A dof tuple is said to be \emph{$\HTwoNCSpace$ conforming} if the entries satisfy: $\vertexValueDofs = 0$, $\vertexDerivDofs = -1,$ $\edgeValueDofs \geq \polOrder-3$, $\edgeNormalDofs \geq \polOrder -2$, and $\innerDofs \geq \polOrder -4 $.}
    
    \red{Recall that $\HTwoNCSpace$ is the $H^2$ nonconforming space detailed in \ref{assumption: A2}. Observe that a VEM space $\vemSpace$ defined by an $\HTwoNCSpace$ conforming dof tuple satisfies ${\vemSpace \subset \HTwoNCSpace}$. }
\end{definition}

We extend Definition \ref{defn: local dofs general} to arrive at the global degrees of freedom.
\begin{definition}\label{defn: global dofs general}
    \emph{Global degrees of freedom} are given by the following. 
    \begin{enumerate}[label=(\emph{D}\arabic*)]
        \item The values $h_v^{j} D^{j} v_h$ for each internal vertex $v$ of $\cT_h$ for $j=0,1$. 
        \item The moments of $\partial_n^j v_h$ up to order $d_j^e$ for $j=0,1$, on each internal edge $e \in \mathcal{E}_h^{\text{int}}$ 
        \begin{align*}
            |e|^{-1+j }\int_e \partial_n^j v_h p \ \mathrm{d}s \quad \forall p \in \prob_{d_j^e}(e).
        \end{align*} 
        \item The moments of $v_h$ up to order $\innerDofs$ inside each $K \in \cT_h$
        \begin{align*}
            \frac{1}{|K|} \int_K v_h p \, \mathrm{d}x \quad \forall p \in \prob_{\innerDofs}(K).
        \end{align*}
    \end{enumerate}
\end{definition}

We set the local degrees of freedom which correspond to boundary vertices and boundary edges, $e \in \mathcal{E}_h^{bdry}$, to zero. Note that the global degrees of freedom are unisolvent - this follows from the unisolvency of the local degrees of freedom and the definition of the local spaces.  

Lastly, we define the discrete bilinear form.
\begin{definition}\label{eqn: discrete bilinear form}
    For any ${u_h, v_h \in \localVEMSpace}$, define the \emph{local discrete bilinear form} $a_h^K $ as 
    \begin{equation*}
        \begin{split}
            a_h^K (u_h,v_h) := &\int_K \kappa \hessProj u_h : \hessProj v_h + \int_K \beta \gradProj u_h \cdot \gradProj v_h + \int_K \gamma \valueProj u_h \valueProj v_h + S^K ( u_h - \valueProj u_h , v_h - \valueProj v_h ).
        \end{split}    
    \end{equation*}
    We take the \emph{stabilization term} $S^K(\cdot,\cdot)$ to be a symmetric, positive definite bilinear form satisfying
    \begin{align}\label{eqn: S^K bounded}
        c_* a^K(v_h,v_h) \leq S^K(v_h,v_h) \leq c^* a^K(v_h,v_h)
    \end{align}
    for constants $c_*$, $c^*$ independent of $h$ and $K$.
\end{definition}
It is clear that the choice of $a^K_h$ in Definition \ref{eqn: discrete bilinear form} with $S^K$ satisfying \eqref{eqn: S^K bounded}, will satisfy the stability property from \ref{assumption: stability property} see e.g., \cite{beirao_da_veiga_basic_2013,cangiani_conforming_2015}. 
We take the standard choice for $S^K$ based on the scalar product of the local degrees of freedom scaled appropriately with some element wise constant approximations of the coefficients $\bar \kappa_K$, $\bar \beta_K$, and $\bar \gamma_K$. \red{In our implementation we use the average of each coefficient $\kappa, \beta,$ and $\gamma$ evaluated at the vertices of an element $K$ to calculate $\bar{\kappa}_K, \bar{\beta}_K,$ and $\bar{\gamma}_K$. Then $S^K$ is given by:}
\begin{align*}
    S^K( u_h , v_h) := \big( \bar \kappa_K h_K^{-2} + \bar \beta_K + \bar \gamma_K h_K^2 \big) \sum_{j=1}^{N_{dof}} dof_j(u_h) dof_j(v_h).
\end{align*}
Then, as in \cite{cangiani_conforming_2015}, it follows that this choice of $S^K$ satisfies \eqref{eqn: S^K bounded} and therefore \ref{assumption: stability property} holds in Assumption \ref{assumptions: discrete assumptions}. 
\begin{remark}\label{remark: polynomial consistency}
    Furthermore, notice that due to Lemma \ref{lemma: l2 projections assumptions satisfied}, it holds that $\valueProj p = p$ for all $p \in \prob_{\polOrder}(K)$ and therefore the stabilization term vanishes when one of $u_h$ or $v_h$ is a polynomial of degree $\polOrder$. The discrete bilinear form therefore possesses the standard \emph{polynomial consistency} property.
\end{remark}

%% file: Sections/tikzDofs.tex
\begin{figure}[!h]
    \centering
    \resizebox{\textwidth}{!}{%
    \begin{tikzpicture}
        \draw (0,0) -- (2,0);
        \draw (0,0) -- (1,1.73);
        \draw (1,1.73) -- (2,0);
        \fill[blue!40!white] (0,0) circle (0.075cm);
        \fill[blue!40!white] (1,1.73) circle (0.075cm);
        \fill[blue!40!white] (2,0) circle (0.075cm);
        \draw[arrows=->,line width=.4pt] (1.5,0.87) -- (1.93,1.12);
        \draw[arrows=->,line width=.4pt] (1,0) -- (1,-0.5);
        \draw[arrows=->,line width=.4pt] (0.5,0.87) -- (0.067,1.12);
        \node at (1, -1) {$\polOrder = 2$};
    
        \draw (4,0) -- (6,0);
        \draw (4,0) -- (5,1.73);
        \draw (5,1.73) -- (6,0);
        \fill[blue!40!white] (6,0) circle (0.075cm);
        \fill[blue!40!white] (5,1.73) circle (0.075cm);
        \fill[blue!40!white] (4,0) circle (0.075cm);
        \draw[arrows=->,line width=.4pt] (4.8,0) -- (4.8,-0.5);
        \draw[arrows=->,line width=.4pt] (5.2,0) -- (5.2,-0.5);
        \draw[arrows=->,line width=.4pt] (5.367,1.046) -- (5.81,1.302);
        \draw[arrows=->,line width=.4pt] (5.61,0.674) -- (6.06,0.93);
        \draw[arrows=->,line width=.4pt] (4.595,1.03) -- (4.152,1.28);
        \draw[arrows=->,line width=.4pt] (4.39,0.68) -- (3.941,0.94);
        \fill[red!40!white] (4.5,0.87) circle (0.075cm);
        \fill[red!40!white] (5.5,0.87) circle (0.075cm);
        \fill[red!40!white] (5,0) circle (0.075cm);
        \node at (5, -1) {$\polOrder = 3$};

        \draw (8,0) -- (10,0);
        \draw (8,0) -- (9,1.73);
        \draw (9,1.73) -- (10,0);
        \fill[blue!40!white] (10,0) circle (0.075cm);
        \fill[blue!40!white] (9,1.73) circle (0.075cm);
        \fill[blue!40!white] (8,0) circle (0.075cm);
        \fill[green] (8.925,0.505) rectangle (9.075,0.655);

        \draw [arrows=->,line width=.4pt] (8.5,0) -- (8.5,-0.5);
        \draw [arrows=->,line width=.4pt] (9,0) -- (9,-0.5);
        \draw [arrows=->,line width=.4pt] (9.5,0) -- (9.5,-0.5);

        \fill[red!40!white]  (8.75,0) circle (0.075cm);
        \fill[red!40!white] (9.25,0) circle (0.075cm);
        \fill[red!40!white] (8.625,1.082531754730546) circle (0.075cm);
        \fill[red!40!white] (8.375,0.6495190528383272) circle (0.075cm);
        \fill[red!40!white] (9.375,1.0825317547305477) circle (0.075cm);
        \fill[red!40!white] (9.625,0.649519052838329) circle (0.075cm);

        \draw [arrows=->,line width=.4pt] (8.25,0.43301270189221786) -- (7.81585938579443,0.683663902369954);
        \draw [arrows=->,line width=.4pt] (8.5,0.8660254037844366) -- (8.065859385794429,1.1166766042621727);
        \draw [arrows=->,line width=.4pt] (8.75,1.2990381056766562) -- (8.315859385794429,1.5496893061543924);

        \draw [arrows=->,line width=.4pt] (9.75,0.43301270189221874) -- (10.184140614205571,0.683663902369954);
        \draw [arrows=->,line width=.4pt] (9.5,0.8660254037844393) -- (9.934140614205571,1.1166766042621745);
        \draw [arrows=->,line width=.4pt] (9.25,1.299038105676658) -- (9.684140614205571,1.5496893061543933);
        \node at (9, -1) {$\polOrder = 4$};

        \draw (12,0) -- (14,0);
        \draw (12,0) -- (13,1.73);
        \draw (13,1.73) -- (14,0);
        \fill[blue!40!white] (14,0) circle (0.075cm);
        \fill[blue!40!white] (13,1.73) circle (0.075cm);
        \fill[blue!40!white] (12,0) circle (0.075cm);

        \fill[green] (12.925,0.605) rectangle (13.075,0.755);
        \fill[green] (12.825,0.405) rectangle (12.975,0.555);
        \fill[green] (13.025,0.405) rectangle (13.175,0.555);

        \fill[red!40!white] (13.497533850441435,0.8702969001189368) circle (0.075cm);
        \fill[red!40!white] (12.295067700882871,0.5281582349396556) circle (0.075cm);
        \fill[red!40!white] (12.495067700882872,0.874568396453431) circle (0.075cm);
        \fill[red!40!white] (12.695067700882873,1.2209785579672063) circle (0.075cm);

        \fill[red!40!white] (13.3,1.2124355652982128) circle (0.075cm);
        \fill[red!40!white] (13.5,0.8660254037844375) circle (0.075cm);
        \fill[red!40!white] (13.7,0.5196152422706621) circle (0.075cm);

        \fill[red!40!white] (12.6,0) circle (0.075cm);
        \fill[red!40!white] (13,0) circle (0.075cm);
        \fill[red!40!white] (13.4,0) circle (0.075cm);

        \draw [arrows=->,line width=.4pt] (13.2,0) -- (13.2,-0.5);
        \draw [arrows=->,line width=.4pt] (13.6,0) -- (13.6,-0.5);
        \draw [arrows=->,line width=.4pt] (12.4,0) -- (12.4,-0.5);
        \draw [arrows=->,line width=.4pt] (12.8,0) -- (12.8,-0.5);

        \draw [arrows=->,line width=.4pt] (13.2,1.3856406460551014) -- (13.579838221736962,1.5988388306005594);
        \draw [arrows=->,line width=.4pt] (13.4,1.0392304845413243) -- (13.778876472189724,1.2540944681667543);
        \draw [arrows=->,line width=.4pt] (13.6,0.6928203230275507) -- (13.982025138915313,0.9022306559081561);
        \draw [arrows=->,line width=.4pt] (13.8,0.34641016151377535) -- (14.186094576856972,0.5487720211211804);

        \draw [arrows=->,line width=.4pt] (12.795067700882871,1.394183638724095) -- (12.429550603832414,1.6054755208373084);
        \draw [arrows=->,line width=.4pt] (12.595067700882872,1.0477734772103178) -- (12.243583761909406,1.2833715021035295);
        \draw [arrows=->,line width=.4pt] (12.395067700882873,0.7013633156965442) -- (12.03859355275168,0.9283180447881755);
        \draw [arrows=->,line width=.4pt] (12.195067700882873,0.35495315418276885) -- (11.842317615947124,0.5883581499395047);
        \node at (13, -1) {$\polOrder = 5$};
    \end{tikzpicture}}
    \caption{Degrees of freedom for polynomial orders $\polOrder=2,3,4,5$ on triangles for the $C^1$ nonconforming space. \red{Circles at vertices} represent vertex dofs, arrows represent edge normal dofs, \red{circles on edges} represent edge value moments and interior \red{squares} represent inner dofs.}
    \label{figure: C^1 non-conforming dofs for l=2,3,4,5 on triangles}
\end{figure}

\begin{figure}[!h]
    \centering
    \resizebox{\textwidth}{!}{%
    \begin{tikzpicture}
                \draw (0,0) -- (2,0);
                \draw (0,0) -- (1,1.73);
                \draw (1,1.73) -- (2,0);
                \fill[blue!40!white] (0,0) circle (0.075cm);
                \fill[blue!40!white] (1,1.73) circle (0.075cm);
                \fill[blue!40!white] (2,0) circle (0.075cm);
                \draw[arrows=->,line width=.4pt] (1.5,0.87) -- (1.93,1.12);
                \draw[arrows=->,line width=.4pt] (1,0) -- (1,-0.5);
                \draw[arrows=->,line width=.4pt] (0.5,0.87) -- (0.067,1.12);
                \fill[red!40!white] (0.5,0.87) circle (0.075cm);
                \fill[red!40!white] (1.5,0.87) circle (0.075cm);
                \fill[red!40!white] (1,0) circle (0.075cm);
                \node at (1, -1) {$\polOrder = 2$};
            
                \draw (4,0) -- (6,0);
                \draw (4,0) -- (5,1.73);
                \draw (5,1.73) -- (6,0);
                \fill[blue!40!white] (6,0) circle (0.075cm);
                \fill[blue!40!white] (5,1.73) circle (0.075cm);
                \fill[blue!40!white] (4,0) circle (0.075cm);
                \draw[arrows=->,line width=.4pt] (4.8,0) -- (4.8,-0.5);
                \draw[arrows=->,line width=.4pt] (5.2,0) -- (5.2,-0.5);
                \draw[arrows=->,line width=.4pt] (5.367,1.046) -- (5.81,1.302);
                \draw[arrows=->,line width=.4pt] (5.61,0.674) -- (6.06,0.93);
                \draw[arrows=->,line width=.4pt] (4.595,1.03) -- (4.152,1.28);
                \draw[arrows=->,line width=.4pt] (4.39,0.68) -- (3.941,0.94);
                \fill[red!40!white] (5.61,0.674) circle (0.075cm);
                \fill[red!40!white] (4.595,1.03) circle (0.075cm);
                \fill[red!40!white] (4.39,0.68) circle (0.075cm);
                \fill[red!40!white] (4.8,0) circle (0.075cm);
                \fill[red!40!white] (5.2,0) circle (0.075cm);
                \fill[red!40!white] (5.367,1.046) circle (0.075cm);
                \node at (5, -1) {$\polOrder = 3$};
    
                \draw (8,0) -- (10,0);
                \draw (8,0) -- (9,1.73);
                \draw (9,1.73) -- (10,0);
                \fill[blue!40!white] (10,0) circle (0.075cm);
                \fill[blue!40!white] (9,1.73) circle (0.075cm);
                \fill[blue!40!white] (8,0) circle (0.075cm);
    
                \draw [arrows=->,line width=.4pt] (8.5,0) -- (8.5,-0.5);
                \draw [arrows=->,line width=.4pt] (9,0) -- (9,-0.5);
                \draw [arrows=->,line width=.4pt] (9.5,0) -- (9.5,-0.5);
    
                \draw [arrows=->,line width=.4pt] (8.25,0.43301270189221786) -- (7.81585938579443,0.683663902369954);
                \draw [arrows=->,line width=.4pt] (8.5,0.8660254037844366) -- (8.065859385794429,1.1166766042621727);
                \draw [arrows=->,line width=.4pt] (8.75,1.2990381056766562) -- (8.315859385794429,1.5496893061543924);
    
                \draw [arrows=->,line width=.4pt] (9.75,0.43301270189221874) -- (10.184140614205571,0.683663902369954);
                \draw [arrows=->,line width=.4pt] (9.5,0.8660254037844393) -- (9.934140614205571,1.1166766042621745);
                \draw [arrows=->,line width=.4pt] (9.25,1.299038105676658) -- (9.684140614205571,1.5496893061543933);
    
                \fill[red!40!white] (8.25,0.43301270189221786) circle (0.075cm);
                \fill[red!40!white] (8.5,0.8660254037844366) circle (0.075cm);
                \fill[red!40!white] (8.75,1.2990381056766562) circle (0.075cm);
                \fill[red!40!white] (8.5,0) circle (0.075cm);
                \fill[red!40!white] (9,0) circle (0.075cm);
                \fill[red!40!white] (9.5,0) circle (0.075cm);
                \fill[red!40!white] (9.75,0.43301270189221874) circle (0.075cm);
                \fill[red!40!white] (9.5,0.8660254037844393) circle (0.075cm);
                \fill[red!40!white] (9.25,1.299038105676658) circle (0.075cm);
                \fill[green] (8.925,0.505) rectangle (9.075,0.655);
                \node at (9, -1) {$\polOrder = 4$};
    
                \draw (12,0) -- (14,0);
                \draw (12,0) -- (13,1.73);
                \draw (13,1.73) -- (14,0);
                \fill[blue!40!white] (14,0) circle (0.075cm);
                \fill[blue!40!white] (13,1.73) circle (0.075cm);
                \fill[blue!40!white] (12,0) circle (0.075cm);
    
                \draw [arrows=->,line width=.4pt] (13.2,0) -- (13.2,-0.5);
                \draw [arrows=->,line width=.4pt] (13.6,0) -- (13.6,-0.5);
                \draw [arrows=->,line width=.4pt] (12.4,0) -- (12.4,-0.5);
                \draw [arrows=->,line width=.4pt] (12.8,0) -- (12.8,-0.5);
    
                \draw [arrows=->,line width=.4pt] (13.2,1.3856406460551014) -- (13.579838221736962,1.5988388306005594);
                \draw [arrows=->,line width=.4pt] (13.4,1.0392304845413243) -- (13.778876472189724,1.2540944681667543);
                \draw [arrows=->,line width=.4pt] (13.6,0.6928203230275507) -- (13.982025138915313,0.9022306559081561);
                \draw [arrows=->,line width=.4pt] (13.8,0.34641016151377535) -- (14.186094576856972,0.5487720211211804);
    
                \draw [arrows=->,line width=.4pt] (12.795067700882871,1.394183638724095) -- (12.429550603832414,1.6054755208373084);
                \draw [arrows=->,line width=.4pt] (12.595067700882872,1.0477734772103178) -- (12.243583761909406,1.2833715021035295);
                \draw [arrows=->,line width=.4pt] (12.395067700882873,0.7013633156965442) -- (12.03859355275168,0.9283180447881755);
                \draw [arrows=->,line width=.4pt] (12.195067700882873,0.35495315418276885) -- (11.842317615947124,0.5883581499395047);
    
                \fill[red!40!white] (13.2,0) circle (0.075cm);
                \fill[red!40!white] (13.6,0) circle (0.075cm);
                \fill[red!40!white] (12.4,0) circle (0.075cm);
                \fill[red!40!white] (12.8,0) circle (0.075cm);
    
                \fill[red!40!white] (13.2,1.3856406460551014) circle (0.075cm);
                \fill[red!40!white] (13.4,1.0392304845413243) circle (0.075cm);
                \fill[red!40!white] (13.6,0.6928203230275507) circle (0.075cm);
                \fill[red!40!white] (13.8,0.34641016151377535) circle (0.075cm);
    
                \fill[red!40!white] (12.795067700882871,1.394183638724095) circle (0.075cm);
                \fill[red!40!white] (12.595067700882872,1.0477734772103178) circle (0.075cm);
                \fill[red!40!white] (12.395067700882873,0.7013633156965442) circle (0.075cm);
                \fill[red!40!white] (12.195067700882873,0.35495315418276885) circle (0.075cm);
    
                \fill[green] (12.925,0.605) rectangle (13.075,0.755);
                \fill[green] (12.825,0.405) rectangle (12.975,0.555);
                \fill[green] (13.025,0.405) rectangle (13.175,0.555);

                \node at (13, -1) {$\polOrder = 5$};
        \end{tikzpicture}}
        \caption{Degrees of freedom for polynomial orders $\polOrder=2,3,4,5$ on triangles for the $C^1$-$C^0$ conforming space \cite{zhang_nonconforming_2020}. \red{Circles at vertices} represent vertex dofs, arrows represent edge normal dofs, \red{circles on edges} represent edge value moments and interior \red{squares} represent inner dofs.}
            \label{figure: C1C0Zhang dofs on triangles}
    \end{figure}

%% file: Sections/errorAnalysis.tex
\section{Error analysis}\label{section: error analysis}
In this section we collect all the building blocks needed to prove a general convergence result in the energy norm, which is presented in Theorem \ref{thm: energy norm}.  
\red{Throughout this section we now assume that our dof tuple is $\HTwoNCSpace$ conforming (recall Definition~\ref{defn: H^{2,nc} conforming dof tuple}) and so \ref{assumption: A2} is satisfied in Assumption~\ref{assumptions: discrete assumptions}, meaning that  Theorem~\ref{thm: a-priori error bound energy norm} holds.}
Proving convergence involves bounding all of the terms in the Strang-type estimate presented in Theorem~\ref{thm: a-priori error bound energy norm}.
In the following we consider each term from \eqref{eqn: a-priori bound} in the order in which they appear, starting with the
approximation error. 

\subsection{Interpolation error estimate}
The following estimate for the interpolation error, the first term in equation \eqref{eqn: a-priori bound}, is simply a consequence of standard scaling arguments.

\begin{theorem}[Approximation error estimate]\label{thm: interpolation error estimate}
    Let \ref{mesh} - \ref{assump: star shaped wrt a ball} hold, defined in Assumptions \ref{assumptions: discrete assumptions} and \ref{assumption: mesh regularity}. 
    \red{If} \strike{Let} $u \in H^{\polOrder+1}(\Omega)$ \red{is} \strike{be} the solution to the continuous problem \eqref{eqn: cts problem}. 
    Then, it follows that 
    \begin{align*}
            \inf_{v_h \in \vemSpace} \vertiii{u-v_h}_h 
            \leq C\big( h^{\polOrder -1} \|\kappa\|_{L^{\infty}}^{\frac{1}{2}} + h^{ \polOrder } \|\beta\|_{L^{\infty}}^{\frac{1}{2}}  + h^{\polOrder+1} \|\gamma\|_{L^{\infty}}^{\frac{1}{2}} 
            \big) |u|_{\polOrder+1}.
    \end{align*}
    for a constant $C$ independent of $h$.
\end{theorem}

\subsection{Load term}
The next term to appear in the a priori bound \eqref{eqn: a-priori bound} is an error estimate for the load term. To treat this term, following the methods in \cite{antonietti_fully_2018,cangiani_conforming_2015}, we define $f_h$ to be the piecewise $L^2$ projection of $f$ on $\cT_h$, $f_h := \cP^{\polOrder}_K f$. Using Lemma \ref{lemma: l2 projections assumptions satisfied} we observe that for $w_h \in\localVEMSpace$,
\begin{align*}
    \langle f_h, w_h \rangle &=
    \sum_{K \in \cT_h} (\cP^{\polOrder}_K f,w_h)_K
    =
    \sum_{K \in \cT_h} (f,\cP^{\polOrder}_K w_h)_K
    =
    \sum_{K \in \cT_h} (f,\valueProj w_h)_K.
\end{align*}
Consequently, the right hand side of the discrete variational problem \eqref{eqn: discrete problem} is computable. The following estimate now follows easily using for example, the method in \cite{antonietti_fully_2018}. 
\begin{lemma}\label{lemma: load term estimate}
    For $\polOrder \geq 2$, let $s$ be an integer with $f \in H^{s}(\Omega)$ and define $r:= \min(s-1,\polOrder)$. Then, the following estimate holds
    \begin{align}\label{eqn: load term}
        | \langle f_h, w_h \rangle - (f,w_h) | 
        \leq C h^{r+2} |f|_{r+1} |w_h|_{1,h}
    \end{align}
    for a constant $C$ independent of $h$.
\end{lemma}
\begin{proof}
    We can show the following using the bounds of Theorem \ref{thm: interpolation estimates}
    \begin{align*}
        \big| \langle f_h, w_h \rangle - (f,w_h) \big|
        =& \bigg| \sum_{K \in \cT_h} \int_K \big( \mathcal{P}^{\polOrder}_K  f - f \big) w_h \bigg|
        = \bigg| \sum_{K \in \cT_h} \int_K \big( \mathcal{P}^{\polOrder}_K  f - f \big) \big( w_h - \mathcal{P}_K^0 w_h \big) \bigg| \\
        \leq& \sum_{K \in \cT_h} \| \mathcal{P}^{\polOrder}_K  f - f  \|_{0,K} \|  w_h - \mathcal{P}_K^0 w_h \|_{0,K}
        \leq 
        C h^{r+1} |f|_{r+1} h^1 |w_h|_{1,h}, 
    \end{align*}
    hence \eqref{eqn: load term} holds, as required.
\end{proof}

\subsection{Nonconformity error}
We now turn our attention to the next term in the a priori bound \eqref{eqn: a-priori bound}, the \emph{nonconformity error} $\mathcal{N}(u,w_h)$. 
From now on we assume that $\polOrder \geq 2$. \red{Recall that our dof tuple is $\HTwoNCSpace$ conforming and as per Definition~\ref{defn: H^{2,nc} conforming dof tuple},} the values in our degrees of freedom tuple satisfy the following $\vertexValueDofs = 0$, $\vertexDerivDofs = -1,$ ${\edgeValueDofs \geq \polOrder-3}$, $\edgeNormalDofs \geq \polOrder -2$, and $\innerDofs \geq \polOrder -4 $. 
Recall that as a consequence of this for $w_h \in \vemSpace$, and for any edge $e$ in the grid
\begin{align}
    \int_e [ \, w_h \, ] p \, \mathrm{d}s = 0 \quad \forall \, p \in \prob_{\polOrder-3}(e), \label{eqn: jump of wh}\\
    \int_e [ \, \partial_n w_h \, ] p \, \mathrm{d}s = 0 \quad \forall \, p \in \prob_{\polOrder-2}(e). \label{eqn: jump of normal wh}
\end{align}
Applying integration by parts gives us, for any $q \in \prob_{\polOrder-2}(e)$,
\begin{align}\label{eqn: jump of tangential w_h}
    \int_e [\, \partial_s w_h \, ] q \, \mathrm{d}s &= 0.   
\end{align}
This follows since the jump is zero at the vertices on the edges and $\partial_s q \in \prob_{\polOrder-2 -1}(e) = \prob_{\polOrder-3}(e)$ which means that we can apply \eqref{eqn: jump of wh}. 

\begin{lemma}[Nonconformity error]\label{lemma: nonconformity proof}
    We assume that the coefficients satisfy $\kappa \in W^{2,\infty}(\Omega)$ and $\beta \in W^{1,\infty}(\Omega)$. 
    Assume that the solution $u$ to \eqref{eqn: cts problem} satisfies $u \in H^4(\Omega)$.
    Then, for $w_h \in \vemSpace$ the nonconformity error $\mathcal{N}(u,w_h)$ satisfies 
    \begin{align}
        \begin{split}
            \mathcal{N}(u,w_h) = \ & \sum_{e \in \mathcal{E}_h} \Big\{ 
                \int_e \kappa \big( \Delta u - \partial_{ss} u \big) [ \partial_n w_h ] 
                + 
                \int_e \kappa \partial_{ns} u \, [ \partial_s w_h ]     
                \\
                &+ \int_e \big( \partial_{ss} u \partial_n \kappa - \partial_{ns} u \partial_s \kappa + \beta \partial_n u - \partial_n (\kappa \Delta u)  \big) [ \,  w_h \, ] 
                \Big\}.
                \label{eqn: lemma nonconformity error}
            \end{split}
        \end{align}

\end{lemma}
\begin{proof}
    Using integration by parts, we can express the hessian and gradient terms as follows 
    \begin{align}
        \int_K \kappa D^2 u : D^2 w_h 
        \label{eqn: hessian term nonconf error}
        &=
        -\int_K \sum_{i,j=1}^2 \partial_j (\kappa \partial_{ij} u) \partial_i w_h + \int_{\partial K} \sum_{i,j=1}^2 \kappa \partial_{ij} u \partial_i w_h n_j
        \\
        \int_K \beta Du \cdot Dw_h &= - \int_K D \cdot (\beta Du) w_h + \int_{\partial K} \beta \partial_n u w_h.
        \label{eqn: grad term nonconf error}
    \end{align}
    Since we assume $u \in H^4(\Omega)$ we use \eqref{eqn: strong form} and an application of integration by parts to see that
    \begin{align*}
        (f,w_h) 
        = \sum_{K \in \cT_h}
        - \int_{K} \sum_{i,j=1}^2 \partial_j (\kappa \partial_{ij} u) \partial_i w_h
        +
        \int_{\partial K} \sum_{i,j=1}^2 \partial_j (\kappa \partial_{ij}u) w_h n_i
        - \int_{K} D \cdot (\beta D u)w_h + \int_{K} \gamma u w_h.
    \end{align*}
    Therefore the nonconformity error is equal to 
    \begin{align}\label{eqn: new nonconf error mid proof}
        \mathcal{N}(u,w_h) = a(u,w_h) - (f,w_h)
        =
        \sum_{K \in \cT_h} \Big\{
        \int_{\partial K} \sum_{i,j=1}^2 \kappa \partial_{ij} u \partial_i w_h n_j
        -
        \int_{\partial K} \sum_{i,j=1}^2 \partial_j (\kappa \partial_{ij}u) w_h n_i
        +
        \int_{\partial K} \beta \partial_n u w_h
        \Big\}.
    \end{align}

    We also use the following identities for rewriting the boundary terms in a way that is useful later on in Theorem \ref{thm: non conformity error without (H4)}.
    \begin{align}
        \int_{\partial K} \sum_{i,j=1}^2 \kappa \partial_{ij} u \partial_i w_h n_j 
        &= 
        \int_{\partial K} \kappa \big( (\Delta u - \partial_{ss} u ) \partial_n w_h + \partial_{ns} u \partial_s w_h \big ) 
        \label{eqn: nonconf bdy terms 1}
        \\
        \int_{\partial K} \sum_{i,j=1}^2 \partial_j (\kappa \partial_{ij} u) w_h n_i 
        &=
        \int_{\partial K} \big( \partial_{ns} u \partial_s \kappa - \partial_{ss} u \partial_n \kappa  + \partial_n (\kappa \Delta u) \big) w_h
        \label{eqn: nonconf bdy terms 2}
    \end{align}
    It is straightforward albeit tedious to show that \eqref{eqn: nonconf bdy terms 1} and \eqref{eqn: nonconf bdy terms 2} hold. It is now clear that the result \eqref{eqn: lemma nonconformity error} follows from these expressions and equation \eqref{eqn: new nonconf error mid proof}.
\end{proof}

The next corollary looks at how \eqref{eqn: lemma nonconformity error} simplifies when our VEM space is $C^0$ conforming.

\begin{corollary}
    Under the assumptions of Lemma \ref{lemma: nonconformity proof} and assuming that $\vemSpace \subset H^1_0(\Omega)$, it follows that 
    \begin{align}
        \mathcal{N}(u,w_h) = \sum_{ e \in \mathcal{E}_h } \int_e \kappa \big( \Delta u - \partial_{ss} u \big) [ \, \partial_n w_h \, ] \ \mathrm{d}s. 
        \label{eqn: lemma nonconformity H^1_0}
    \end{align}    
\end{corollary}

We now bound each term in \eqref{eqn: lemma nonconformity error} to achieve an error estimate for the nonconformity error. This essentially involves the jump properties of the VEM space \eqref{eqn: jump of wh} - \eqref{eqn: jump of tangential w_h} as well as standard interpolation estimates detailed in Theorem \ref{thm: interpolation estimates}. 

\begin{theorem}[Nonconformity error bound]\label{thm: non conformity error without (H4)}
    Let \ref{mesh} - \ref{assump: star shaped wrt a ball} hold. 
    Assume that the solution $u$ to \eqref{eqn: cts problem} satisfies $u \in H^{\polOrder+1}(\Omega)$ and assume that the coefficients satisfy $\kappa \in W^{\polOrder-1,\infty}(\Omega)$ and $\beta \in W^{\polOrder-2,\infty}(\Omega)$. 
    Then, for $\polOrder \geq 3$ the nonconformity error satisfies the following estimate
    \begin{align}
        |\mathcal{N}(u,w_h)| \leq 
        \begin{cases}
            C h^{l-1} \big(  \| \kappa \|_{W^{\polOrder-1,\infty}} ( |u|_{\polOrder+1} + |u|_{\polOrder} ) + \|\beta\|_{W^{\polOrder-2,\infty}} |u|_{\polOrder-1} )| w_h|_{2,h}, 
            \quad &\text{ if } \vemSpace \not \subset H^1_0(\Omega),
            \\
            C h^{\polOrder -1} \| \kappa \|_{W^{\polOrder-1,\infty}} |u|_{\polOrder+1} |w_h|_{2,h}, \quad &\text{ if } \vemSpace \subset H^1_0(\Omega),
        \end{cases}
        \label{eqn: nonconformity error bound}
    \end{align}
    for any $w_h \in \vemSpace$, for a constant $C$ independent of $h$.
    For $\polOrder=2$, the nonconformity error satisfies 
    \begin{align}\label{eqn: l=2 nonconformity error bound}
         | \mathcal{N}(u,w_h)| \leq 
        \begin{cases}
            C h \big( ( \|\kappa\|_{W^{1,\infty}} |u|_{3}+ \| \beta\|_{L^{\infty}} |u|_{1} ) |w_h|_{2,h} 
            + (\|\gamma\|_{L^{\infty}} |u|_0  + |f|_0 ) |w_h|_{1,h} \big), 
            \quad & \text{ if } \vemSpace \not \subset H^1_0(\Omega),
            \\
            C h \|\kappa\|_{W^{1,\infty}} |u|_{3} |w_h|_{2,h}, \quad &\text{ if } \vemSpace \subset H^1_0(\Omega),
        \end{cases}
    \end{align}
    for a constant $C$ independent of $h$.
\end{theorem}
\begin{proof}
    Firstly, consider the case $\polOrder \geq 3$. Then since $u \in H^4(\Omega)$, using Lemma \ref{lemma: nonconformity proof}, it holds that 
    \begin{align*}
        \mathcal{N}(u,w_h) =&    
        \sum_{e \in \mathcal{E}_h} \int_e \kappa (\Delta u - \partial_{ss} u) [\ \partial_n w_h \ ] + \sum_{e \in \mathcal{E}_h} \int_e \kappa \partial_{ns} u [ \ \partial_s w_h \ ]  
        \\
        &+ 
        \sum_{e \in \mathcal{E}_h} \int_e \big( \partial_{ss} u \partial_n \kappa - \partial_{ns} u \partial_s \kappa - \partial_n (\kappa \Delta u) + \beta \partial_n u \big) [ \ w_h \ ] 
        =: 
        I_1 + I_2 + I_3.
    \end{align*}
    For $I_1$, we apply \eqref{eqn: jump of normal wh}, to see that,
    \begin{align*}
        I_1 &\leq 
        \Big| \sum_{e \in \mathcal{E}_h} \int_e \big( \kappa (\Delta u - \partial_{ss} u ) - \mathcal{P}^{\polOrder-2}_e \big( \kappa (\Delta u -\partial_{ss} u) \big) \big) \big( [\ \partial_n w_h \ ] - \mathcal{P}_e^{0} [ \ \partial_n w_h \ ] \big) \Big| \\
        &\leq \sum_{e \in \mathcal{E}_h} \big\| \kappa ( \Delta u - \partial_{ss} u ) - \mathcal{P}^{\polOrder-2}_e \big( \kappa (\Delta u - \partial_{ss} u) \big) \big\|_{0,e} \big\| [ \ \partial_n w_h \ ] - \mathcal{P}_e^{0} [ \ \partial_n w_h \ ] \big\|_{0,e}.
    \end{align*}
    Where we have used the properties of the $L^2$ projection, as well as Cauchy-Schwarz in the last step.
    Using the estimates
    \begin{align*}
        \| \kappa ( \Delta u - \partial_{ss} u ) - \mathcal{P}^{\polOrder-2}_e \big( \kappa (\Delta u - \partial_{ss} u) \big) \|_{0,e} 
        &\leq C h^{\polOrder-2+1-1/2} |\kappa ( \Delta u - \partial_{ss}  u)|_{\polOrder-1} ,  \\
        \| [ \ \partial_n w_h \ ] - \mathcal{P}^0_e [ \ \partial_n w_h \ ] \|_{0,e} &\leq C h^{1-1/2} | \partial_n w_h |_{1,h}.
    \end{align*}
    It therefore follows that,
    \begin{align}
        I_1 \leq \Big|  \sum_e \int_e \kappa( \Delta u - \partial_{ss} u ) [ \ \partial_n w_h \ ] \mathrm{d}s \Big| 
        \leq 
        C h^{\polOrder -1} \|\kappa\|_{W^{\polOrder-1,\infty}} |u|_{\polOrder+1} |w_h|_{2,h}.
        \label{eqn: normal jump of w_h term in error anal}
    \end{align}
    \red{Notice that \eqref{eqn: normal jump of w_h term in error anal} gives the desired result when $\vemSpace \subset H^1_0(\Omega)$, as required.}

    For the term $I_2$, we apply \eqref{eqn: jump of tangential w_h}, introduce the polynomial $\cP_e^{\polOrder-2}(\kappa \partial_{ns} u)$ and use standard interpolation estimates to get, 
    \begin{align}
        I_2 
        \leq 
        \Big| \sum_{e \in \mathcal{E}_h} \int_e  \big(\kappa \partial_{ns} u - \mathcal{P}_e^{\polOrder -2} (\kappa \partial_{ns} u) \big) \big([ \, \partial_s w_h \, ] - \mathcal{P}_e^0[ \, \partial_s w_h \, ] \big) \mathrm{d}s \Big| 
        \leq 
        C h^{\polOrder -1} \|\kappa\|_{W^{\polOrder-1,\infty}}  |u|_{\polOrder+1} |w_h|_{2,h}. 
        \label{eqn: tangential jump of wh in error analysis}
    \end{align} 
    Finally, consider the term $I_3$. For ease of notation, let us set 
    $
    u^{*} := \partial_{ss} u \partial_n \kappa - \partial_{ns} u \partial_s \kappa - \partial_n (\kappa \Delta u) + \beta \partial_n u
    $
    and note that 
    \begin{align*}
        |u^{*}|_{\polOrder-2}
        \leq
        \| \kappa \|_{W^{\polOrder-1,\infty}} \big( |u|_{\polOrder}
        + |u|_{\polOrder+1} \big)
        +
        \| \beta \|_{W^{\polOrder-2,\infty}} |u|_{\polOrder-1}. 
    \end{align*}
    For this term, $I_3$, we follow the approach taken in \cite{zhao_morley-type_2018} and introduce the interpolation of $w_h$ into the lowest order conforming VEM space $\Pi^1 w_h \in H^1_0(\Omega)$. 
    \begin{align*}
        I_3 \leq 
        \Big|  \sum_{e \in \mathcal{E}_h} \int_e ( u^{*} - \mathcal{P}^{\polOrder-3}_e(u^{*})) [\, w_h - \Pi^1 w_h \, ] \Big| 
        \leq \sum_{e \in \mathcal{E}_h} \big\| u^{*} - \mathcal{P}^{\polOrder-3}_e ( u^{*}) \|_{0,e} \big\| [ \, w_h - \Pi^1 w_h \, ] \big\|_{0,e} .
    \end{align*}
    Therefore,
    \begin{align}
        I_3 \leq \Big| \sum_{e \in \mathcal{E}_h} \int_e \big( \partial_{ss} u \partial_n \kappa - \partial_{ns} u \partial_s \kappa  - \partial_n (\kappa \Delta u) - \beta \partial_n u  \big) [ \,  w_h \, ] \, \mathrm{d}s \Big| 
        \leq
        C h^{\polOrder-1} |u^{*}|_{\polOrder-2} |w_h|_{2,h}. 
        \label{eqn: jump of wh in error anal}
    \end{align}    
    Hence, when $\polOrder \geq 3$, combining \eqref{eqn: normal jump of w_h term in error anal}, \eqref{eqn: tangential jump of wh in error analysis}, and \eqref{eqn: jump of wh in error anal}, the result \eqref{eqn: nonconformity error bound} follows.

    Now consider the case $\polOrder=2$. \strike{Assume that $u$ solves problem }. For a test function $v \in H^1_0(\Omega)$ (and for $u \in H^3(\Omega)$) it holds that 
    \begin{align}\label{eqn: cts test function}
        (f,v) = 
        \sum_{K \in \cT_h} \Big\{
        - \int_{K} \sum_{i,j=1}^2 \partial_j (\kappa \partial_{ij} u) \partial_i  v
        + \int_K 
        \beta Du \cdot D v
        + \int_K \gamma uv
        \Big\}.
    \end{align}
    Using \eqref{eqn: hessian term nonconf error} to express the hessian term in the bilinear form, it follows that 
    \begin{align*}
        \mathcal{N}(u,w_h) = \ & a(u,w_h) - (f,w_h) 
        \\
        = \ & \Big(
        \sum_{K \in \cT_h} \Big\{ - \int_K \sum_{i,j=1}^2 \partial_j (\kappa \partial_{ij} u) \partial_i w_h
        + \int_K \beta Du \cdot D w_h + \int_K \gamma uw_h \Big\}
        - (f,w_h)
        \Big)
        \\
        &+
        \sum_{K \in \cT_h}
        \int_{\partial K} \sum_{i,j=1}^2 \kappa \partial_{ij} u \partial_i w_h n_j 
        \ =: \ E_1 + E_2.
    \end{align*}
    Using arguments as in the proof of Lemma \ref{lemma: nonconformity proof} it follows that $E_2 = I_1 + I_2$. These terms can be bounded as in \eqref{eqn: normal jump of w_h term in error anal} and \eqref{eqn: tangential jump of wh in error analysis} respectively. Therefore, it follows that
    \begin{align*}
        E_2 \leq C h \| \kappa\|_{W^{1,\infty}} |u|_{3} |w_h|_{2,h}.
    \end{align*} 
    Note that when the VEM space is conforming, $\vemSpace \subset H^1_0(\Omega)$, it follows that $w_h$ is a viable test function in $H^1_0(\Omega)$ and therefore we can take $v=w_h$ in \eqref{eqn: cts test function}.  
    Hence, the nonconformity error reduces to the following 
    \begin{align*}
        \mathcal{N}(u,w_h) = E_2 \leq C h \| \kappa \|_{W^{1,\infty}} |u|_3 |w_h|_{2,h}.
    \end{align*} 
    To treat $E_1$ in the case that $\vemSpace \not \subset H^1_0(\Omega)$ we now introduce the interpolation of $w_h$ into the lowest order conforming VEM space, $\Pi^{1} w_h \in H^1_0(\Omega)$. Then, it holds that 
    \begin{align*}
        E_1 = \ &
        \sum_{K \in \cT_h} \big\{ 
        \int_K 
        \mathrm{div} (\kappa D^2 u) \cdot D(\Pi^{1} w_h - w_h)
        + \int_K \beta Du \cdot D (w_h - \Pi^{1} w_h) + \int_K \gamma u (w_h - \Pi^{1} w_h)
        \big\}  
        \\
        &
        +
        (f,\Pi^{1} w_h-w_h).
    \end{align*}
    Using standard estimates, we see that 
    \begin{align*}
        \sum_{K \in \cT_h} 
        \int_K 
        \mathrm{div} (\kappa D^2 u) \cdot D(\Pi^{1} w_h - w_h) 
        &\leq
        C h \| \kappa\|_{W^{1,\infty}} |u|_{3} |w_h|_{2,h},
        \\
        \sum_{K \in \cT_h} \int_K \beta Du \cdot D (w_h - \Pi^{1} w_h)
        &\leq 
        C h \| \beta \|_{L^{\infty}} |u|_1 |w_h|_{2,h},
        \\
        \sum_{K \in \cT_h} \int_K \gamma u (w_h -  \Pi^{1} w_h) 
        &\leq C h \| \gamma \|_{L^{\infty}} |u|_{0} |w_h|_{1,h},
    \intertext{and finally,}
        (f,\Pi^{1} w_h-w_h) &\leq C h |f|_0 |w_h|_{1,h}.
    \end{align*}
    This concludes the proof.
\end{proof}

\subsection{Energy norm estimate}
Now that we have successfully bounded the nonconformity error, we look at the final term in Theorem \ref{thm: a-priori error bound energy norm} and prove convergence in the energy norm.  

\red{ To ease the presentation of the next Theorems and Corollary, we observe that the following inequality holds
\begin{align}\label{eqn: norm equiv}
    |w_h|_{s,h} \leq C \eta_s \vertiii{w_h}_h
\end{align}
for each $s=0,1,2$, where  $\eta_2 = \frac{1}{\sqrt{\kappa_0}}$, $\eta_1 = \frac{1}{\max{( \sqrt{\kappa_0} , \sqrt{\beta_0} )} }  $, and $\eta_0 = \frac{1}{\max{( \sqrt{\kappa_0} , \sqrt{\beta_0}, \sqrt{\gamma_0} )} } $. Notice that the inequality in \eqref{eqn: norm equiv} holds due to standard discrete Poincar\'{e} inequalities, and for example can be shown similarly as in \cite{antonietti_fully_2018}.  
}

\begin{theorem}\label{thm: variational crime}
    Assume that \ref{mesh} - \ref{assump: star shaped wrt a ball} hold, defined in Assumptions \ref{assumptions: discrete assumptions} and \ref{assumption: mesh regularity}.  
    Let $u \in H^{\polOrder+1}(\Omega)$ be the solution to continuous problem \eqref{eqn: cts problem}.
    Assume that the coefficients satisfy $\kappa \in W^{\polOrder-1,\infty}(\Omega),\beta \in W^{\polOrder,\infty}(\Omega)$, and $\gamma \in W^{\polOrder+1,\infty}(\Omega)$.
    Then it holds that
    \begin{align}\label{eqn: var crime}
            \inf_{p \in \prob_{\polOrder}(\cT_h)}  \bigg[& \vertiii{u-p}_h + \sum_{K \in \cT_h} \sup_{w_h \in \vemSpace^K} \frac{|a^K(p,w_h)-a_h^K(p,w_h)|}{ \vertiii{w_h}_K} \bigg]
            \leq \ 
            C \big( h^{\polOrder -1}c_2 + h^{ \polOrder } c_1 
            + h^{\polOrder+1} c_0 
            \big) |u|_{\polOrder+1}
    \end{align}
    for a constant $C$ independent of $h$.
    We define the remaining constants $c_0,c_1,$ and $c_2$ as follows; let 
    $c_2 = \|\kappa\|_{L^{\infty}}^{\frac{1}{2}} + \eta_2 \| \kappa \|_{W^{\polOrder-1,\infty}}$,
    $c_1 = \|\beta\|_{L^{\infty}}^{\frac{1}{2}} + \eta_1 \| \beta\|_{W^{\polOrder,\infty}}$, and $c_0 = \|\gamma\|_{L^{\infty}}^{\frac{1}{2}} +\eta_0
    \| \gamma \|_{W^{\polOrder+1,\infty}} $.
\end{theorem}
\begin{proof}
    To show that \eqref{eqn: var crime} holds, we \red{note} \strike{have} that
    \begin{align*}
        \inf_{p \in \prob_{\polOrder}(\cT_h)}  \bigg[& \vertiii{u-p}_h + \sum_{K \in \cT_h} \sup_{w_h \in \vemSpace^K} \frac{|a^K(p,w_h)-a_h^K(p,w_h)|}{ \vertiii{w_h}_K} \bigg] \\
        \leq& \, \vertiii{ u - \mathcal{P}^{\polOrder} u }_{h} + 
        \sum_{K \in \cT_h} \sup_{w_h \in \vemSpace^K} \frac{|a^K(\mathcal{P}_K^{\polOrder} u,w_h)-a_h^K( \mathcal{P}_K^{\polOrder} u,w_h)|}{\vertiii{w_h}_{K}}. 
    \end{align*}
    Note that the first term is bounded easily by standard interpolation estimates.
    For the other term, we \red{recall} \strike{use the property that $\valueProj p = p$, see} Remark \ref{remark: polynomial consistency}, which implies that the stabilization part of the discrete bilinear form vanishes.  
    Therefore,
    \begin{align*}
        |a^K(\mathcal{P}^{\polOrder}_K u,w_h)-a_h^K( \mathcal{P}_K^{\polOrder} u,w_h)| 
        \leq \Big| &\int_K \kappa D^2 (\mathcal{P}_K^{\polOrder} u) : D^2 w_h - \int_K \kappa \hessProj  (\mathcal{P}_K^{\polOrder} u) : \hessProj w_h \Big|  \\
        + \; \Big| &\int_K \beta D(\mathcal{P}_K^{\polOrder} u) \cdot D w_h - \int_K \beta \gradProj  (\mathcal{P}_K^{\polOrder} u) \cdot \gradProj w_h \Big| \\
        + \; \Big| &\int_K \gamma (\mathcal{P}_K^{\polOrder} u )w_h - \int_K \gamma \valueProj  (\mathcal{P}_K^{\polOrder} u) \valueProj w_h \Big|  \\
        =: & \ T_2 + T_1 + T_0.
    \end{align*} 
    Then, denoting the coefficients as $\alpha_2 := \kappa, \alpha_1 := \beta$ and $\alpha_0 := \gamma$, we can show that  
    \begin{align}\label{eqn: Ts estimate}
        T_s \leq C h_K^{\polOrder -s+1} \| \alpha_s \|_{W^{\polOrder-s+1,\infty}} |u|_{\polOrder +1} |w_h|_{s,K}.
    \end{align}
    To see that \eqref{eqn: Ts estimate} holds, we use the results of Lemma \ref{lemma: l2 projections assumptions satisfied} to express our projections as $L^2$ projections. Hence, it follows that 
    \begin{align*}
        T_s = \ & \bigg| \int_K \alpha_s \big( D^s (\cP^{\polOrder}_K u) : D^s w_h - \cP^{\polOrder-s}_K D^s ( \cP^{\polOrder}_K u) : \cP^{\polOrder-s}_K D^s w_h \big) \bigg| 
        = \bigg| \int_K \alpha_s D^s \cP^{\polOrder}_K u : (I - \cP^{\polOrder-s}_K)D^sw_h \bigg|  
        \\
        \leq \ &
        \| ( I - \cP^{\polOrder-s}_K ) \alpha_s D^s \cP_K^{\polOrder} u \|_{0,K} \| D^s w_h \|_{0,K} 
        \leq C h_K^{\polOrder-s+1} 
        \| \alpha_s \|_{W^{\polOrder-s+1,\infty}}
        | u |_{\polOrder+1} |w_h|_{s,K}.
    \end{align*}
    The result now follows.
\end{proof}

We now have the following convergence theorem which is a result of Theorems \ref{thm: interpolation error estimate}, \ref{thm: non conformity error without (H4)}, \ref{thm: variational crime} and Lemma \ref{lemma: load term estimate}. 
\begin{theorem}[Convergence in the energy norm]\label{thm: energy norm}
    Assume that \ref{mesh} - \ref{assump: star shaped wrt a ball} hold, defined in Assumptions \ref{assumptions: discrete assumptions} and \ref{assumption: mesh regularity}. 
    \strike{Let} \red{Assume that} $u \in H^{\polOrder+1}(\Omega)$ \strike{be} \red{is} the solution to the continuous problem \eqref{eqn: cts problem} and suppose that $u_h \in \vemSpace$ is the solution to the discrete problem \eqref{eqn: discrete problem}. 
    Assume that the coefficients satisfy $\kappa \in W^{\polOrder-1,\infty}(\Omega),\beta \in W^{\polOrder,\infty}(\Omega)$, and $\gamma \in W^{\polOrder+1,\infty}(\Omega)$.
    Let $f \in H^{s}(\Omega)$ and define $r:=\min(s-1,\polOrder)$.
    Then, under these assumptions there exists a constant $C$ independent of $h$ such that 
    \begin{align}\label{eqn: error energy norm}
        \vertiii{u - u_h}_h 
        \leq 
        C \big\{ h^{l-1} \big( c_2 |u|_{\polOrder+1}
        +
        c_3 |u|_{\polOrder} +  c_4 |u|_{\polOrder-1} \big)  
        +  h^{ \polOrder } c_1 |u|_{\polOrder+1}
        +
        h^{\polOrder+1} c_0 |u|_{\polOrder+1}      
        +
        h^{r+2}  \eta_1 |f|_{r+1} \big\}.
    \end{align}
    Recall that the constants $c_0,c_1,$ and $c_2$ are defined in Theorem \ref{thm: variational crime} and define $c_3 = \eta_2  \| \kappa \|_{W^{\polOrder-1,\infty}}$ and $c_4= \eta_2  \| \beta \|_{W^{\polOrder-2,\infty}}$.

    If $\vemSpace \subset H^1_0(\Omega)$ then it follows that 
    \begin{align}\label{eqn: error energy norm cts}
        \vertiii{u - u_h}_h 
        \leq 
        C \big\{ \big( h^{l-1} c_2 
        +  h^{ \polOrder } c_1 
        +
        h^{\polOrder+1} c_0 \big) |u|_{\polOrder+1}      
        +
        h^{r+2}  \eta_1 |f|_{r+1} \big\}.
    \end{align}
\end{theorem}

%% file: Sections/perturbationModifiedProblem.tex
\section{Perturbation problem} \label{section: perturbation problem}
We turn our attention to the following fourth order perturbation problem. 
For a polygonal domain $\Omega \subset \R^2$ the perturbation problem reads as follows
\begin{equation}\label{eqn: perturbation problem}
    \begin{split}    
            \epsilon^2 \Delta^2 u - \Delta u &=f, \quad \text{ in } \Omega, \\
            u = \partial_n u &= 0, \quad \text{ on } \partial \Omega.
    \end{split} 
\end{equation}
We make the minimal assumptions that $f \in L^2(\Omega)$ and $\epsilon \in \R$ such that $0 < \epsilon \leq 1$. Taking $\kappa(x) = \epsilon^2$, $\beta(x) = 1$, and $\gamma(x) = 0$ we can examine the error analysis from the previous section, with the energy norm now becoming $ \vertiii{ v }_{h}^2 = \epsilon^2 |v|^2_{2,h} + |v|^2_{1,h}.$

It is well known that for example the lowest order $C^1$ nonconforming
space on triangles (the Morley element, \cite{morley1967triangular}) does not lead to a scheme that is
robust with respect to $\epsilon\to 0$ (see for example \cite{wang2001necessity,nilssen_robust_2000}). 
There have been a range of modifications suggested  to the original Morley element, for example in \cite{wang_uniformly_2011,wang2006modified,zhang_nonconforming_2020}.
In \cite{nilssen_robust_2000}
a modification is suggested which on triangles corresponds to our $\CzeroconfSpace$ conforming
space in the lowest order setting and convergence of the method is proven in
this case. 
We give error estimates for the higher order version of
those two spaces in the following. In addition we study a new modified $C^1$ nonconforming
discretization, $\Conemod$, which has the same degrees of freedom as the original $C^1$
nonconforming space but is stable with respect to the perturbation paramter $\epsilon$. This is
achieved by a modification to the gradient projection, $\gradProj$, given next.

\begin{definition}\label{defn: modified grad proj}
    We define the \emph{modified gradient projection} to be the following 
    \begin{align}\label{eqn: modified grad proj}
        \int_K \gradProj v_h p = - \int_K \valueProj v_h \nabla p + \sum_{e \subset \partial K } \int_e \Confinterpolation v_h p n, \quad \forall \, p \in \prob_{\polOrder-1 }(K)^2
    \end{align}
    for any $v_h \in \vemSpace^K$.
We denote with $\Confinterpolation$ the interpolation into a $H^1$ conforming VEM space of order $\polOrder-1$.
We use the ``lazy'' version of the serendipity spaces discussed in \cite{da_veiga_serendipity_2015}. The dofs for this space were mentioned in Remark \ref{rmk: serendipity dofs} and for order $\polOrder-1$ are described by the dof tuple $(0,-1,\polOrder-3,-1,\polOrder-4)$. Therefore the dofs are a subset of the dofs defining the nonconforming $C^1$ space. The vertex values and the $\polOrder-3$ moments on the edges uniquely define $\Confinterpolation v|_e\in \prob_{\polOrder-1}(e)$ so that the gradient projection given above is computable using the dofs for the  $C^{1}$ nonconforming space.
\end{definition}

\ifthenelse{\boolean{thesis}}
{
    \begin{remark}
        This change to the gradient projection could also be achieved by replacing the edge projection $\edgeProj v_h$ with $(\Confinterpolation v_h)|_e$ which is dof compatible (see Definition \ref{defn: dof compatible}) with the exception that $\edgeProj q = q|_e$ only holds for $q\in\prob_{\polOrder-1}(K)$. The resulting virtual element space satisfies $\vemSpace\subset H^1_0 (\Omega)$ but only $\prob_{\polOrder-1} (\cT_{h})\subset\vemSpace$. The $L^2$ projection properties of the value and the hessian projection are still satisfied (in fact they are both not changed by the use of $\Confinterpolation$ for the edge projection). But the property given in Lemma \ref{lemma: l2 projections assumptions satisfied} does not hold anymore for the gradient projection. While, due to the continuity of the discrete function space, we could use many of the bounds from the previous section avoiding the scaling with $\varepsilon^{-1}$, we cannot use all of the results due to the missing polynomial exactness. Therefore, we briefly describe a convergence proof based on using the original $C^1$ nonconforming VEM space with the modified gradient projection given above.
    \end{remark}
}
{
}

The following property is now obtained for the modified gradient projection. 
\begin{lemma}\label{lemma: H4 L2 property}
    For the modified gradient projection detailed in Definition \ref{defn: modified grad proj} it holds that
    \begin{align*}
        \gradProj v_h = \cP^{\polOrder-2}_K (D \Confinterpolation v_h) \quad \text{ for any } v_h \in \localVEMSpace.
    \end{align*}
\end{lemma}
\red{Our discrete bilinear form $a_h(\cdot,\cdot)$ is now chosen as in Definition \ref{eqn: discrete bilinear form} with coefficients $\kappa(x) = \epsilon^2$, $\beta(x) =1$, and  $\gamma(x) =0$ but using the modified gradient projection operator as detailed in Definition \ref{defn: modified grad proj}.
So for any ${u_h, v_h \in \localVEMSpace}$, define the \emph{local discrete bilinear form} $a_h^K $ as 
    \begin{equation} \label{eqn: modified local form}
        \begin{split}
            a_h^K (u_h,v_h) := \epsilon^2 &\int_K \hessProj u_h : \hessProj v_h + \int_K \gradProj u_h \cdot \gradProj v_h + S^K ( u_h - \valueProj u_h , v_h - \valueProj v_h ).
        \end{split}    
    \end{equation}
where we choose the same stabilization term $S^K$ presented previously.}

To prove convergence in the energy norm for this modified scheme we use the ideas seen in \cite{wang2006modified,wang_robust_nodate}. We consider a modified bilinear form \red{ $b(\cdot,\cdot)$ } by changing the lower order contribution. 
\begin{align*}
    b(w,v) := \epsilon^2 \int_{\Omega} D^2 w : D^2 v + \int_{\Omega} D \Confinterpolation w \cdot D \Confinterpolation v 
\end{align*}
for all $w,v \in H^2_0(\Omega)$. Using $b(\cdot,\cdot)$ we can prove a 
a Strang-type lemma similar to Theorem \ref{thm: a-priori error bound energy norm}.
\begin{theorem}[Abstract a priori error bound]
    Let $\polOrder \geq 1$ be an integer. \strike{Under Assumption} \red{Let $u$ be the solution to problem \eqref{eqn: cts problem} with coefficients $\kappa(x)=\epsilon^2$, $\beta(x)=1$, and $\gamma(x)=0$. Suppose that $u_h \in \vemSpace$ is the solution to \eqref{eqn: discrete problem} using the modified local bilinear form \eqref{eqn: modified local form}}. \red{Under Assumptions \ref{assumptions: discrete assumptions},} it holds that
    \begin{equation}\label{eqn: strang H4}
        \begin{split}
            \vertiii{u-u_h}_{h} \leq \ &
            \inf_{v_h \in \vemSpace} \vertiii{u-v_h}_{h} + \sup_{\substack{w_h \in \vemSpace \\ w_h \neq 0}} \frac{|\langle f_h,  w_h \rangle - (f,\Confinterpolation w_h)|}{\vertiii{w_h}_{h}} 
            + \sup_{\substack{w_h \in \vemSpace \\ w_h \neq 0}} \frac{|b(u,w_h)-(f,\Confinterpolation w_h)|}{\vertiii{w_h}_{h}}
            \\
            &+ \inf_{p \in \prob_{\polOrder}(\cT_h)} \Big[ \vertiii{u-p}_{h} + \sum_{K \in \cT_h} \sup_{\substack{w_h \in \vemSpace^K \\ w_h \neq 0}} \frac{|b^K(p,w_h) - a_h^K(p,w_h)|}{\vertiii{w_h}_{K}} \Big].
        \end{split}
    \end{equation}
\end{theorem}

As in the previous section we can bound each term in the above abstract estimate to obtain an error estimate in powers of $h$ and $\varepsilon$.

\begin{theorem}\label{thm: H4 energy norm convergence}
    Assume that \ref{mesh} - \ref{assump: star shaped wrt a ball} hold, defined in Assumptions \ref{assumptions: discrete assumptions} and \ref{assumption: mesh regularity}. 
    Let $\polOrder \geq 2$ and \red{suppose that} \strike{let} $u \in H^{\polOrder+1}(\Omega)$ \red{is} \strike{be} the solution to the continuous problem \eqref{eqn: cts problem} \red{with coefficients $\kappa(x)=\epsilon^2$, $\beta(x)=1$, and $\gamma(x)=0$.} \strike{and.} Suppose that $u_h \in \vemSpace$ is the solution to the discrete problem \eqref{eqn: discrete problem} using the modified gradient projection (Definition \ref{defn: modified grad proj}) in the discrete bilinear form. 
    Assume that $f \in H^{\polOrder-1}(\Omega)$. 
    Then, under these assumptions it follows that
    \begin{align*}
        \vertiii{u-u_h}_{h} \leq C \big\{ \epsilon h^{\polOrder-1}|u|_{\polOrder+1} + h^{\polOrder-1}|u|_{\polOrder} + h^{\polOrder-1} (|f|_{\polOrder-1} + |f|_{\polOrder-2}) \big\}
    \end{align*}
    for a constant $C$ independent of $h$.
\end{theorem}

\begin{proof}    
The proof follows similar arguments used in Section \ref{section: error analysis} and therefore we keep the presentation here brief.

It is straightforward to show the following two bounds
\begin{align*}
    \inf_{v_h \in \vemSpace} \vertiii{u-v_h}_{h} \leq \vertiii{u-\cP^{\polOrder}u}_{h} &\leq C h^{\polOrder-1} (\epsilon |u|_{\polOrder+1} + |u|_{\polOrder})
    \\
    |\langle f_h, w_h \rangle - (f,\Confinterpolation w_h) | 
    &\leq C h^{\polOrder-1} (|f|_{\polOrder-1} + |f|_{\polOrder-2})\vertiii{w_h}_h.
\end{align*}
We focus on the remaining two terms involving the modified bilinear form $b(\cdot,\cdot)$. 

For the nonconformity type error term, it holds that
\begin{align*}
    |b(u,w_h) - (f, \Confinterpolation w_h) | \leq C ( \epsilon h^{\polOrder-1}|u |_{\polOrder+1} + h^{\polOrder-1} |u|_{\polOrder} )\vertiii{w_h}_{h}.
\end{align*}
We show this by multiplying the strong equation \eqref{eqn: perturbation problem} by $\Confinterpolation w_h\in H^1_0(\Omega)$ and integrating over $\Omega$ 
\begin{align*}
    (f,\Confinterpolation w_h) = \epsilon^2 \int_{\Omega} \Delta^2 u \Confinterpolation w_h - \int_{\Omega} \Delta u \Confinterpolation w_h 
       = \int_{\Omega} \big( - \epsilon^2 D(\Delta u) + Du \big) \cdot D \Confinterpolation w_h
\end{align*}
and similarly
\begin{align*}
    b(u,w_h) 
    &= -\sum_K\Big( \epsilon^2 \int_K D (\Delta u)\cdot D w_h - \epsilon^2 \int_{\partial K} \big( \Delta u - \partial_{ss} u \big) \partial_n w_h + \partial_{ns} u \partial_s w_h
       - \int_K D \Confinterpolation u\cdot D \Confinterpolation w_h
       \Big).
\end{align*}
Therefore, it holds that
\begin{align*}
    b(u,w_h) - (f,\Confinterpolation w_h) 
    = \ &
    \epsilon^2 \sum_{K \in \cT_h} \int_K D(\Delta u) \cdot D(\Confinterpolation w_h - w_h) 
    + \epsilon^2 \sum_{K \in \cT_h} \int_{\partial K} (\Delta u - \partial_{ss} u) \partial_n w_h + \partial_{ns} u \partial_s w_h
    \\
    &+
    \sum_{K \in \cT_h} \int_K D (\Confinterpolation u - u) \cdot D \Confinterpolation w_h  
    \\
    =: \ & J_1 + J_2 + J_3. 
\end{align*}
Notice that $J_2$ can be bounded as before in the nonconformity error proof (Theorem \ref{thm: non conformity error without (H4)}). Then for $J_1$ using Cauchy-Schwarz and standard estimates, 
\begin{align*}
    J_1 \leq \ & \epsilon^2 \sum_{K \in \cT_h} \| D(\Delta u) -\cP^{\polOrder-3}_K(D\Delta u) \|_{0,K} |\Confinterpolation w_h - w_h |_{1,K} 
    \leq C \epsilon h^{\polOrder-1} |u|_{\polOrder+1} \vertiii{w_h}_{h}.
\end{align*}
Now, for $J_3$ we bound this term using the optimal interpolation properties of the lower order VEM space as well as stability of the interpolation operator. 
\begin{align*}
    J_3 
    \leq \sum_{K \in \cT_h} | u - \Confinterpolation u |_{1,K} |\Confinterpolation w_h|_{1,K} 
    \leq C h^{\polOrder-1} |u|_{\polOrder} \vertiii{w_h}_{h}.
\end{align*}
Finally, for the last term in \eqref{eqn: strang H4} we can show that
\begin{align*}
    \inf_{p \in \prob_{\polOrder}(\cT_h)} \Big( \vertiii{u-p}_h + \sum_{K \in \cT_h} \sup_{\substack{w_h \in \vemSpace^K \\ w_h \neq 0}} \frac{|b^K(p,w_h)-a^K_h(p,w_h)|}{ \vertiii{w_h}_K } \Big)
    \leq C h^{\polOrder-1} (\epsilon |u|_{\polOrder+1} +|u|_{\polOrder}).
\end{align*}
This follows from
\begin{align*}
    b^K(\cP^{\polOrder}_K u,w_h) - a_h^K ( \cP^{\polOrder}_K u,w_h) = \ & \epsilon^2 \int_K D^2 \cP^{\polOrder}_K u : D^2 w_h 
    - \epsilon^2 \int_K \hessProj \cP^{\polOrder}_K u : \hessProj w_h 
    \\
    &+ \int_K D \Confinterpolation \cP^{\polOrder}_K u \cdot D \Confinterpolation w_h 
    - \int_K \gradProj \cP^{\polOrder}_K u \cdot \gradProj w_h.
\end{align*}
Due to the $L^2$ properties in Lemma \ref{lemma: l2 projections assumptions satisfied} the hessian terms cancel. 
Secondly, using the property in Lemma \ref{lemma: H4 L2 property} it holds that $\gradProj v_h = \cP^{\polOrder-2}_K (D \Confinterpolation v_h)$. Therefore for any $v_h \in \localVEMSpace$ it holds that
\begin{align*}
    \int_K \ & D \Confinterpolation \cP^{\polOrder}_K u \cdot D \Confinterpolation w_h - \int_K \gradProj \cP^{\polOrder}_K u \cdot \gradProj w_h
    \leq C h_K^{\polOrder-1}|u|_{\polOrder}\vertiii{w_h}_K.
\end{align*}
Therefore, the result follows from combining all of the intermediary results.
\end{proof}

\red{The next corollary details all error estimates for the following spaces applied to the perturbation problem \eqref{eqn: perturbation problem}. We analyse the fully nonconforming space, $\Conenonconf$, with dof tuple ${(0,-1,\polOrder-3,\polOrder-2,\polOrder-4)}$, the $\CzeroconfSpace$ conforming space with dof tuple $(0,-1,\polOrder-2,\polOrder-2,\polOrder-4)$ (see Example \ref{example: dof tuples}), and finally the new modified scheme $\Conemod$ with dof tuple  $(0,-1,\polOrder-3,\polOrder-2,\polOrder-4)$. For completeness, we repeat the result obtained in Theorem \ref{thm: H4 energy norm convergence} for the modified scheme $\Conemod$. Recall that this space has the same dof tuple as the $\Conenonconf$ space however we define the discrete bilinear form for this problem using the modified gradient projection (Definition \ref{defn: modified grad proj}).} 
\strike{Finally, we finish this section by presenting the results for all the considered spaces applied} \strike{to the perturbation problem}  
Due to keeping track of the coefficients in the previous error analysis section, the following corollary is simply a consequence of Theorem \ref{thm: energy norm} and Theorem \ref{thm: H4 energy norm convergence}.

\begin{corollary}\label{thm: convergence for perturbation problem}
    Under the same assumptions as Theorem \ref{thm: energy norm} with $f\in H^{\polOrder-2}(\Omega)$ the approximate solution in the $C^1$ nonconforming space $\Conenonconf$ \red{of order $\polOrder$} satisfies
    \begin{align*}
        \vertiii{u-u_h}_h 
        &\leq
        C \big\{ ( h^{\polOrder-1} \epsilon + h^{\polOrder} )|u|_{\polOrder+1} + \frac{h^{\polOrder-1}}{\epsilon}  |u|_{\polOrder-1} + h^{\polOrder-1} |f|_{\polOrder-2} \big\}.
    \intertext{The solution in the $\CzeroconfSpace$ conforming space, assuming that $f \in H^{\polOrder-1}(\Omega)$, satisfies the following error estimate}
        \vertiii{u-u_h}_h 
        &\leq 
        C \big\{ ( h^{\polOrder-1} \epsilon + h^{\polOrder} )|u|_{\polOrder+1} + h^{\polOrder} |f|_{\polOrder-1} \big\}.
        \intertext{Finally, for the modified $C^1$ nonconforming scheme defined by the modified gradient projection in Definition \ref{defn: modified grad proj} and assuming that $f \in H^{\polOrder-2}(\Omega)$, it holds that}
        \vertiii{u-u_h}_{h} &\leq C \big\{ h^{\polOrder-1} \epsilon |u|_{\polOrder+1} + h^{\polOrder-1}|u|_{\polOrder} + h^{\polOrder-1} (|f|_{\polOrder-1} + |f|_{\polOrder-2}) \big\}. 
    \end{align*}
    In all three cases, the constant $C$ remains independent of $h$ and $\epsilon$.
\end{corollary}
\red{
\begin{remark}
Note that for $\epsilon=O(1)$ all methods converge in the energy norm with order $l-1$ while for $\epsilon=O(h)$ the convergence rate of the original $\Conenonconf$ method reduces to $O(h^{l-2})$ while the modified scheme retains the order $h^{l-1}$. The $H^1$ conformity of the $\CzeroconfSpace$ space leads to an improvement of the order to $O(h^l)$ but requires more degrees of freedom.
\end{remark}
}

%% file: Sections/results.tex
\section{Numerical results}\label{section: numerical testing}
In this section we show some numerical results to verify the a priori bounds from the previous sections
for three virtual element spaces.
\begin{itemize}
\item $\Conenonconf$: the nonconforming space from \cite{antonietti_fully_2018,zhao_morley-type_2018} defined by the dof tuple $(0,-1,l-3,l-2,l-4)$ (see Example \ref{example: dof tuples}) and using the default projection operators (see Example \ref{ex: choice of value, edge, edge normal projections}).
\item $\Conemod$: the nonconforming space defined by the dof tuple $(0,-1,l-3,l-2,l-4)$ (see Example \ref{example: dof tuples}) and using the modified gradient projection (see Section \ref{section: perturbation problem}).
\item $\CzeroconfSpace$: the continuous space from \cite{zhao_nonconforming_2016,zhang_nonconforming_2020} defined by the dof tuple $(0,-1,l-2,l-2,l-4)$ (see Example \ref{example: dof tuples}) and using the default projection operators (see Example \ref{ex: choice of value, edge, edge normal projections}).
\end{itemize}
We focus on results for $l=2,3,4$ on both structured triangular grids and to demonstrate the flexibility of the virtual element method, on grids consisting of \red{Voronoi cells.} \red{Figure~\ref{fig: Voronoi grids} contains the images of the first few mesh refinements for this choice of grid.} 
\red{Grid data for all considered spaces and mesh refinements is shown in Table~\ref{table: grid data}.}
\strike{remapped hexagonal elements}
\strike{also used in}

We first study \red{convergence of the methods for a linear problem of the general form \eqref{eqn: cts problem} with varying coefficients.} As a second example we study the simple \strike{linear} \red{perturbation} problem \eqref{eqn: perturbation problem} with constant coefficients, and we show results with varying $h,\epsilon$.
 \red{ For both examples, we solve the problem on the domain $\Omega= (0,1)
 \times (0,1)$. The approximation errors are measured in the relative energy norm. }

Finally we highlight how our approach based on general projection operators allows us to handle nonlinear problems by showing results for the Cahn-Hilliard equation and the Willmore flow of graphs. \strike{and the Cahn-Hilliard equation.}

The code used to perform the simulations is based on the DUNE software framework \cite{dunegridpaperII}. We implemented our VEM approach within the module DUNE-FEM \cite{dedner2010generic}. This is an extension module for DUNE that provides interfaces for the implementation of general grid based numerical schemes on general unstructured grids. It is open source software implemented in \texttt{C++}.
\red{In our implementation we construct quadrature rules by subdividing polygons into triangles and applying a quadrature of sufficient order on each triangle.
The required order of the quadrature is the same as required by standard finite elements. The consistency error due to quadrature was studied in \cite{cangiani_conforming_2015} for the second order method and similar arguments can be applied here. Other possibly more efficient methods for constructing quadrature rules are available for general polygons, for example, \cite{antonietti2018fast}.} 
Recently a Python based frontend was added to DUNE \cite{dedner_dune_2018,dednerPythonBindings2020}.
The domain specific language UFL \cite{alnaes_unified_2012} can be used to describe mathematical models. A detailed tutorial including some VEM examples, e.g., for linear elasticity and also the Willmore flow example described here, showcases the flexibility of the approach \cite{dednerPythonBindings2020}. 

\input{Sections/meshdata}

\subsection{Linear varying coefficient problem}

\red{We start by studying the more general linear fourth order problem with varying coefficients.}

\begin{example}\label{ex: varying coefficients}
    Consider problem \eqref{eqn: cts problem} using 
    \begin{equation*}
      \kappa(x,y) = \frac{1}{1+x^2+y^2}, \quad  \beta(x,y)=e^{-xy}, \quad \gamma(x,y)=\big(\sin(x^2+y^2)\big)^2.
    \end{equation*} 
    We \strike{again} choose the forcing so that \red{$u(x,y) = (\sin(2\pi x)\sin(2\pi y))^2$ } is the exact solution on the domain \red{$\Omega$ =}$(0,1)^2$. \strike{Overall the results show the same picture as we already saw for the simple constant coefficient setting with $\epsilon=1$} 
    \strike{so we only show results here for the simplex grid in Tables} 
    \red{ We show results on both the structured simplex grid and the Voronoi cells for polynomial orders $\polOrder = 2,3,4$ in Table~\ref{table: varying coeff}.} 
    All methods converge with the expected order of $\polOrder-1$ and produce very similar errors on a given grid with the $\CzeroconfSpace$ space requiring more degrees of freedom.
\end{example}

\ifthenelse{\boolean{oldResultsDisplay}}
{
  {\tiny
  \input{results_simplex/problemvarying-coefficients_l2.tex}
  \input{results_simplex/problemvarying-coefficients_l3.tex}
  \input{results_simplex/problemvarying-coefficients_l4.tex}
  }

  \FloatBarrier
}
{

  \input{new_results/problemvarying-coefficientsfinal}

}

\subsection{Perturbation problem}
\red{We now study the perturbation problem considered in Section \ref{section: perturbation problem} for various values of $\epsilon \in (0,1]$ and mesh size $h$. The aim is to verify the convergence orders discussed in Section \ref{section: perturbation problem} especially the improved order of the new scheme $\Conemod$ compared to the original $\Conenonconf$ space. We are only considering problems without boundary layers, a discussion of problem \eqref{eqn: perturbation problem} with boundary layers can be found in \cite{zhang_nonconforming_2020}.}
\begin{example}\label{ex: perturbation problem}
  Consider problem \eqref{eqn: perturbation problem} with right hand side given by $f = \epsilon^2 \Delta^2 u - \Delta u$. As exact solution we use \red{$u(x,y) = (\sin(2\pi x)\sin(2\pi y))^2$ } which satisfies the Dirichlet boundary conditions on the domain $\Omega=(0,1)^2$. \strike{We start by taking $\epsilon = 1$ to study the convergence of our three VEM spaces.} The results on both the structured triangular grid and the Voronoi grid for the values $\epsilon=10^{-2}$, $\epsilon=10^{-8}$ and for polynomial orders $\polOrder=2,3$ are summarized in Table~\ref{table: perturbation results}, the results for $l=4$ are in line with expectations.

  
\end{example}

\input{new_results/problemperturbation_final}

\ifthenelse{\boolean{oldResultsDisplay}}
{
  \begin{figure}[p]\centering
    \includegraphics[width=0.483\textwidth]{results_simplex/problemperturbation_1e-02_l2Energyerr-eps-converted-to.pdf}
    \includegraphics[width=0.483\textwidth]{results_simplex/problemperturbation_1e-02_l2Energyeoc-eps-converted-to.pdf}
    \\
    \includegraphics[width=0.483\textwidth]{results_simplex/problemperturbation_1e-02_l3Energyerr-eps-converted-to.pdf}
    \includegraphics[width=0.483\textwidth]{results_simplex/problemperturbation_1e-02_l3Energyeoc-eps-converted-to.pdf}
    \\
    \includegraphics[width=0.483\textwidth]{results_simplex/problemperturbation_1e-02_l4Energyerr-eps-converted-to.pdf}
    \includegraphics[width=0.483\textwidth]{results_revised/simplex/problemperturbation_1e-02_l4Energyeoc-eps-converted-to.pdf}
    \caption{Perturbation problem with $\epsilon=10^{-2}$ and $l=2,3,4$ (top to bottom). Left column shows error in the energy norm with respect to number of degrees of freedom for the three spaces. Right column shows experimental order of convergence again for the energy norm versus the grid spacing $h$.}
    \label{fig: perteps1e-2 simplex}
  \end{figure}

  \begin{figure}[p]\centering
    \includegraphics[width=0.483\textwidth]{results_simplex/problemperturbation_1e-08_l2Energyerr-eps-converted-to.pdf}
    \includegraphics[width=0.483\textwidth]{results_simplex/problemperturbation_1e-08_l2Energyeoc-eps-converted-to.pdf}
    \\
    \includegraphics[width=0.483\textwidth]{results_simplex/problemperturbation_1e-08_l3Energyerr-eps-converted-to.pdf}
    \includegraphics[width=0.483\textwidth]{results_simplex/problemperturbation_1e-08_l3Energyeoc-eps-converted-to.pdf}
    \\
    \includegraphics[width=0.483\textwidth]{results_simplex/problemperturbation_1e-08_l4Energyerr-eps-converted-to.pdf}
    \includegraphics[width=0.483\textwidth]{results_revised/simplex/problemperturbation_1e-08_l4Energyeoc-eps-converted-to.pdf}
    \caption{Perturbation problem with $\epsilon=10^{-8}$ and $l=2,3,4$ (top to bottom). Left column shows error in the energy norm with respect to number of degrees of freedom for the three spaces. Right column shows experimental order of convergence again for the energy norm versus the grid spacing $h$.}
    \label{fig: perteps1e-8 simplex}
    \end{figure}
}
{
}

\subsection{Nonlinear problems}
We conclude this section with some preliminary results, which demonstrate that the method discussed in this paper is well suited to solve complex nonlinear fourth order problems. We choose \red{two} problems which have been studied in the literature, both energy minimization problems solved by a gradient descent algorithm. In both cases the mathematical models are time dependent fourth order problems. We use a Rothe approach in which first the problem is discretized in time. The resulting spatial problems are stationary fourth order problems with linearizations of the form studied here. 
\strike{In both cases we only show results for the standard space with.}

Note that no special linearization is required, we use a standard Newton solver to handle the fourth order nonlinear problems arising from the implicit time discretization. Please note that this final subsection is only a first investigation into applying our approach to more complicated settings and it is beyond the scope of this paper to do a detailed analysis. We therefore restrict the presentation to the $\Conenonconf$ space with $\polOrder=3$. We show results on triangular grids on $\Omega=(0,1)^2$. 
\strike{We show results on both a structured simplex grid and a Voronoi grid on $\Omega=(0,1)^2$.}

\begin{example}\label{ex: Cahn Hilliard}
  For our first problem we solve the Cahn-Hilliard equation using a fully implicit backward Euler method to discretize the problem in time. A virtual element method for this problem was studied in \cite{antonietti_c1_2016} where a different approach for defining the projection operators, requiring a special linearization, restricted the method to $l=2$.

  Let $\psi : \R \rightarrow \R$ be defined as $\psi(x) = \frac{(1-x^2)^2}{4}$ and define $\phi(x) = \psi(x)^{\prime}$.
  Then, the Cahn-Hilliard problem reads as follows: find $u\colon \Omega \times [0,T] \rightarrow \R$ such that
  \newpage
  \begin{align*}
    \partial_t u  - \Delta (\phi(u)-\gamma^2 \Delta u) = 0 \quad &\emph{in} \ \Omega \times [0,T] ,\\
    u(\cdot,0) = u_0(\cdot)  \quad &\emph{in} \ \Omega,\\
    \partial_n u = \partial_n \big( \phi(u) - \gamma^2\Delta u \big) = 0  \quad &\emph{on} \ \partial \Omega \times [0,T].
  \end{align*} 
  Note that this problem requires slightly different boundary conditions compared to the problems studied so far.
  Some snapshots from a simulation with $\gamma=0.02$ on a $60 \times 60$ grid and time step $\tau=10^{-3}$ are displayed in Figure \ref{figure: CH example}. The initial conditions were a small perturbation in the centre of the domain. The first snapshot is taken at a point in time where the phase separation is already well developed. The following figures then show the usual coarsening ending with the red phase concentrated in approximately a circle in the centre of the domain.

  \begin{figure}[h] \centering
          \includegraphics[width=0.2\textwidth]{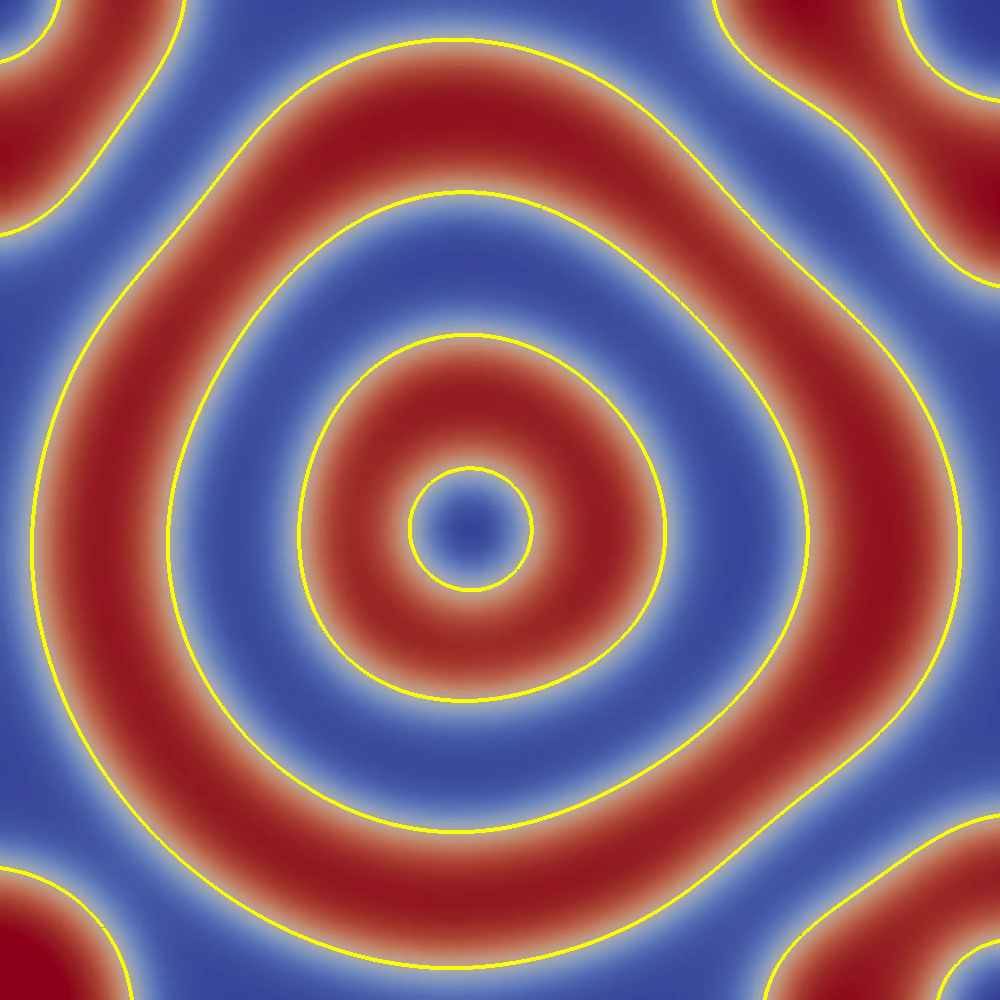}
          \includegraphics[width=0.2\textwidth]{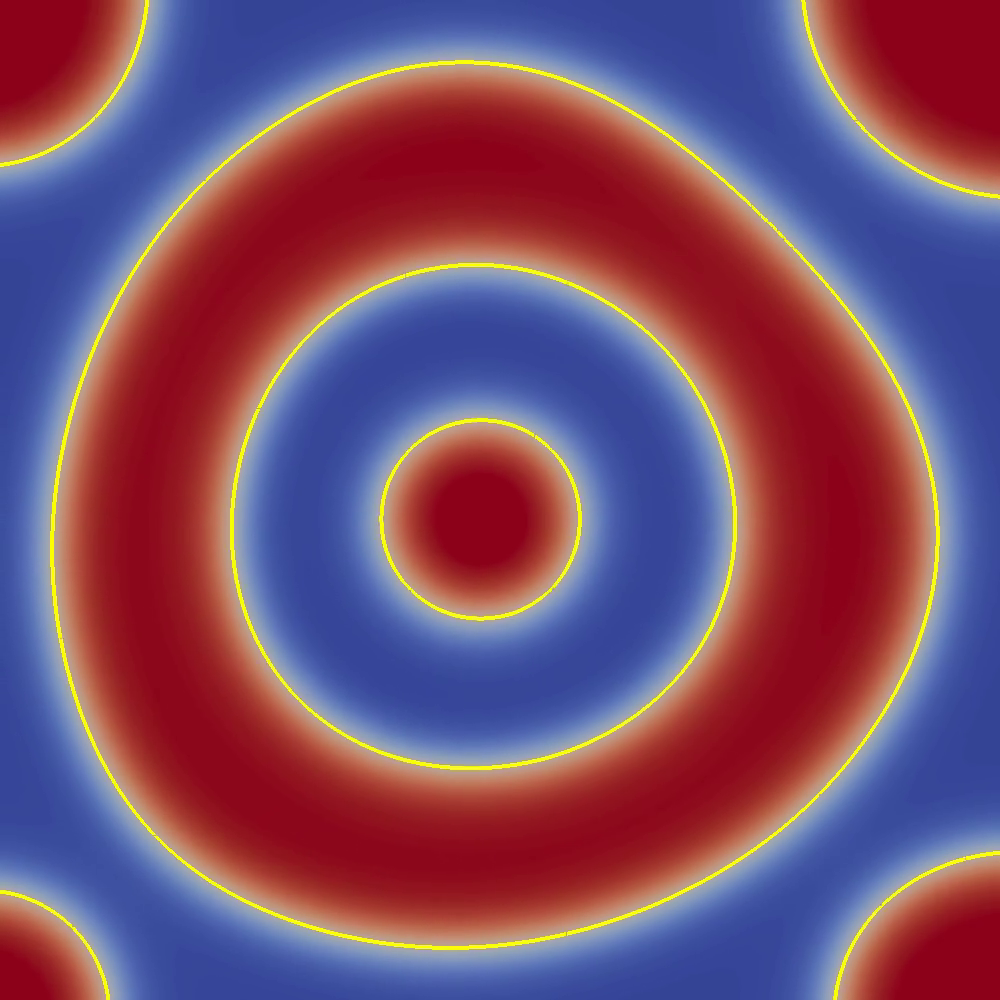}
          \includegraphics[width=0.2\textwidth]{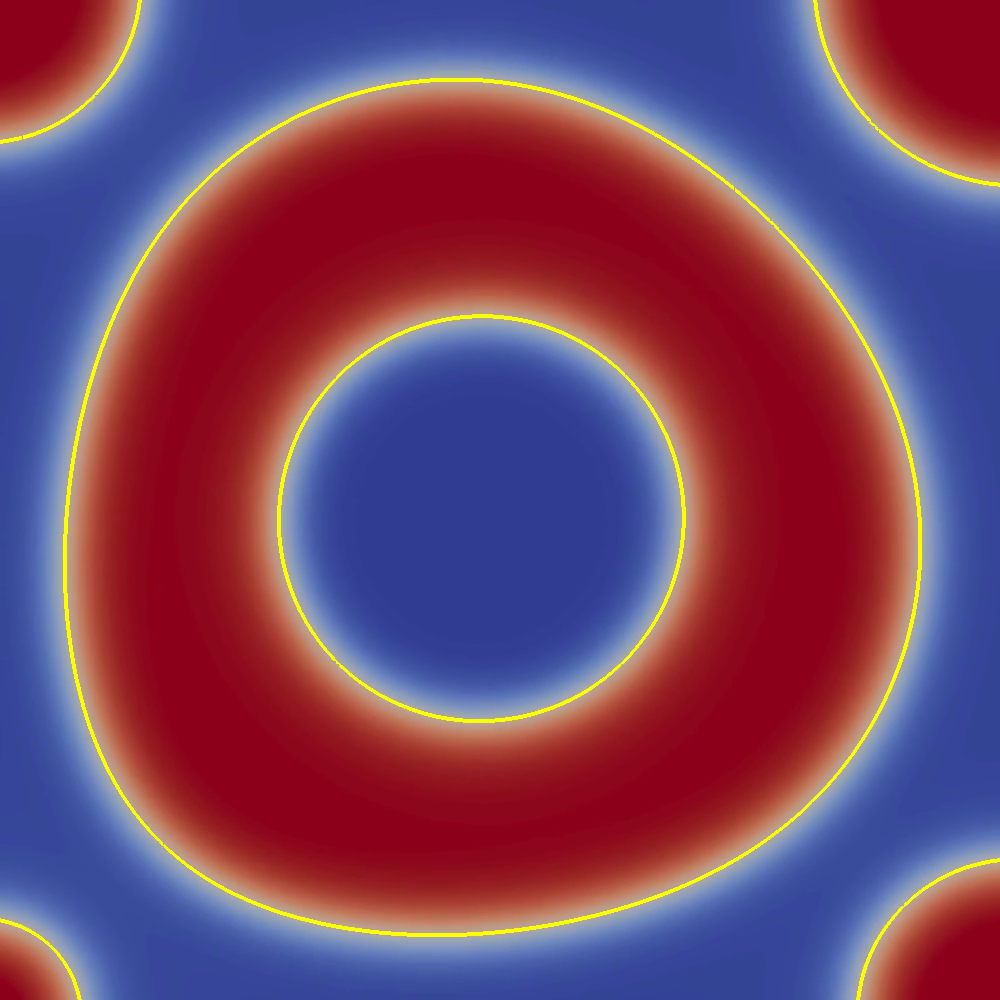} 
          \includegraphics[width=0.2\textwidth]{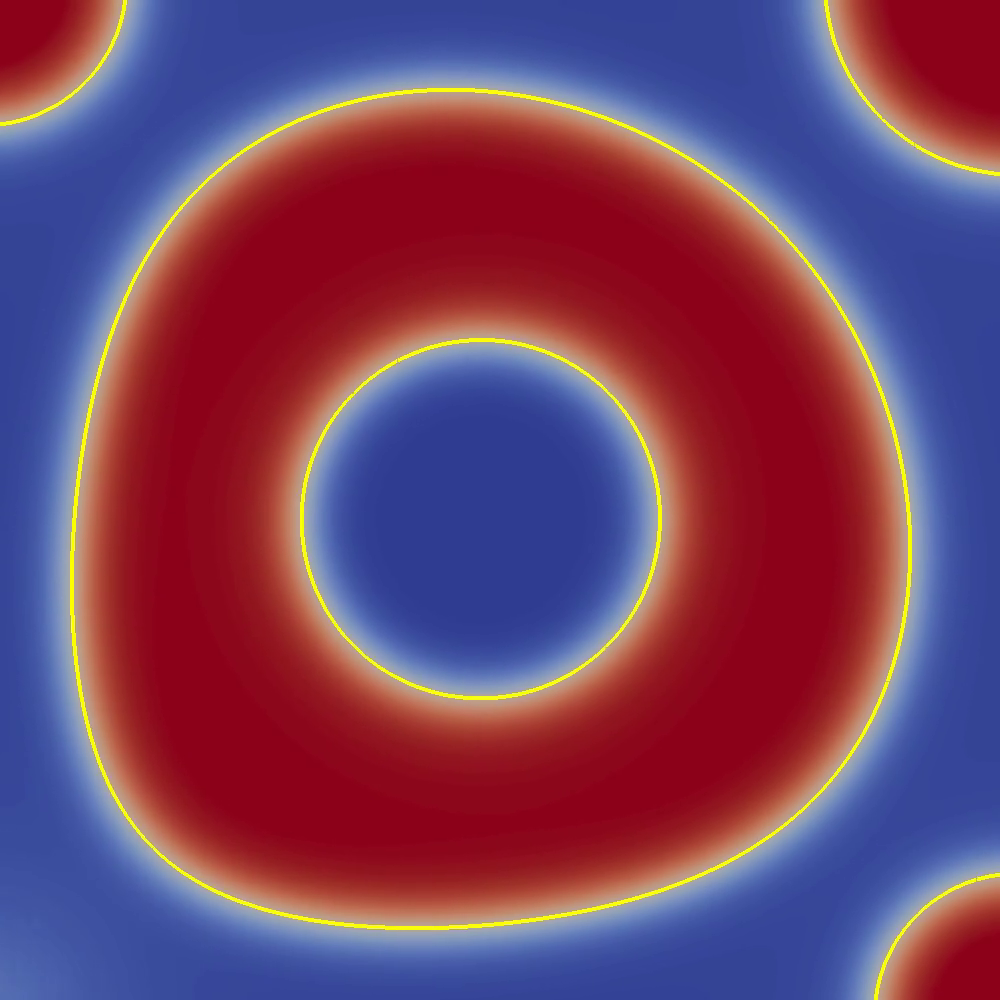} \\
          \includegraphics[width=0.2\textwidth]{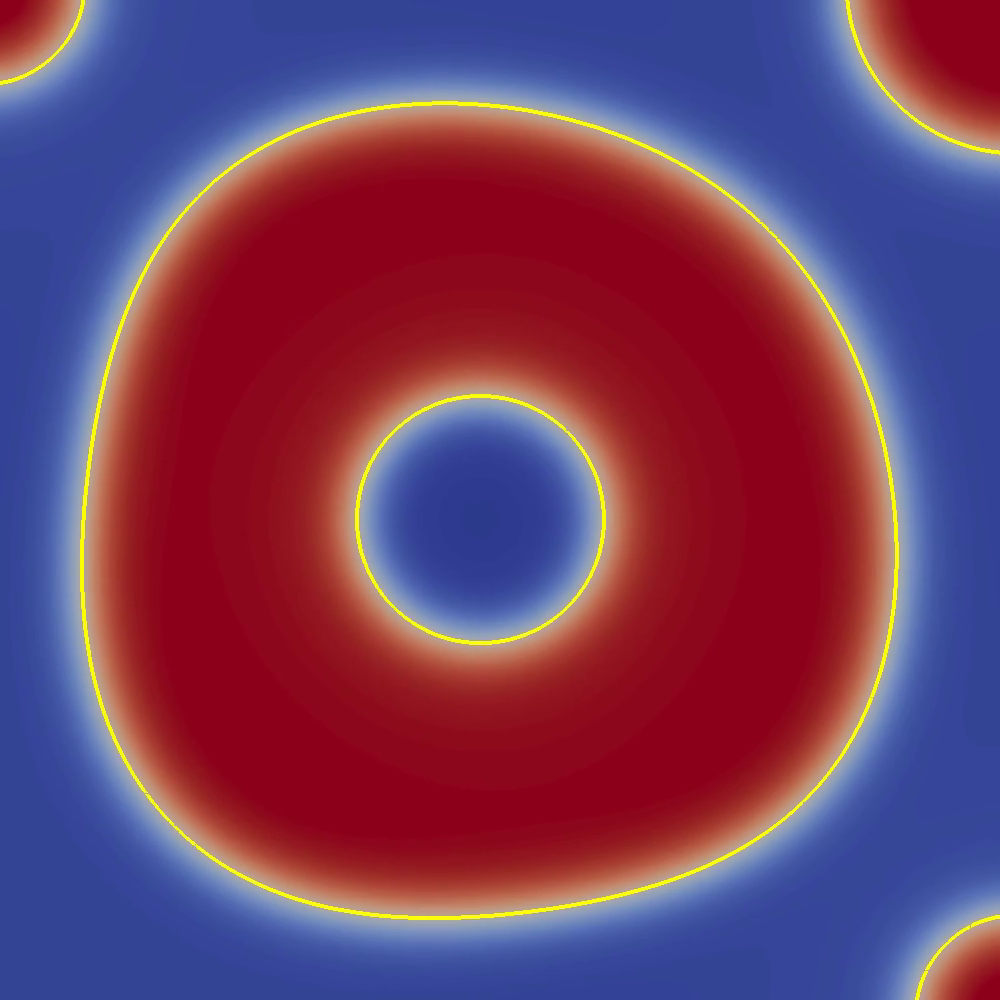} 
          \includegraphics[width=0.2\textwidth]{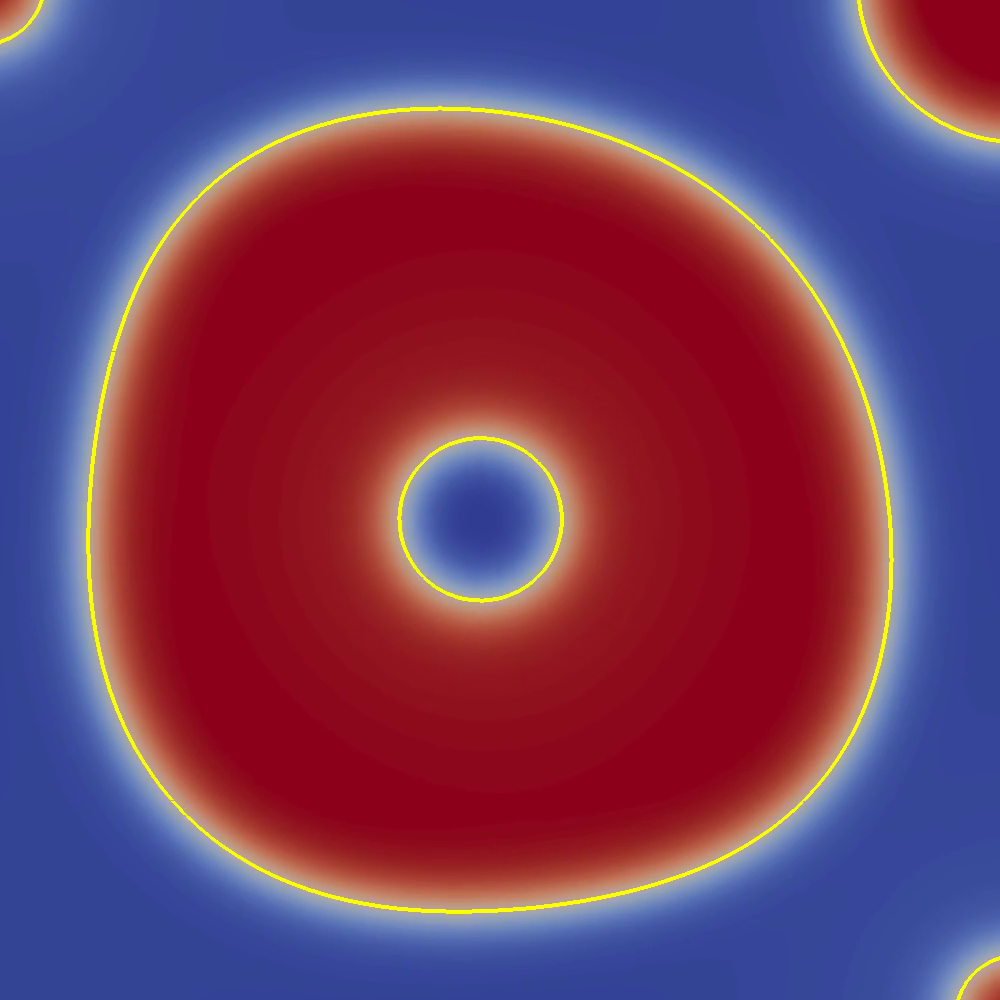}
          \includegraphics[width=0.2\textwidth]{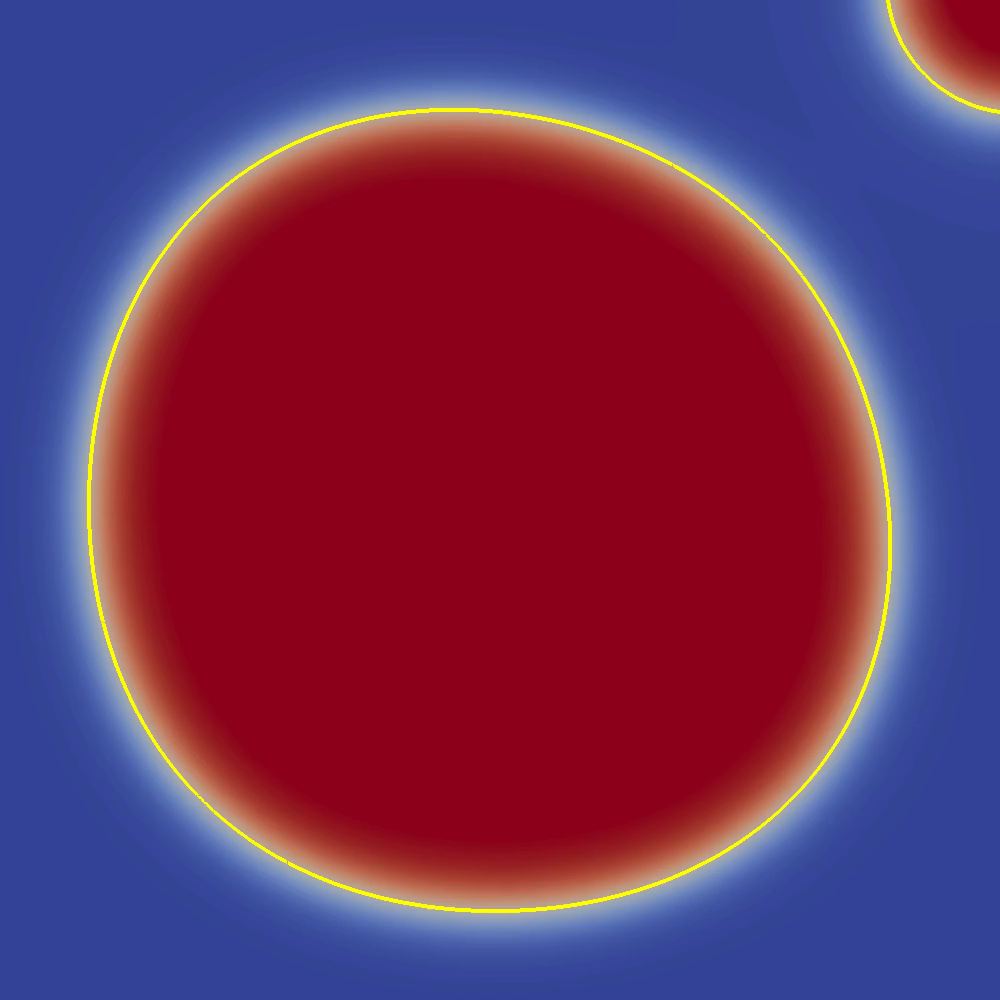}
          \includegraphics[width=0.2\textwidth]{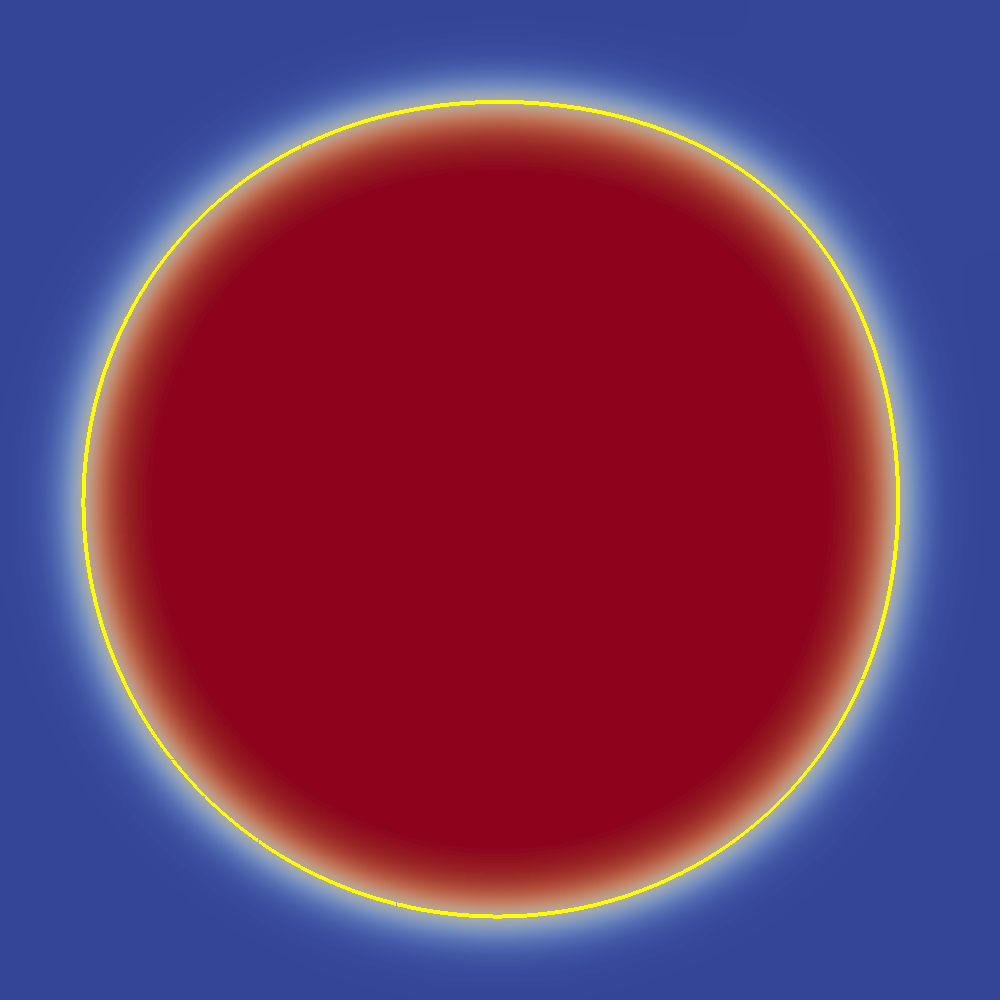}
      \caption{Some snapshots from the simulation of the Cahn-Hilliard
      problem. \label{figure: CH example}}
  \end{figure}
\end{example}

\begin{example}\label{ex: Willmore flow}
  In our final example we study the minimization of the Willmore energy of a surface, in the case where the surface is given by a graph over a flat domain $\Omega$. The corresponding Euler-Lagrange equations can be rewritten as a fourth order problem for a function 
  $u$ defined over $\Omega$. We use a second order two stage implicit Runge-Kutta method as suggested in \cite{deckelnick_c1finite_2015}. The resulting problem is a system of two nonlinear fourth order partial differential equations. The corresponding linearized problem is a system of equations for the two Runge-Kutta stages which is of the form \eqref{eqn: bilinear form cts}.

  As detailed for example in \cite{deckelnick_c1finite_2015}, the Willmore functional for the graph of a function ${u \in W^{2,\infty}(\Omega)}$ is given by 
  \begin{align*}
    W(u) = \frac{1}{2}\int_{\Omega} [ \ E(Du) : D^2  u \ ]^2 \, \mathrm{d}x.
  \end{align*}
  The function $u$ satisfies the same Dirichlet boundary conditions given in (\ref{eqn: strong form}).
  The function $E : \R^2 \rightarrow Sym(2)$, where $Sym(2)$ denotes the space of symmetric $2 \times 2$ matrices, is given by
  $$
    E_{ij}(w) := \frac{1}{(1+|w|^2)^{\frac{1}{4}}} \big( \delta_{ij} - \frac{w_i w_j}{1+|w|^2} \big)
    \quad \text{for } i,j=1,2, \, \text{ and } w \in \R^2.
  $$
  We initialize the gradient descent algorithm with $u(x,y)=\big(\sin(2\pi x)\sin(2\pi y)\big)^2$.

  \red{The time stepping scheme \cite{deckelnick_c1finite_2015} produces an approximation $u_h^n$ which can be constructed in the following way:
  given $u_h^n$ and time step $\tau_n$, solve the following coupled system for the two functions $u^{n,1}_h, u^{n,2}_h$,}
  \begin{align*}
    \sum_{j=0}^2 \frac{\hat \alpha_{1,j}}{\tau_n} 
    \int_{\Omega} \frac{u_{h}^{n,j} \varphi_h}{\sqrt{1+|\nabla (u_{h}^{n,1}) |^2}} \mathrm{d}x + \langle W^{\prime} ( u_{h}^{n,1} ),\varphi_h \rangle = 0 
    \\
    \sum_{j=0}^2 \frac{\hat \alpha_{2,j}}{\tau_n} 
    \int_{\Omega} \frac{u_{h}^{n,j} \varphi_h}{\sqrt{1+|\nabla (u_{h}^{n,2}) |^2}} \mathrm{d}x + \langle W^{\prime} ( u_{h}^{n,2} ),\varphi_h \rangle = 0 
  \end{align*}
  \red{setting $u_h^{n,0} = u_h^n$, where $\hat \alpha_{1,j} = (-2, \tfrac{3}{2}, \tfrac{1}{2})$ and ${\hat \alpha_{2,j} = (2, -\tfrac{9}{2}, \tfrac{5}{2})}$, for $j=0,1,2$. 
  Then set $u_h^{n+1} = u_h^{n,2}$.}
\end{example}

Figure~\ref{figure: Willmore example} shows the evolution of the surface and a graph showing the decay of the Willmore energy over time. Overall the results indicate the method is well suited for solving this type of problem.  

\begin{figure}[h] \centering
     \includegraphics[width=0.19\textwidth]{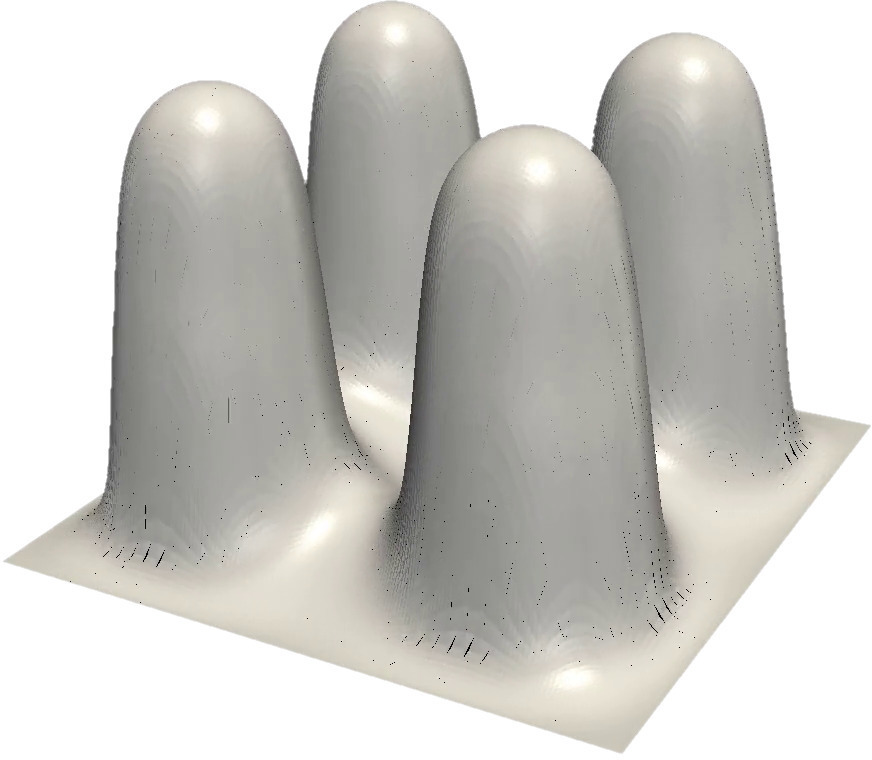}
     \includegraphics[width=0.19\textwidth]{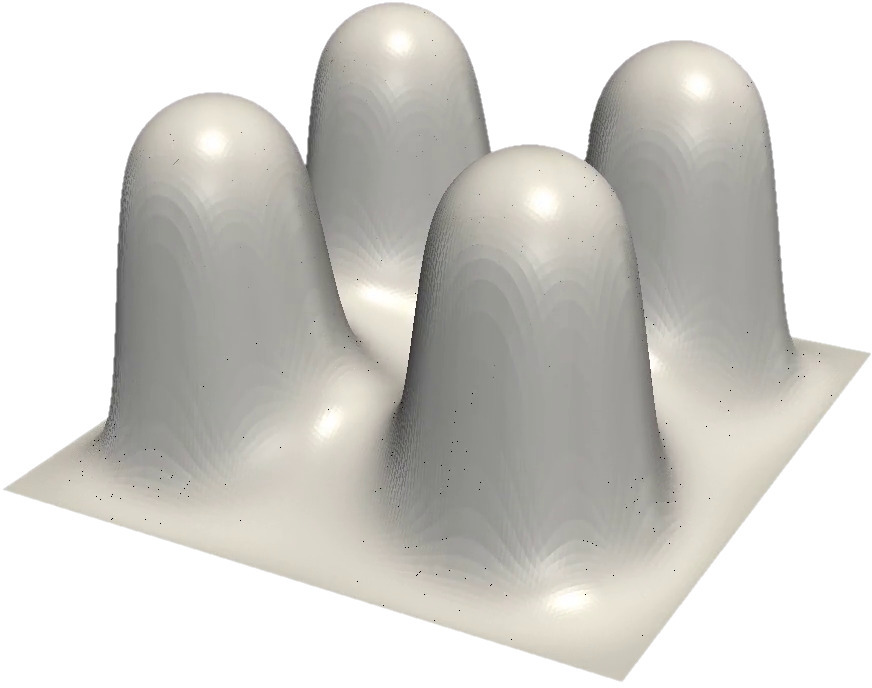}
     \includegraphics[width=0.19\textwidth]{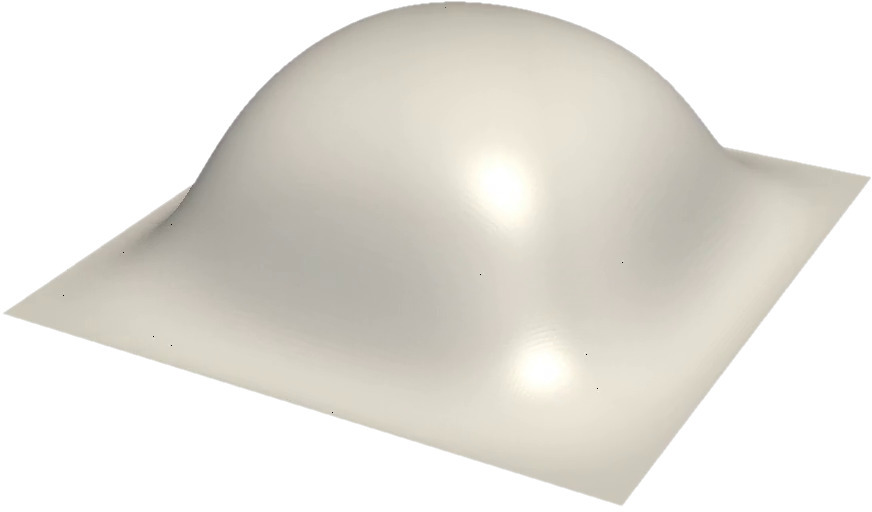}
     \includegraphics[width=0.19\textwidth]{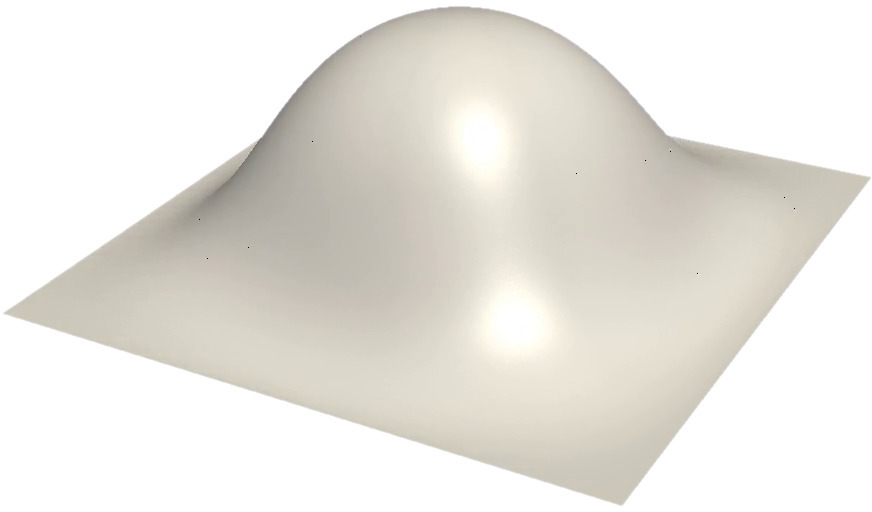} 
     \includegraphics[width=0.19\textwidth]{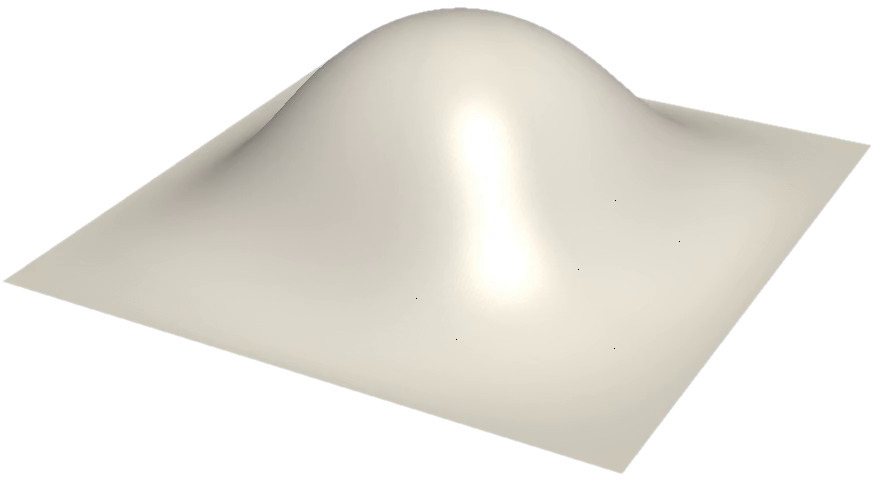} \\
     \includegraphics[width=0.5\textwidth]{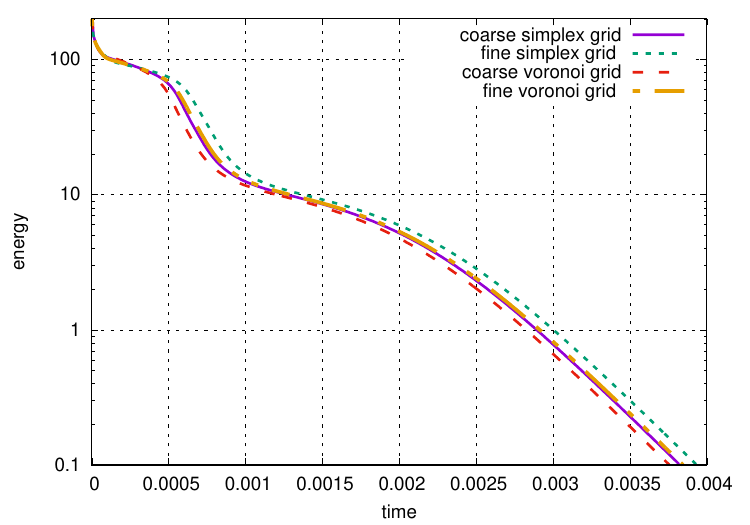}
     \caption{Some snapshots from the Willmore flow problem
     \red{(top row, left to right:
     $t=2.5\times 10^{-5},2.5\times 10^{-4}, 1.5\times 10^{-3}, 2.5\times 10^{-3}, 4.0\times 10^{-3}$)}.  The evolution of the Willmore energy computed on simplicial \red{and Voronoi grids} with approximate resolutions of ($40 \times 40$ and $80 \times 80$) is shown in the bottom figure. 
     \label{figure: Willmore example}}
\end{figure}

%% file: Sections/meshdata.tex
\begin{figure}[p]
    \centering
    \includegraphics[trim={1.75cm 0 2.75cm 0}, clip, width=0.3\textwidth]{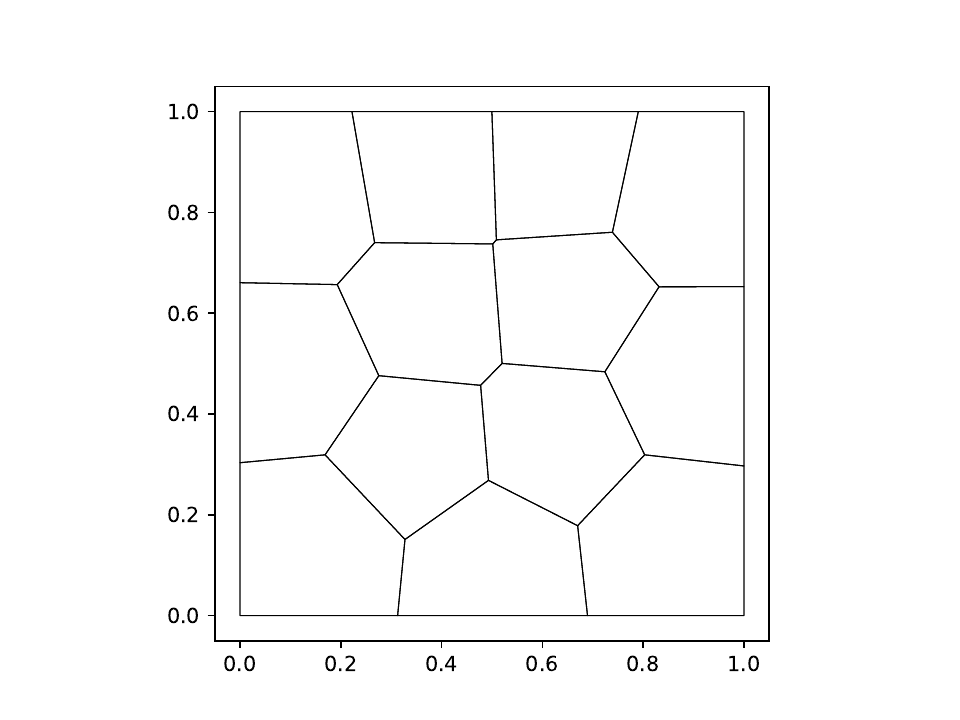}
    \includegraphics[trim={1.75cm 0 2.75cm 0}, clip, width=0.3\textwidth]{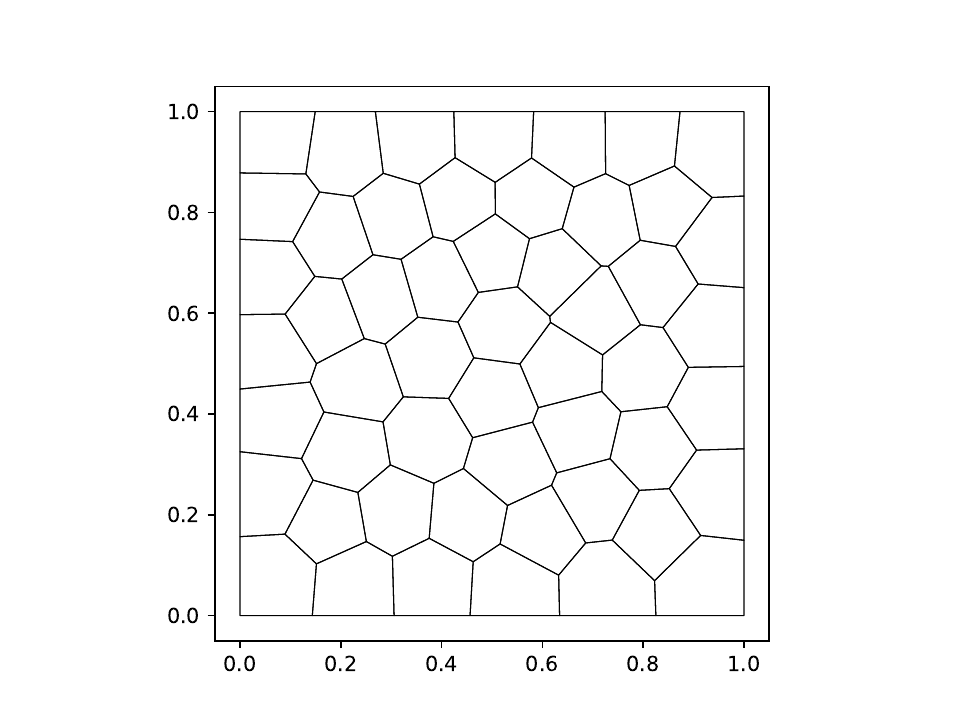}
    \includegraphics[trim={1.75cm 0 2.75cm 0}, clip, width=0.3\textwidth]{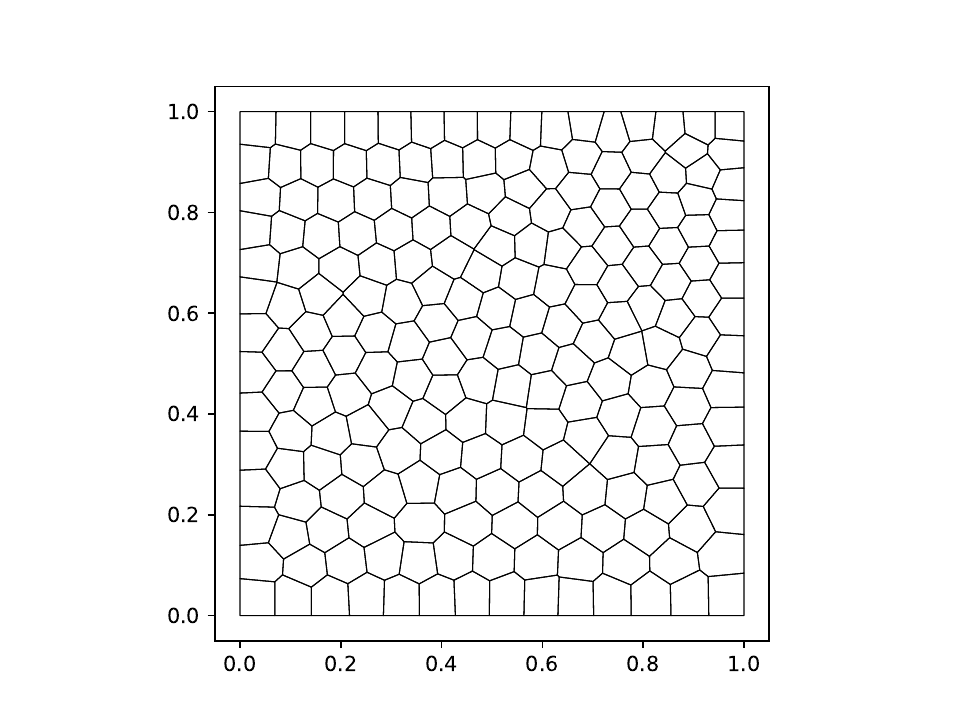}
    \caption{First three mesh refinements for the Voronoi grid on the domain $\Omega = (0,1)^2$.}
    \label{fig: Voronoi grids}
  \end{figure}

\input{new_results/gridData/gridData}

%% file: new_results/gridData/gridData.tex
\begin{table}[!htbp]
  \centering
  \caption{Grid data and total number of degrees of freedom of the sequence
  of simplex meshes (top) and Voronoi meshes (below) for all three spaces $\Conenonconf$, $\Conemod$, and $\CzeroconfSpace$. 
  The second and third columns correspond to the number of grid elements and the mesh size, $h$, respectively. The final columns show the total number of degrees of freedom for the three VEM spaces, for each order $\polOrder=2,3,$ and $4$. Notice that the $\Conenonconf$ and $\Conemod$ spaces have the same dof tuple and hence share the same number of dofs for all polynomial orders and grids.}
  {
  \setlength{\tabcolsep}{4pt}
      {\scriptsize
      \begin{tabular}{lrr|rrr|rrr|rrr}
        \multicolumn{12}{c}{\emph{Simplex meshes grid data}} 
        \\[1em]
        \toprule
        {} & \multirow{2}{*}{grid size} & \multirow{2}{*}{mesh size} &  \multicolumn{3}{c|}{$\polOrder=2$} &  \multicolumn{3}{c|}{$\polOrder=3$} &  \multicolumn{3}{c}{$\polOrder=4$} \\
        & & & $\Conenonconf$ & $\Conemod$ & $\CzeroconfSpace$ & $\Conenonconf$ & $\Conemod$ & $\CzeroconfSpace$ & $\Conenonconf$ & $\Conemod$ & $\CzeroconfSpace$ \\
        \midrule
        0 &         18 &  0.471405 &             49 &              49 &             82 &            115 &             115 &            148 &                  199 &                  199 &                  232 \\
        1 &         72 &  0.235702 &            169 &             169 &            289 &            409 &             409 &            529 &                  721 &                  721 &                  841 \\
        2 &        288 &  0.117851 &            625 &             625 &           1081 &           1537 &            1537 &           1993 &                 2737 &                 2737 &                 3193 \\
        3 &       1152 &  0.058926 &           2401 &            2401 &           4177 &           5953 &            5953 &           7729 &                10657 &                10657 &                12433 \\
        4 &       4608 &  0.029463 &           9409 &            9409 &          16417 &          23425 &           23425 &          30433 &                42049 &                42049 &                49057 \\
        5 &      18432 &  0.014731 &          37249 &           37249 &          65089 &          92929 &           92929 &         120769 &               167041 &               167041 &               194881 \\
        6 &      73728 &  0.007366 &         148225 &          148225 &         259201 &         370177 &          370177 &         481153 &                  --- &                  --- &                  --- \\
        \bottomrule
        \\[2em]
        \multicolumn{12}{c}{\emph{Voronoi meshes grid data}} 
        \\[0.75em]
        \toprule
        {} & \multirow{2}{*}{grid size} & \multirow{2}{*}{mesh size} &  \multicolumn{3}{c|}{$\polOrder=2$} &  \multicolumn{3}{c|}{$\polOrder=3$} &  \multicolumn{3}{c}{$\polOrder=4$} \\
        & &  & $\Conenonconf$ & $\Conemod$ & $\CzeroconfSpace$ & $\Conenonconf$ & $\Conemod$ & $\CzeroconfSpace$ & $\Conenonconf$ & $\Conemod$ & $\CzeroconfSpace$ \\
        \midrule
        0 &         13 &  0.353416 &             68 &              68 &            108 &            148 &             148 &            188 &                  241 &                  241 &                  281 \\
        1 &         52 &  0.186351 &            263 &             263 &            420 &            577 &             577 &            734 &                  943 &                  943 &                 1100 \\
        2 &        208 &  0.098618 &           1043 &            1043 &           1668 &           2293 &            2293 &           2918 &                 3751 &                 3751 &                 4376 \\
        3 &        832 &  0.047502 &           4163 &            4163 &           6660 &           9157 &            9157 &          11654 &                14983 &                14983 &                17480 \\
        4 &       3328 &  0.025168 &          16643 &           16643 &          26628 &          36613 &           36613 &          46598 &                59911 &                59911 &                69896 \\
        5 &      13312 &  0.012486 &          66563 &           66563 &         106500 &         146437 &          146437 &         186374 &               239623 &               239623 &               279560 \\
        6 &      53248 &  0.006444 &         266243 &          266243 &         425988 &         585733 &          585733 &         745478 &                  --- &                  --- &                  --- \\
        \bottomrule
      \end{tabular}
      }
  }   
  \label{table: grid data} 
\end{table}

%% file: new_results/problemvarying-coefficientsfinal.tex
   \begin{table}[!htbp]
    \centering
    \caption{
        Example \ref{ex: varying coefficients} for $\polOrder=2,3,4$
        (top to bottom) on simplex (left) and Voronoi grids (right). 
        The errors
        and eocs are computed with respect to the relative energy norm.
        The observed convergence rates of $\polOrder-1$ are in accordance with
        our convergence results summarized in Theorem~\ref{thm: energy norm}. 
        For a given grid all methods produce comparable errors but require a varying number
        of degrees of freedom as discussed in Table~\ref{table: grid data}.    
    }
    {
    \setlength{\tabcolsep}{4pt}
        {\scriptsize
        \begin{tabular*}{0.48\textwidth}{lcccccc}
            \toprule
            & \multicolumn{2}{c}{$\Conenonconf$} & \multicolumn{2}{c}{$\Conemod$} & \multicolumn{2}{c}{$\CzeroconfSpace$} \\
            {} & error &  eoc &  error &  eoc & error &  eoc \\
            \midrule
            0 &  1.4561e+00 &        --- &   1.4088e+00 &         --- &  1.2460e+00 &        --- \\
            1 &  9.5268e-01 &       0.61 &   9.4683e-01 &        0.57 &  8.3299e-01 &       0.58 \\
            2 &  5.5475e-01 &       0.78 &   5.5435e-01 &        0.77 &  4.8641e-01 &       0.78 \\
            3 &  2.9218e-01 &       0.92 &   2.9233e-01 &        0.92 &  2.5887e-01 &       0.91 \\
            4 &  1.4813e-01 &       0.98 &   1.4825e-01 &        0.98 &  1.3190e-01 &       0.97 \\
            5 &  7.4329e-02 &       0.99 &   7.4395e-02 &        0.99 &  6.5994e-02 &       1.00 \\
            6 &  3.7198e-02 &       1.00 &   3.7231e-02 &        1.00 &  3.2984e-02 &       1.00 \\
            \bottomrule 
            \\
            \toprule
            & \multicolumn{2}{c}{$\Conenonconf$} & \multicolumn{2}{c}{$\Conemod$} & \multicolumn{2}{c}{$\CzeroconfSpace$} \\
            {} & error &  eoc &  error &  eoc & error &  eoc \\
            \midrule
            0 &  8.8844e-01 &        --- &   8.8176e-01 &         --- &  9.6512e-01 &        --- \\
            1 &  2.6370e-01 &       1.75 &   2.6376e-01 &        1.74 &  2.6919e-01 &       1.84 \\
            2 &  7.5506e-02 &       1.80 &   7.5546e-02 &        1.80 &  7.5609e-02 &       1.83 \\
            3 &  2.0295e-02 &       1.90 &   2.0304e-02 &        1.90 &  2.0324e-02 &       1.90 \\
            4 &  5.2140e-03 &       1.96 &   5.2159e-03 &        1.96 &  5.2225e-03 &       1.96 \\
            5 &  1.3127e-03 &       1.99 &   1.3132e-03 &        1.99 &  1.3148e-03 &       1.99 \\
            6 &  3.2910e-04 &       2.00 &   3.2921e-04 &        2.00 &  3.2962e-04 &       2.00 \\
            \bottomrule
            \\
            \toprule
            & \multicolumn{2}{c}{$\Conenonconf$} & \multicolumn{2}{c}{$\Conemod$} & \multicolumn{2}{c}{$\CzeroconfSpace$} \\
            {} & error &  eoc &  error &  eoc & error &  eoc \\
            \midrule
            0 &  4.9566e-01 &        --- &   4.9328e-01 &         --- &  7.7064e-01 &        --- \\
            1 &  1.1762e-01 &       2.08 &   1.1708e-01 &        2.07 &  1.2327e-01 &       2.64 \\
            2 &  2.0309e-02 &       2.53 &   2.0245e-02 &        2.53 &  1.9575e-02 &       2.65 \\
            3 &  2.7562e-03 &       2.88 &   2.7503e-03 &        2.88 &  2.6595e-03 &       2.88 \\
            4 &  3.5277e-04 &       2.97 &   3.5212e-04 &        2.97 &  3.4121e-04 &       2.96 \\
            5 &  4.4191e-05 &       3.00 &   4.4113e-05 &        3.00 &  4.2735e-05 &       3.00 \\
            \bottomrule
        \end{tabular*}
        }
    \quad
        {\scriptsize
        \begin{tabular*}{0.48\textwidth}{lcccccc}
            \toprule
            & \multicolumn{2}{c}{$\Conenonconf$} & \multicolumn{2}{c}{$\Conemod$} & \multicolumn{2}{c}{$\CzeroconfSpace$} \\
            {} & error &  eoc &  error &  eoc & error &  eoc \\
            \midrule
            0 &  1.0681e+00 &        --- &   1.0673e+00 &         --- &  1.0757e+00 &        --- \\
            1 &  6.3424e-01 &       0.81 &   6.4477e-01 &        0.79 &  6.2755e-01 &       0.84 \\
            2 &  3.3367e-01 &       1.01 &   3.3724e-01 &        1.02 &  3.3168e-01 &       1.00 \\
            3 &  1.6561e-01 &       0.96 &   1.6618e-01 &        0.97 &  1.6376e-01 &       0.97 \\
            4 &  8.1813e-02 &       1.11 &   8.1914e-02 &        1.11 &  8.1217e-02 &       1.10 \\
            5 &  4.0523e-02 &       1.00 &   4.0547e-02 &        1.00 &  4.0366e-02 &       1.00 \\
            6 &  2.0073e-02 &       1.06 &   2.0080e-02 &        1.06 &  2.0028e-02 &       1.06 \\
            \bottomrule \\
            \toprule
            & \multicolumn{2}{c}{$\Conenonconf$} & \multicolumn{2}{c}{$\Conemod$} & \multicolumn{2}{c}{$\CzeroconfSpace$} \\
            {} & error &  eoc &  error &  eoc & error &  eoc \\
            \midrule
            0 &  8.0649e-01 &        --- &   8.0454e-01 &         --- &  8.1866e-01 &        --- \\
            1 &  4.4383e-01 &       0.93 &   4.4233e-01 &        0.93 &  4.4265e-01 &       0.96 \\
            2 &  1.8880e-01 &       1.34 &   1.8817e-01 &        1.34 &  1.8826e-01 &       1.34 \\
            3 &  5.8297e-02 &       1.61 &   5.8130e-02 &        1.61 &  5.8119e-02 &       1.61 \\
            4 &  1.3256e-02 &       2.33 &   1.3222e-02 &        2.33 &  1.3213e-02 &       2.33 \\
            5 &  2.5665e-03 &       2.34 &   2.5596e-03 &        2.34 &  2.5600e-03 &       2.34 \\
            6 &  5.4140e-04 &       2.35 &   5.3985e-04 &        2.35 &  5.4051e-04 &       2.35 \\
            \bottomrule
            \\
            \toprule
            & \multicolumn{2}{c}{$\Conenonconf$} & \multicolumn{2}{c}{$\Conemod$} & \multicolumn{2}{c}{$\CzeroconfSpace$} \\
            {} & error &  eoc &  error &  eoc & error &  eoc \\   
            \midrule
            0 &  7.7981e-01 &        --- &   7.7476e-01 &         --- &  7.6938e-01 &        --- \\
            1 &  1.9288e-01 &       2.18 &   1.9199e-01 &        2.18 &  1.9056e-01 &       2.18 \\
            2 &  2.6162e-02 &       3.14 &   2.6157e-02 &        3.13 &  2.6103e-02 &       3.12 \\
            3 &  2.6309e-03 &       3.14 &   2.6279e-03 &        3.15 &  2.6224e-03 &       3.15 \\
            4 &  3.0888e-04 &       3.37 &   3.0821e-04 &        3.37 &  3.0738e-04 &       3.38 \\
            5 &  3.8858e-05 &       2.96 &   3.8761e-05 &        2.96 &  3.8665e-05 &       2.96 \\
            \bottomrule
        \end{tabular*}
        }
    }
    \label{table: varying coeff}    
\end{table}

%% file: new_results/problemperturbation_final.tex
   \begin{table}[!htbp]
    \centering
    \caption{
        Example \ref{ex: perturbation problem} with $\epsilon=10^{-2}$ ($\polOrder=2,3$ correspond to first two rows of tables) and $\epsilon=10^{-8}$ ($\polOrder=2,3$ correspond to last two rows of tables). 
        Relative energy errors and eocs on simplex grids (left column) and Voronoi grids (right column) are shown.
        The observed convergence rates of $\polOrder-1$ for $\frac{\epsilon}{h}$ large are in accordance with our convergence results summarized in Corollary~\ref{thm: convergence for perturbation problem} and confirm results from Example~\ref{ex: varying coefficients}. 
        Note that the errors on a given grid are larger for the standard $\Conenonconf$ method compared to the other two due to lower convergence on the coarser grids (see two top tables).
        In the case of small $\epsilon$ shown in the lower two rows of tables, the convergence of the three methods differs considerably as expected from Corollary~\ref{thm: convergence for perturbation problem}.
        The $\Conenonconf$ method converges with order $\polOrder-2$, the $\Conemod$ methods converges with order $\polOrder-1$ while the $\CzeroconfSpace$ method converges with order $\polOrder$ but requires more degrees of freedom (see Table~\ref{table: grid data}).   
    }
    {
    \setlength{\tabcolsep}{4pt}
        {\scriptsize
        \begin{tabular*}{0.48\textwidth}{lcccccc}
            \toprule
            & \multicolumn{2}{c}{$\Conenonconf$} & \multicolumn{2}{c}{$\Conemod$} & \multicolumn{2}{c}{$\CzeroconfSpace$} \\
            {} & error &  eoc &  error &  eoc & error &  eoc \\
            \midrule
            0 &  8.6620e-01 &        --- &   8.6005e-01 &         --- &  6.5528e-01 &        --- \\
            1 &  6.5538e-01 &       0.40 &   5.8568e-01 &        0.55 &  2.3493e-01 &       1.48 \\
            2 &  8.3186e-01 &      -0.34 &   3.7994e-01 &        0.62 &  7.6406e-02 &       1.62 \\
            3 &  7.5196e-01 &       0.15 &   2.7569e-01 &        0.46 &  3.1562e-02 &       1.28 \\
            4 &  5.1395e-01 &       0.55 &   1.6165e-01 &        0.77 &  1.7037e-02 &       0.89 \\
            5 &  2.9106e-01 &       0.82 &   8.3481e-02 &        0.95 &  9.1284e-03 &       0.90 \\
            6 &  1.5119e-01 &       0.94 &   4.2045e-02 &        0.99 &  4.6942e-03 &       0.96 \\
            \bottomrule 
            \\
            \toprule
            & \multicolumn{2}{c}{$\Conenonconf$} & \multicolumn{2}{c}{$\Conemod$} & \multicolumn{2}{c}{$\CzeroconfSpace$} \\
            {} & error &  eoc &  error &  eoc & error &  eoc \\
            \midrule
            0 &  5.3974e-01 &        --- &   4.9962e-01 &         --- &  3.9094e-01 &        --- \\
            1 &  3.0540e-01 &       0.82 &   2.2405e-01 &        1.16 &  8.2201e-02 &       2.25 \\
            2 &  1.4972e-01 &       1.03 &   6.5526e-02 &        1.77 &  1.2373e-02 &       2.73 \\
            3 &  5.8172e-02 &       1.36 &   1.7814e-02 &        1.88 &  2.6009e-03 &       2.25 \\
            4 &  1.8078e-02 &       1.69 &   4.7389e-03 &        1.91 &  6.8328e-04 &       1.93 \\
            5 &  4.8713e-03 &       1.89 &   1.2134e-03 &        1.97 &  1.8207e-04 &       1.91 \\
            6 &  1.2455e-03 &       1.97 &   3.0544e-04 &        1.99 &  4.6780e-05 &       1.96 \\
            \bottomrule
            \\[2em]
            \toprule
            & \multicolumn{2}{c}{$\Conenonconf$} & \multicolumn{2}{c}{$\Conemod$} & \multicolumn{2}{c}{$\CzeroconfSpace$} \\
            {} & error &  eoc &  error &  eoc & error &  eoc \\
            \midrule
            0 &  8.5877e-01 &        --- &   8.5917e-01 &         --- &  6.4609e-01 &        --- \\
            1 &  6.5097e-01 &       0.40 &   5.7891e-01 &        0.57 &  2.2744e-01 &       1.51 \\
            2 &  8.9793e-01 &      -0.46 &   3.3200e-01 &        0.80 &  7.3729e-02 &       1.63 \\
            3 &  1.0011e+00 &      -0.16 &   1.7292e-01 &        0.94 &  2.0751e-02 &       1.83 \\
            4 &  1.0297e+00 &      -0.04 &   8.7387e-02 &        0.98 &  4.6582e-03 &       2.16 \\
            5 &  1.0370e+00 &      -0.01 &   4.3811e-02 &        1.00 &  1.1082e-03 &       2.07 \\
            6 &  1.0389e+00 &      -0.00 &   2.1920e-02 &        1.00 &  2.7320e-04 &       2.02 \\
            \bottomrule
            \\
            \toprule
            & \multicolumn{2}{c}{$\Conenonconf$} & \multicolumn{2}{c}{$\Conemod$} & \multicolumn{2}{c}{$\CzeroconfSpace$} \\
            {} & error &  eoc &  error &  eoc & error &  eoc \\
            \midrule
            0 &  5.3752e-01 &        --- &   4.9718e-01 &         --- &  3.8838e-01 &        --- \\
            1 &  3.1739e-01 &       0.76 &   2.2192e-01 &        1.16 &  7.8389e-02 &       2.31 \\
            2 &  1.7739e-01 &       0.84 &   5.6318e-02 &        1.98 &  9.3143e-03 &       3.07 \\
            3 &  9.1963e-02 &       0.95 &   1.4081e-02 &        2.00 &  1.1839e-03 &       2.98 \\
            4 &  4.6623e-02 &       0.98 &   3.6443e-03 &        1.95 &  1.4773e-04 &       3.00 \\
            5 &  2.3454e-02 &       0.99 &   9.2168e-04 &        1.98 &  1.8430e-05 &       3.00 \\
            6 &  1.1760e-02 &       1.00 &   2.3113e-04 &        2.00 &  2.3017e-06 &       3.00 \\
            \bottomrule
        \end{tabular*}
        }
    \quad
        {\scriptsize
        \begin{tabular*}{0.48\textwidth}{lcccccc}
            \toprule
            & \multicolumn{2}{c}{$\Conenonconf$} & \multicolumn{2}{c}{$\Conemod$} & \multicolumn{2}{c}{$\CzeroconfSpace$} \\
            {} & error &  eoc &  error &  eoc & error &  eoc \\
            \midrule
            0 &  8.4665e-01 &        --- &   8.1639e-01 &         --- &  8.2302e-01 &        --- \\
            1 &  2.3787e-01 &       1.98 &   3.1611e-01 &        1.48 &  2.3225e-01 &       1.98 \\
            2 &  1.5968e-01 &       0.63 &   1.4442e-01 &        1.23 &  6.7207e-02 &       1.95 \\
            3 &  1.3462e-01 &       0.23 &   6.2298e-02 &        1.15 &  2.4902e-02 &       1.36 \\
            4 &  1.0860e-01 &       0.34 &   3.1004e-02 &        1.10 &  1.1665e-02 &       1.19 \\
            5 &  6.0813e-02 &       0.83 &   1.5579e-02 &        0.98 &  5.8586e-03 &       0.98 \\
            6 &  2.5955e-02 &       1.29 &   7.8050e-03 &        1.04 &  2.9069e-03 &       1.06 \\
            \bottomrule \\
            \toprule
            & \multicolumn{2}{c}{$\Conenonconf$} & \multicolumn{2}{c}{$\Conemod$} & \multicolumn{2}{c}{$\CzeroconfSpace$} \\
            {} & error &  eoc &  error &  eoc & error &  eoc \\
            \midrule
            0 &  4.6137e-01 &        --- &   4.1236e-01 &         --- &  4.3589e-01 &        --- \\
            1 &  2.1238e-01 &       1.21 &   6.1344e-02 &        2.98 &  6.0831e-02 &       3.08 \\
            2 &  1.2734e-01 &       0.80 &   1.0752e-02 &        2.74 &  1.0077e-02 &       2.83 \\
            3 &  6.2350e-02 &       0.98 &   3.1436e-03 &        1.68 &  2.9684e-03 &       1.67 \\
            4 &  2.6306e-02 &       1.36 &   1.3123e-03 &        1.38 &  1.2834e-03 &       1.32 \\
            5 &  8.4011e-03 &       1.63 &   4.6954e-04 &        1.47 &  4.6350e-04 &       1.45 \\
            6 &  2.1721e-03 &       2.05 &   1.1816e-04 &        2.09 &  1.1686e-04 &       2.08 \\
            \bottomrule
            \\[2em]
            \toprule
            & \multicolumn{2}{c}{$\Conenonconf$} & \multicolumn{2}{c}{$\Conemod$} & \multicolumn{2}{c}{$\CzeroconfSpace$} \\
            {} & error &  eoc &  error &  eoc & error &  eoc \\   
            \midrule
            0 &  8.4602e-01 &        --- &   8.1569e-01 &         --- &  8.2237e-01 &        --- \\
            1 &  2.2460e-01 &       2.07 &   3.1350e-01 &        1.49 &  2.2157e-01 &       2.05 \\
            2 &  1.5228e-01 &       0.61 &   1.4361e-01 &        1.23 &  5.1926e-02 &       2.28 \\
            3 &  1.3845e-01 &       0.13 &   6.1106e-02 &        1.17 &  1.1795e-02 &       2.03 \\
            4 &  1.4227e-01 &      -0.04 &   2.9420e-02 &        1.15 &  2.8687e-03 &       2.23 \\
            5 &  1.3795e-01 &       0.04 &   1.4242e-02 &        1.03 &  7.1092e-04 &       1.99 \\
            6 &  1.3925e-01 &      -0.01 &   6.9820e-03 &        1.08 &  1.7657e-04 &       2.11 \\
            \bottomrule
            \\
            \toprule
            & \multicolumn{2}{c}{$\Conenonconf$} & \multicolumn{2}{c}{$\Conemod$} & \multicolumn{2}{c}{$\CzeroconfSpace$} \\
            {} & error &  eoc &  error &  eoc & error &  eoc \\   
            \midrule
            0 &  4.5379e-01 &        --- &   4.0824e-01 &         --- &  4.3257e-01 &        --- \\
            1 &  1.7132e-01 &       1.52 &   5.3975e-02 &        3.16 &  5.3958e-02 &       3.25 \\
            2 &  7.6736e-02 &       1.26 &   7.5116e-03 &        3.10 &  7.0484e-03 &       3.20 \\
            3 &  3.3586e-02 &       1.13 &   1.1450e-03 &        2.58 &  8.9886e-04 &       2.82 \\
            4 &  1.6043e-02 &       1.16 &   2.1365e-04 &        2.64 &  1.1410e-04 &       3.25 \\
            5 &  7.8712e-03 &       1.02 &   4.8106e-05 &        2.13 &  1.4131e-05 &       2.98 \\
            6 &  3.9540e-03 &       1.04 &   1.1722e-05 &        2.13 &  1.7628e-06 &       3.15 \\
            \bottomrule
        \end{tabular*}
        }
    }  
    \label{table: perturbation results}  
\end{table}

%% file: Sections/conclusion.tex
\section{Conclusion}\label{section: conclusions}

We have presented a general approach for constructing nonconforming VEM projection operators suitable for solving a wide range of fourth order problems. We have analysed the resulting methods and shown that the virtual element approximation in each of the considered spaces converges to the true solution in the energy norm with optimal convergence rates. We also introduced a novel modified nonconforming scheme for solving the fourth order perturbation problem which remained convergent as $\epsilon \rightarrow 0$. Unlike modifications seen in the literature, our change did not require an enlargement of the space and was obtained via an adjustment to the gradient projection. These results were verified with numerical experiments on a variety of polygonal meshes.

We introduced a new concept of describing the degrees of freedom via a \emph{dof tuple} which allowed us to easily encode a variety of different VEM spaces. We followed the approach taken in \cite{ahmad_equivalent_2013,cangiani_conforming_2015} and introduced \emph{dof compatible} projections which were constructed in a way that made them computable entirely from the available degrees of freedom. In particular, we gave examples of how dof compatible projections could be constructed based on constraint least squares problems. 
We were then able to construct the VEM spaces in a way that ensured that the value, gradient, and hessian projections were all identical to $L^2$ projections.

The ease with which our approach can extend to the application of nonlinear fourth order problems was also demonstrated with some additional numerical experiments. However, note that the theory behind the numerical results displayed is out of the scope of this paper and is future work. 

%% file: Sections/acknowledgements.tex
\section*{Acknowledgements}
The authors would like to thank the Isaac Newton Institute for Mathematical Sciences, Cambridge, for support and hospitality during the programme \emph{Geometry, compatibility and structure preservation in computational differential equations}, where work on this paper was carried out.

%% file: main.bbl
\begin{thebibliography}{10}

\bibitem{ahmad_equivalent_2013}
{\sc Ahmad, B., Alsaedi, A., Brezzi, F., Marini, L., and Russo, A.}
\newblock Equivalent projectors for virtual element methods.
\newblock {\em Comput. Math. Appl. 66}, 3 (Sept. 2013), 376--391.

\bibitem{alnaes_unified_2012}
{\sc Aln{\ae}s, M.~S., Logg, A., {\O}lgaard, K.~B., Rognes, M.~E., and Wells,
  G.~N.}
\newblock Unified form language: A domain-specific language for weak
  formulations of partial differential equations.
\newblock {\em ACM Trans. Math. Softw. 40}, 2 (2014), 1--37.

\bibitem{antonietti_c1_2016}
{\sc Antonietti, P.~F., Beir\~{a}o~da Veiga, L., Scacchi, S., and Verani, M.}
\newblock A ${C}{^1}$ {virtual} {element} {method} for the {Cahn}--{Hilliard}
  {Equation} with {polygonal} {meshes}.
\newblock {\em SIAM J. Numer. Anal. 54}, 1 (Jan. 2016), 34--56.

\bibitem{antonietti2018fast}
{\sc Antonietti, P.~F., Houston, P., and Pennesi, G.}
\newblock Fast numerical integration on polytopic meshes with applications to
  discontinuous galerkin finite element methods.
\newblock {\em J. Sci. Comput. 77}, 3 (2018), 1339--1370.

\bibitem{antonietti_fully_2018}
{\sc Antonietti, P.~F., Manzini, G., and Verani, M.}
\newblock The fully nonconforming virtual element method for biharmonic
  problems.
\newblock {\em Math. Models Methods Appl. Sci. 28}, 02 (Feb. 2018), 387--407.

\bibitem{antonietti_conforming_2018}
{\sc Antonietti, P.~F., Manzini, G., and Verani, M.}
\newblock The conforming virtual element method for polyharmonic problems.
\newblock {\em Comput. Math. Appl. 79}, 7 (2020), 2021--2034.

\bibitem{de_dios_nonconforming_2014}
{\sc Ayuso~de Dios, B., Lipnikov, K., and Manzini, G.}
\newblock The nonconforming virtual element method.
\newblock {\em ESAIM Math. Model. Numer. Anal. 50}, 3 (2016), 879--904.

\bibitem{dunegridpaperII}
{\sc Bastian, P., Blatt, M., Dedner, A., Engwer, C., Kl{\"o}fkorn, R.,
  Kornhuber, R., Ohlberger, M., and Sander, O.}
\newblock A generic grid interface for parallel and adaptive scientific
  computing. part {II}: Implementation and tests in {DUNE}.
\newblock {\em Computing 82}, 2--3 (2008), 121--138.

\bibitem{beirao_da_veiga_basic_2013}
{\sc Beir\~{a}o~da Veiga, L., Brezzi, F., Cangiani, A., Manzini, G., Marini,
  L.~D., and Russo, A.}
\newblock {Basic} {principles} {of} {virtual} {element} {methods}.
\newblock {\em Math. Models Methods Appl. Sci. 23}, 01 (Jan. 2013), 199--214.

\bibitem{da_veiga_serendipity_2015}
{\sc Beir{\~a}o~da Veiga, L., Brezzi, F., Marini, L., and Russo, A.}
\newblock Serendipity nodal {VEM} spaces.
\newblock {\em Comput. \& Fluids 141\/} (2016), 2--12.

\bibitem{beirao2016virtual}
{\sc Beir{\~a}o~da Veiga, L., Brezzi, F., Marini, L., and Russo, A.}
\newblock Virtual element method for general second-order elliptic problems on
  polygonal meshes.
\newblock {\em Math. Models Methods Appl. Sci. 26}, 04 (2016), 729--750.

\bibitem{hdiv}
{\sc Beir{\~a}o~da Veiga, L., Brezzi, F., Marini, L.~D., and Russo, A.}
\newblock H(div) and {H}(curl)-conforming {VEM}.
\newblock {\em Numer. Math. 133}, 2 (2016), 303--332.

\bibitem{da_veiga_c1_2019}
{\sc Beir{\~a}o~da Veiga, L., Dassi, F., and Russo, A.}
\newblock A ${C}^{1}$ virtual element method on polyhedral meshes.
\newblock {\em Comput. Math. Appl. 79}, 7 (2020), 1936--1955.

\bibitem{da2014mimetic}
{\sc Beir{\~a}o~da Veiga, L., Lipnikov, K., and Manzini, G.}
\newblock {\em The mimetic finite difference method for elliptic problems},
  vol.~11.
\newblock Springer, 2014.

\bibitem{da_veiga_divergence_2015}
{\sc Beir{\~a}o~da Veiga, L., Lovadina, C., and Vacca, G.}
\newblock Divergence free virtual elements for the {Stokes} problem on
  polygonal meshes.
\newblock {\em ESAIM Math. Model. Numer. Anal. 51}, 2 (2017), 509--535.

\bibitem{bonaldi2017hybrid}
{\sc Bonaldi, F., Di~Pietro, D.~A., Geymonat, G., and Krasucki, F.}
\newblock A hybrid high-order method for kirchhoff-love plate bending problems.
\newblock {\em arXiv preprint arXiv:1706.06781\/} (2017).

\bibitem{brenner_mathematical_2008}
{\sc Brenner, S.~C., and Scott, L.~R.}
\newblock {\em The mathematical theory of finite element methods}.
\newblock No.~15. Springer, New York, 2008.

\bibitem{brenner2004poincare}
{\sc Brenner, S.~C., Wang, K., and Zhao, J.}
\newblock Poincar{\'e}--friedrichs inequalities for piecewise ${H}^{2}$
  functions.
\newblock {\em Numer. Funct. Anal. Optim. 25}, 5-6 (2004), 463--478.

\bibitem{brezzi_virtual_2013}
{\sc Brezzi, F., and Marini, L.~D.}
\newblock Virtual {element} {methods} for plate bending problems.
\newblock {\em Comput. Methods Appl. Mech. Eng 253\/} (Jan. 2013), 455--462.

\bibitem{cangiani_non-conforming_2016}
{\sc Cangiani, A., Gyrya, V., and Manzini, G.}
\newblock The nonconforming virtual element method for the {Stokes} equations.
\newblock {\em SIAM J. Numer. Anal. 54}, 6 (2016), 3411--3435.

\bibitem{cangiani_conforming_2015}
{\sc Cangiani, A., Manzini, G., and Sutton, O.~J.}
\newblock Conforming and nonconforming virtual element methods for elliptic
  problems.
\newblock {\em IMA J. Numer. Anal. 37}, 3 (2017), 1317--1354.

\bibitem{clarlet1987finite}
{\sc Ciarlet, P.~G.}
\newblock {\em The finite element method for elliptic problems}.
\newblock North-Holland Publishing Company, 1987.

\bibitem{dassi_bricks_2018}
{\sc Dassi, F., and Vacca, G.}
\newblock Bricks for the mixed high-order virtual element method: Projectors
  and differential operators.
\newblock {\em Applied Numerical Mathematics 155\/} (2020), 140--159.

\bibitem{deckelnick_c1finite_2015}
{\sc Deckelnick, K., Katz, J., and Schieweck, F.}
\newblock A ${C}^{1}$–finite element method for the {Willmore} flow of
  two-dimensional graphs.
\newblock {\em Math. Comp. 84}, 296 (May 2015), 2617--2643.

\bibitem{dednerPythonBindings2020}
{\sc Dedner, A., Kloefkorn, R., and Nolte, M.}
\newblock Python bindings for the {DUNE-FEM} module.
\newblock {\em Zenodo\/} (Mar 2020).

\bibitem{dedner2010generic}
{\sc Dedner, A., Kl{\"o}fkorn, R., Nolte, M., and Ohlberger, M.}
\newblock A generic interface for parallel and adaptive discretization schemes:
  abstraction principles and the {DUNE-FEM} module.
\newblock {\em Computing 90}, 3-4 (2010), 165--196.

\bibitem{dedner_dune_2018}
{\sc Dedner, A., and Nolte, M.}
\newblock The {DUNE-PYTHON} module.
\newblock {\em arXiv preprint arXiv:1807.05252\/} (2018).

\bibitem{di2015hybrid}
{\sc Di~Pietro, D.~A., and Ern, A.}
\newblock A hybrid high-order locking-free method for linear elasticity on
  general meshes.
\newblock {\em Comput. Methods Appl. Mech. Engrg 283\/} (2015), 1--21.

\bibitem{lemaire2019bridging}
{\sc Lemaire, S.}
\newblock Bridging the hybrid high-order and virtual element methods.
\newblock {\em IMA Journal of Numerical Analysis\/} (2019).

\bibitem{liu2019virtual}
{\sc Liu, X., and Chen, Z.}
\newblock A virtual element method for the {Cahn-Hilliard} problem in mixed
  form.
\newblock {\em Appl. Math. Lett. 87\/} (2019), 115--124.

\bibitem{liu2020fully}
{\sc Liu, X., He, Z., and Chen, Z.}
\newblock A fully discrete virtual element scheme for the {Cahn-Hilliard}
  equation in mixed form.
\newblock {\em Computer Physics Communications 246\/} (2020), 106870.

\bibitem{morley1967triangular}
{\sc Morley, L. S.~D.}
\newblock {\em The triangular equilibrium element in the solution of plate
  bending problems}.
\newblock RAE, 1967.

\bibitem{nilssen_robust_2000}
{\sc Nilssen, T., Tai, X.-C., and Winther, R.}
\newblock A robust nonconforming ${H^2}$-element.
\newblock {\em Math. Comp. 70}, 234 (2001), 489--505.

\bibitem{wang_uniformly_2011}
{\sc Wang, L., Wu, Y., and Xie, X.}
\newblock Uniformly stable rectangular elements for fourth order elliptic
  singular perturbation problems.
\newblock {\em Numer. Methods Partial Differ. Equ. 29}, 3 (2013), 721--737.

\bibitem{wang2001necessity}
{\sc Wang, M.}
\newblock On the necessity and sufficiency of the patch test for convergence of
  nonconforming finite elements.
\newblock {\em SIAM J. Numer. Anal. 39}, 2 (2001), 363--384.

\bibitem{wang_robust_nodate}
{\sc Wang, M., and Meng, X.}
\newblock A robust finite element method for a {3-D} elliptic singular
  perturbation problem.
\newblock {\em J. Comput. Math.\/} (2007), 631--644.

\bibitem{wang2006modified}
{\sc Wang, M., Xu, J.-c., and Hu, Y.-c.}
\newblock Modified {Morley} element method for a fourth order elliptic singular
  perturbation problem.
\newblock {\em J. Comput. Math.\/} (2006), 113--120.

\bibitem{zhang_nonconforming_2020}
{\sc Zhang, B., Zhao, J., and Chen, S.}
\newblock The nonconforming virtual element method for fourth-order singular
  perturbation problem.
\newblock {\em Advances in Computational Mathematics 46}, 2 (Apr. 2020), 19.

\bibitem{zhao_nonconforming_2016}
{\sc Zhao, J., Chen, S., and Zhang, B.}
\newblock The nonconforming virtual element method for plate bending problems.
\newblock {\em Math. Models Methods Appl. Sci. 26}, 09 (Aug. 2016), 1671--1687.

\bibitem{zhao_morley-type_2018}
{\sc Zhao, J., Zhang, B., Chen, S., and Mao, S.}
\newblock The {Morley}-type virtual element for plate bending problems.
\newblock {\em J. Sci. Comput. 76}, 1 (2018), 610--629.

\end{thebibliography}
